\documentclass[11pt]{amsart}
\usepackage[foot]{amsaddr}
\usepackage{import}
\usepackage[english]{babel}
\usepackage[
pdftex,hyperfootnotes]
{hyperref}

\usepackage[total={18cm,24cm},top=2.5cm, left=2cm]{geometry}
\usepackage{rotating,multirow,pdflscape}
\usepackage{amssymb,latexsym,amsmath,amsthm}
\usepackage{mathrsfs}
\usepackage{mathtools,arydshln,mathdots}
\usepackage{bigints}
\makeatletter
\renewcommand*\env@matrix[1][*\c@MaxMatrixCols c]{%
  \hskip -\arraycolsep
  \let\@ifnextchar\new@ifnextchar
  \array{#1}}
\makeatother

\mathtoolsset{centercolon}
\usepackage[table]{xcolor}
\usepackage[toc,page]{appendix}

\usepackage{tocvsec2}

\usepackage[T1]{fontenc}
\usepackage[scaled]{helvet} 
\usepackage{courier} 
\usepackage{eulervm} 
\normalfont

\newtheorem{coro}{{Corollary}}

\newtheorem{defi}{{ Definition}}
\newtheorem{teo}{Theorem}
\newtheorem{pro}{ Proposition }

\renewcommand{\d}{\operatorname{d}}
\newcommand{\Exp}[1]{\operatorname{e}^{#1}}

\newcommand{\diag}{\operatorname{diag}}

\newcommand{\C}{\mathbb{X}}
\newcommand{\T}{\mathbb{T}}

\newcommand{\prodint}[1]{\left\langle{#1}\right\rangle}

\begin{document}
	
	\title[CMV biorthogonal  Laurent polynomials: Geronimus --Uvarov transformations]{CMV biorthogonal  Laurent polynomials. II:\\ Christoffel formulas for Geronimus--Uvarov transformations}
	
	 \author[G Ariznabarreta]{Gerardo Ariznabarreta$^{1,2}$}
	 \address{Departamento de Física Teórica II (Métodos Matemáticos de la Física), Universidad Complutense de Madrid, Ciudad Universitaria, Plaza de Ciencias 1,  28040 Madrid, Spain}
	 \email{gariznab@ucm.es}
	
	 \author[M Mañas]{Manuel Mañas$^{1}$}
	 \email{manuel.manas@ucm.es}
	
	 \author[A Toledano] {Alfredo Toledano}
	 \email{alfrtole@ucm.es}
	
	 \thanks{$^1$Thanks financial support from the Spanish ``Ministerio de Economía y Competitividad" research project [MTM2015-65888-C4-3-P],\emph{ Ortogonalidad, teoría de la aproximación y aplicaciones en física matemática}}
	 \thanks{$^2$Thanks financial support from the Universidad Complutense de Madrid  Program ``Ayudas para Becas y Contratos Complutenses Predoctorales en España 2011"}
	
\begin{abstract}
This paper is a continuation of  the recent paper \emph{CMV biorthogonal  Laurent polynomials: Christoffel formulas for Christoffel and Geronimus transformations} by the same authors.
The behavior of  quasidefinite sesquilinear forms for Laurent polynomials  in the complex plane, characterized by bivariate linear functionals,  and corresponding CMV biorthogonal Laurent polynomial families --including Sobolev and discrete Sobolev  orthogonalities---   under
two type of Geronimus--Uvarov transformations is studied.
Either the  linear functionals are multiplied by a Laurent polynomial  and divided by the  complex conjugate of a Laurent polynomial, with the  addition of appropriate masses (linear functionals supported on the zeros of the perturbing Laurent polynomial in the denominator) or vice-versa,  multiplied by the complex conjugate of a Laurent polynomial  and divided by a Laurent polynomial.
The connection formulas for the CMV biorthogonal Laurent polynomials, their norms, and Christoffel--Darboux kernels are given.
For prepared Laurent polynomials, i.e. of the form $L_{\mathfrak C}(z)=L_{\mathfrak C,N_{\mathfrak C}}z^{N_{\mathfrak C}}+\cdots+L_{\mathfrak C,-N_{\mathfrak C}}z^{-N_{\mathfrak C}}$, $L_{\mathfrak C,N_{\mathfrak C}}L_{\mathfrak C,-N_{\mathfrak C}}\neq 0$ and
$L_{\Gamma}(z)=L_{\Gamma,N_\Gamma}z^{N_\Gamma}+\cdots+L_{\Gamma,-N_\Gamma}z^{-N_\Gamma}$, $L_{\Gamma,N_\Gamma}L_{\Gamma,-N_\Gamma}\neq 0$, these connection formulas lead to quasideterminantal (quotient of determinants) Christoffel  formulas  expressing an arbitrary degree  perturbed biorthogonal Laurent polynomial in terms of
$2N_{\mathfrak C}+2N_\Gamma$ unperturbed biorthogonal Laurent polynomials, their second kind functions or Christoffel--Darboux kernel and its mixed versions.
When the linear functionals are supported on the  unit circle, a particularly relevant role is played by the reciprocal polynomial, and
the Christoffel formulas  provide now with two possible ways of expressing  the same perturbed quantities in terms of the original ones, one using only the nonperturbed biorthogonal family of  Laurent polynomials, and the other using  the Christoffel--Darboux kernel and its mixed versions.	
\end{abstract}



 \maketitle	


\section{Introduction}

In this paper we continue with the discussion of \cite{toledano} regarding transformations of CMV biorthogonal Laurent polynomials. In that precedent paper we studied Christoffel and Geronimus perturbations and the corresponding Christoffel formulas and in this one we complete the analysis with the Geronimus--Uvarov perturbations, which could be thought as a composition of a Christoffel and a Geronimus transformation.
For a deeper historical description and a more extended discussion of the state of the art regarding these issues we refer to the previous paper \cite{toledano}. Here we just reproduce some of the more essential facts for the discussion of the Geronimus--Uvarov transformations.

Christoffel formulas \cite{christoffel} for  perturbations, $\hat{u}= p(x) u$, where $u$ is a linear functional and  $p(x)$ is a polynomial, constitutes a classical result in the theory of orthogonal polynomials \cite{Chi,Sze,Gaut2}.
Connection formulas between two families of orthogonal polynomials allow to express any polynomial of a given degree $n$ as a linear combination of all polynomials of degree less than or equal to $n$ in the second family. Remarkably, for the  Christoffel  formulas, which are connection formulas, the number of terms does not grow with the degree $n$ but remain constant,  equal to the degree of the perturbing polynomial.
The problem of  given two functionals $\tilde u$ and $u$ such that  $\tilde u_x= \frac{p(x)}{q(x)}u_x$ where $p(x), q(x)$ are polynomials was  analyzed  by Uvarov in \cite{Uva} when $u,v$ are positive definite measures supported on the real line.  See also \cite{Zhe} for a discussion including spectral masses, in where the perturbation is written in the $q(x)\tilde u_x= p(x)u_x$, in that paper Zhedanov named these transformations as linear spectral transformations. We prefer to call them Geronimus--Uvarov.

Orthogonal polynomials in the unit circle $\mathbb T$ or Szeg\H{o} polynomials are  monic polynomials  $P_n$ of degree  $n$ such that $\int_{\T}P_n(z) z^{-k} \d \mu(z)=0$, for $ k=0,1,\dots,n-1$ \cite{Sze}.
 The extension to this  context of the three-term relations and tridiagonal Jacobi matrices needs of  Hessenberg matrices and give the Szeg\H{o} recursion relation, which  is expressed in terms of  the reciprocal Szeg\H{o} polynomials  and
the  Verblunsky  coefficients.
The papers  \cite{Jones-3,Jones-4} on   the strong Stieltjes moment problem can be considered as the starting point  for the consideration of orthogonal Laurent polynomials on the real line.
For recursion relations, Favard's theorem, quadrature problems, and Christoffel--Darboux formulas for Laurent polynomials on the unit circle $\T$ see  \cite{Thron,Barroso-Vera,CMV,Barroso-Daruis,Barroso-Snake}.
Orthogonal Laurent polynomials are dense in $L^2(\T,\mu)$, but Szeg\H{o} polynomials are not in general \cite{Bul} and \cite{Barroso-Vera}.  The CMV (Cantero--Moral--Velázquez) matrices \cite{CMV}  constitute a representation of the multiplication operator in terms of the basis of orthonormal Laurent polynomials and where discussed in \cite{Cantero} in  connection with Darboux transformations.
In \cite{Godoy1}  extensions of the Christoffel determinantal type formulas were given for the analogue of the Christoffel transformation, with an arbitrary degree polynomial having multiple roots, using the original Szeg\H{o} polynomials and its  Christoffel--Darboux kernels.
The Geronimus transformation for OPUC , with a perturbation of degree 2 and no masses,  was  discussed in  \cite{Godoy2}. In \cite{ismail}, alternative formulas \emph{á la Christoffel}, not based on the Christoffel--Darboux kernel  \cite{Godoy1},  were given in terms of determinantal expressions of the Szeg\H{o} polynomials and their reverse polynomials, also as Uvarov did in \cite{Uva}, they considered multiplication by rational functions, but no masses at all where discussed in this paper.
In \cite{R. W. Ruedemann}  some concrete cases where considered within  the biorthogonal scenario.
The transformations  considered in this work are also known as Darboux transformations \cite{matveev}. Indeed,
in the context of  the  Sturm--Liouville theory, Darboux discussed  in \cite{darboux2} a dimensional simplification of a geometrical transformation in two dimensions founded previously  \cite{moutard} which can be considered, as we called it today, a \emph{Darboux} transformation.

Our discussion framework is constructed upon the   \emph{noyaux-distribution} \cite{Schwartz1}. A  space of fundamental functions, in the sense of \cite{gelfand-distribu1,gelfand-distribu2}, and  the corresponding space of generalized functions provides with a linear functional setting for orthogonal polynomials. Discrete orthogonality appears when we consider linear functionals with discrete  and infinite support  \cite{Nikiforov1991Discrete}.
 We will consider an arbitrary nondegenerate continuous sesquilinear form given by a generalized kernel $u_{z_1,\bar z_2}$ with a quasidefinite Gram matrix. This scheme not only contains the more usual choices of Gram matrices like those  of Toeplitz  type on the unit circle, or those leading to  discrete orthogonality but also Sobolev orthogonality.

The  Gauss--Borel factorization problem,   has been applied  by our group in Madrid not only to the Christoffel and Geronimus transformations for sequilinear forms \cite{toledano} but also to the following cases
\begin{enumerate}
\item Laurent orthogonal polynomials on the unit circle  in \cite{CM}.
	\item Some extensions of the Christoffel--Darboux formula to generalized orthogonal polynomials \cite{adler-van moerbeke} and to multiple orthogonal polynomials,  \cite{am, ari3} .
	\item  Christoffel transformations and the relation with non-Abelian Toda hierarchies for real  matrix orthogonal polynomials were studied in \cite{alvarez2015Christoffel}, and in \cite{alvarez2016Transformation} we extended those results to include the Geronimus, Geronimus--Uvarov and Uvarov transformations.
	\item  Multiple orthogonal polynomials and multicomponent Toda \cite{amu}.
	\item  For matrix orthogonal Laurent polynomials on the unit circle, CMV orderings,  and non-Abelian  lattices on the circle \cite{ari}.
	\item Multivariate orthogonal polynomials in several real variables and corresponding multispectral integrable Toda hierarchy  \cite{MVOPR,ari0}.
Multivariate orthogonal polynomials on the multidimensional unit torus, the  multivariate extension of the CMV ordering and integrable Toda hierarchies \cite{ari1}.
\end{enumerate}

\subsection{Objectives, results,   layout of the paper and perspectives}

In the paper we continue and extend the studies of \cite{toledano}. As in that paper, we
consider a general  sesquilinear form in the complex plane  determined by a bivariate linear functional, its  biorthogonal Laurent families and their behavior under Geronimus--Uvarov perturbations, that can be thought as an appropriate consecutive composition of a Geronimus and a Christoffel perturbation. Both of these transformations were analyzed in \cite{toledano}.

The ideas of \cite{toledano} are used again in this paper, namely we consider a Gauss--Borel factorization of  the Gram matrix, which we assume to be quasidefinite, this will lead  to connection formulas for the biorthogonal Laurent polynomial families, the corresponding second kind functions and the standard and mixed Christoffel--Darboux kernels.  We find determinantal Christoffel formulas for the Geronimus--Uvarov transformation. Let us stress the relation of the results of the present paper and  the previous works \cite{Godoy1,Godoy2,ismail}
\begin{enumerate}
	\item In  \cite{Godoy1,Godoy2,ismail} the sesquilinear forms are supported on the diagonal with linear functionals of  zero order and positive (given, therefore, by a positive Borel measure). Moreover, the studies  in \cite{Godoy1,ismail} are restricted to measures supported on the unit circle. Our scheme allows for a more general biorthogonality and therefore includes Sobolev orthogonality and discrete Sobolev orthogonality with arbitrary support (with an infinity number points) on the complex plane.
	\item  The papers \cite{Godoy1,Godoy2} do not consider Geronimus--Uvarov transformations. In \cite{Godoy1} only the Christoffel transformations for orthogonal polynomials on the unit circle is analyzed, and in \cite{Godoy2} a particular Geronimus transformation of degree two, with no masses, is discussed. Regarding Geronimus--Uvarov transformations in the unit circle, \cite{ismail} does not incorporate masses at all. In our paper we include a very general class of masses.
\end{enumerate}
We have studied two possible Geronimus--Uvarov transformations, and give the corresponding  Christoffel formulas. These two transformations can be made to coincide when we have an initial zero order diagonal case supported on the unit circle, the Toeplitz situation.
Then, analogously as what we discovered in \cite{toledano},   two  alternative Christoffel formulas emerge.

The layout of the paper is as follows. We now proceed with a resume regarding basic facts  about CMV biorthogonal  Laurent polynomials.  In \S 2 we perturb a general quasidefinite sesquilinear form by multiplying the corresponding bivariate linear functional by a quotient of  a Laurent polynomial and the complex conjugate of another Laurent polynomial, one depending on the first variable of the bivariate linear functional and the other in the second variable, and another by the complex conjugate of the previous quotient.   We include the addition of masses
supported on the zeros of the  Laurent polynomials in the denominator of the perturbation. Quasidefiniteness of the perturbed sesquilinear forms allows for  the Gauss--Borel factorization, which leads to   connection  formulas for the biorthogonal Laurent polynomials,  second kind functions, Christoffel--Darboux kernels,  and mixed Christoffel--Darboux kernels, see Propositions \ref{Geronimus Connection Laurent}, \ref{Geronimus Connection Cauchy}, \ref{Geronimus Connection CD}, and \ref{Geronimus Connection mixed}. With this at hand we present Theorem \ref{Christoffel-Geronimus formulas}, where   the Christoffel--Geronimus--Uvarov formulas for both type of perturbations are given. We write, for the first time, this type of expressions including the masses.    In \S 3 we discuss possible reductions, first to zero order supported on the diagonal, then to the unit circle, and finally linear functionals taking real values.
  There is an Appendix containing some of the  proofs.

\subsection{Basic facts regarding CMV biorthogonal Laurent polynomials}

 \begin{defi} [Sequilinear forms]\label{def:sesquilinear}
	A sesquilinear form     $\prodint{\cdot,\cdot}$  in complex linear space $V$  is a continous map
$\prodint{\cdot,\cdot}: V\times V\longrightarrow \mathbb{C}$
such that for any triple $f,g,h\in  V$ the following conditions are satisfied
\begin{enumerate}
	\item  $\prodint{Af+Bg,h}=A\prodint{f,h}+B\prodint{g),h}$, $\forall A,B\in\mathbb{C}$,
	\item $\prodint{f,Ag+Bh}=\prodint{f,g}\bar A+\prodint{f,h}\bar B$, $\forall A,B\in\mathbb C$.
\end{enumerate}
\end{defi}

Given the ordered bi-infinite  basis $\{ z^{l}\}_{l=-\infty}^\infty$ or the semiinfinite CMV basis   $\{\chi^{(l)}(z)\}_{l=0}^\infty$,  where we have $\chi^{(l)}:=\begin{cases}
z^{l/2}, & \text{$l$ even}\\
z^{-(l+1)/2}, & \text{$l$ odd.}
\end{cases}$ of  $\mathbb C[z,z^{-1}]$ the sesquilinear form is characterized by the corresponding Gram matrix. For example, in the first case we have the biinfinite Gram matrix
$
g=\begin{bsmallmatrix}
g_{0,0 } &g_{0,1}& \dots\\
g_{1,0} & g_{1,1} & \dots\\
\vdots & \vdots
\end{bsmallmatrix}$ with $g_{k,l}=\prodint{(z_1)^k ,(z_2)^l }$, $k,l\in\mathbb Z$.
\begin{defi}[Laurent polynomial spectrum]
The zero set of a function $L(z)$ in  $\mathbb C^*:=\mathbb C\setminus\{0\}$ will be denoted by  $\sigma(L)$ and said to be the spectrum of $L(z)$.
\end{defi}
In this paper we will consider sesquilinear forms constructed in terms of  bivariate linear functionals with well-defined  support.
We will work with generalized functions $\mathcal F'$, i.e. continuous linear functionals over the space of fundamental functions $\mathcal F$, \cite{gelfand-distribu1,gelfand-distribu2},  such that $\mathbb C[z,z^{-1}]\subsetneq \mathcal F$ and there is a well defined notion of support (this last condition forbids the choice  $C[z,z^{-1}]= \mathcal F$).
The space of distributions is the space of generalized functions  when the  fundamental functions  space is the set of complex smooth functions \cite{Schwartz} with compact support  $ \mathcal D^*:=C_0^\infty(\mathbb C^*)$.
The zero set of a distribution $u\in(\mathcal D^*)'$ is the open region  $\Omega\subset \mathbb C^*$ whenever for every  $f(z)$ supported on $\Omega$ we have $\langle u, f\rangle =0$. Its complementary set, which is closed, is the support, $\operatorname{supp} (u)$, of the distribution $u$.   Obviously $\mathbb C[z,z^{-1}]\not\subset\mathcal D$ and, consequently, the space of test functions $\mathcal D$ is not suitable for our aims. However, the next example  gives an adequate scenario for our constructions.
When we take as our space of fundamental functions $\mathcal F$ the space of smooth functions in $\mathbb C^*$,  $\mathcal F=\mathcal E^*=C^\infty(\mathbb C^*)$ the corresponding space of generalized functions is the space $(\mathcal E^*)'$ of  distributions of compact support in $\mathbb C^*$. Now, as  $\mathbb C[z,z^{-1}]\subsetneq \mathcal E^*$ we deduce that $(\mathcal E^*)'\subsetneq (\mathbb C[z,z^{-1}])'\cap \mathcal (\mathcal D^*)'$.   The set of  distributions of compact support  is a first example of an appropriate framework for the consideration of polynomials and supports simultaneously.
For any bivariate linear functional   $u_{z_1,\bar z_2}$ with support $\operatorname{supp} (u_{z_1,\bar z_2})$ its projections in the axis $z_i$ are denoted by
$\operatorname{supp}_{i} (u_{z_1,\bar z_2})$, $i=1,2$. Sesquilinear forms are constructed in terms of bivariate linear functionals. 
\begin{defi}\label{def:sesquilineal}
We consider the following sesquilinear forms
 \begin{align*}
 \prodint{f(z_1), g(z_2)}_{u}&=\prodint{u_{ z_1,\bar z_2}, f( z_1)\otimes \overline{ g(z_2)}}, & f(z),g(z)&\in\mathcal F.
 \end{align*}
\end{defi}
 Hence, the following sesquilinear forms
$ \prodint{f(z_1), g(z_2)}_{u}=\sum\limits_{0\leq n,m\ll\infty}\int  \frac{\partial^nf}{\partial z^n} ( z_1)\overline{\frac{\partial^{m} g}{\partial z^m} (z_2)}\d\mu^{(m,n)}(z_1,z_2)$,
 for Borel  measures $\mu^{(m,n)}(z_1,z_2)$ in $\mathbb C^{2}$, with at least one of them with infinite support, are included in our considerations.
In the bi-infinite basis
$\{z^n\}_{n\in\mathbb Z}$ we have the Gram matrix
$g=[g_{n,m}]$ with $ g_{n,m}= \prodint {(z_1)^n, ( z_2)^{m}}_{u}=\prodint{u_{z_1,\bar z_2},z_1\otimes (\bar z_2)^m}$.

Following \cite{CMV,watkins}, we will use the  CMV   basis $\big\{\chi^{(0)},\chi^{(1)},\chi^{(2)},\dots\big\}$ with $\chi^{(l)}(z)=\begin{cases}
z^k, &l=2k,\\
z^{-k-1}, &l=2k+1.
\end{cases}$
\begin{defi}
Let us consider
\begin{align*}
\chi_{1}(z)&:= [1,0,z,0,z^{2},0,\ldots]^\top,  &
\chi_{2}(z)&:= [0,1,0,z,0,z^{2},0,\dots]^{\top}
\end{align*}
and
\begin{align*}
\chi_{1}^{*}(z)&:=z^{-1}\chi_{1}(z^{-1})=[z^{-1},0,z^{-2},0,z^{-3},0,\ldots]^\top,&
\chi_{2}^{*}(z)&:=z^{-1}\chi_{2}(z^{-1})= [0,z^{-1},0,z^{-2},0,z^{-3},0,\dots]^{\top}.
\end{align*}
In terms of which we define the CMV sequences
\begin{align*}
\chi(z)&:=\chi_{1}(z)+\chi_{2}^{*}(z)=[1,z^{-1},z,z^{-2},\ldots]^{\top},&
\chi^{*}(z)&:=\chi_{1}^{*}(z)+\chi_{2}(z)=[z^{-1},1,z^{-2},z,z^{-3},z^{2},\dots]^{\top}.
\end{align*}
The matrices
\begin{align*}\Upsilon&=\left[
\begin{array}{c|cc|cc|cc|cc|cc}
0 & 0 & 1 & 0 & 0 & 0 & 0 & 0 & 0 & 0 &\cdots  \\\hline
1 & 0 & 0 & 0 & 0 & 0 & 0 & 0 & 0 & 0 &\cdots  \\
0 & 0 & 0 & 0 & 1 & 0 & 0 & 0 & 0 & 0 &\cdots  \\\hline
0 & 1 & 0 & 0 & 0 & 0 & 0 & 0 & 0 & 0 &\cdots  \\
0 & 0 & 0 & 0 & 0 & 0 & 1 & 0 & 0 & 0 &\cdots  \\\hline
0 & 0 & 0 & 1 & 0 & 0 & 0 & 0 & 0 & 0 &\cdots  \\
0 & 0 & 0 & 0 & 0 & 0 & 0 & 0 & 1 & 0 &\cdots  \\\hline
0 & 0 & 0 & 0 & 0 & 1 & 0 & 0 & 0 & 0 &\cdots  \\
0 & 0 & 0 & 0 & 0 & 0 & 0 & 0 & 0 & 0 &\cdots  \\\hline
0 & 0 & 0 & 0 & 0 & 0 & 0 & 1 & 0 & 0 &\cdots  \\
\vdots & \vdots & \vdots & \vdots & \vdots & \vdots &
\vdots & \vdots & \vdots & \vdots &\ddots
\end{array}\right]
\end{align*}
Given a bivariate linear functional $u_{z_1,\bar z_2}$ and the associated sesquilinear form  we consider the corresponding Gram matrix
\begin{align*}
G=\prodint{\chi(z_1),(\chi(z_2))^\top}_u=\prodint{u_{z_1,\bar z_2},\chi(z_1)\otimes \big(\chi(z_2)\big)^\dagger}.
\end{align*}
\end{defi}

\begin{pro}
The semi-infinite matrix $\Upsilon$, is unitary
$	\Upsilon^{\top}=\Upsilon^{-1}$,
	and has the important spectral properties
$
	\Upsilon \chi(z)=z\chi(z)$ and $\Upsilon^{-1}\chi(z)=z^{-1}\chi(z)$.
\end{pro}
For the truncation of  semi-infinity matrices we will use the following notation
$
A=\begin{bsmallmatrix}
A_{0,0} & A_{0,1} &\dots\\
A_{1,0} & A_{1,1} &\dots\\
\vdots & \vdots & \\
\end{bsmallmatrix}$ and $
A^{\left[l\right]}:=
\begin{bsmallmatrix}
A_{0,0} & A_{0,1} & \cdots & A_{0,l-1} \\
A_{1,0} & A_{1,1} & \cdots & A_{1,l-1} \\
\vdots &  \vdots&  &  \vdots\\
A_{l-1,0} & A_{l-1,1} & \cdots & A_{l-1,l-1} \\
\end{bsmallmatrix}$.
For the corresponding block structure we write
$
A=
\left[
\begin{array}{c|c}
	A^{\left[l\right]} &
	A^{\left[l,\geq l\right]}\\
	\hline
	A^{\left[\geq l,l\right]} & A^{\left[\geq l\right]}
	\end{array}
	\right]$. We will assume that the Gram matrix $G$  is  quasidefinite, i.e.,  all its principal minors are not zero, so that  the following  Gauss--Borel o $LU$ factorization of  $G$ holds
\begin{align}\label{LU}
G=S_{1}^{-1}H(S_{2}^{-1})^{\dag},
\end{align}
where $S_{1}$ and $S_{2}$ are lower unitriangular matrices and $H$ is a diagonal nonsingular matrix.
\begin{defi}Let us introduce the following vectors of Laurent poynomials
\begin{align}\label{P}
\phi_{1}(z)&:=S_{1}\chi(z), &
\phi_{2}(z)&:=S_{2}\chi(z).
\end{align}
\end{defi}
Its components $\phi_1(z)=[\phi_{1,0}(z), \phi_{1,1}(z),\dots]^\top$ and $\phi_2(z)=[\phi_{2,0}(z), \phi_{2,1}(z),\dots]^\top$ are such that
\begin{pro}[Biorthogonal polynomials]
The following biothogonality conditions
\begin{align*}
\prodint{\phi_{1,n}(z_1),\phi_{2,m}(z_2)}_u=\delta_{n,m}H_n,
\end{align*}
hold for  $n,m\in\{0,1,2,\dots\}$.
\end{pro}

\begin{coro}\label{ortogonalidad}
The orthogonality relations
 \begin{align*}
 \prodint{ \phi_{1,2k}(z_1),(z_2)^l}_u&=0, &  -k\leq & l \leq k-1, \\
 \prodint{ \phi_{1,2k+1}(z_1),(z_2)^l} _u&=0, &  -k\leq & l \leq k, \\
\prodint{  (z_1)^l,\phi_{2,2k}(z_2) }_u&=0, &  -k\leq & l \leq k-1, \\
\prodint{  (z_1)^l,\phi_{2,2k+1}(z_2) }_u&=0, &  -k\leq & l \leq k,
 \end{align*}
 are satisfied.
\end{coro}

\begin{defi}
	The Christoffel--Darboux  kernel is
	\begin{align}\label{kernelChristoffeldefinicion}
	K^{[l]}(\bar z_1,z_2):=
	\sum_{k=0}^{l-1}\overline{\phi{_{2,k}}(z_1)}H_{k}^{-1}\phi_{1,k}(z_2)
	=[\phi_{2}(z_1)^{\dagger}]^{[l]}(H^{-1})^{[l]}[\phi_{1}(z_2)]^{[l]}.%
	\end{align}
\end{defi}

\begin{pro}
The  Christoffel--Darboux kernel satisfies the projection   properties
\begin{align*}
\prodint{
	\sum_{j=0}^M f_j \phi_{1,j}(z_1),
	\overline{K^{[l]}(\bar z_2,z)}
}_u&=	\sum_{j=0}^{l-1} f_j \phi_{1,j}(z),&
\prodint{K^{[l]}(\bar z,z_1),
	\sum_{j=0}^M f_j \phi_{2,j}(z_2)}_u&=\overline{\sum_{j=0}^{l-1} f_j \phi_{2,j}(z)}.
\end{align*}
\end{pro}
This implies that, when acting on the right  $\overline{K^{[l+1]}(\bar z_2,z)}$ projects  over $\Lambda_{l}:=\mathbb C\{\chi^{(k)}(z)\}_{k=0}^{l}$ while when acting on the left  $K^{[l+1]}(\bar z,z_1)$ projects on $\overline{\Lambda_l}$. Notice that
\begin{align*}
\Lambda_{2k}&=\{1,z^{-1},z,\dots,z^{-k},z^k\}, & \Lambda_{2k+1}&=\{1,z^{-1},z,\dots,z^k,z^{-k-1}\}.
\end{align*}

\begin{coro}\label{CDproyeccion}
	If $L(z)\in\Lambda_l:=\mathbb C\{\chi^{(k)}(z)\}_{k=0}^{l-1}$  then
	\begin{align*}
	\prodint{
		L(z_1),
		\overline{K^{[l]}(\bar z_2,z)}
	}_u&=	L(z),&
	\prodint{K^{[l]}(\bar z,z_1),
		L(z_2)}_u&=\overline{L(z)}.
	\end{align*}
\end{coro}

\begin{defi}
The  second kind functions are given by
	\begin{align*}
	C_{1}(z)&=\prodint{\phi_1(z_1), \frac{1}{\bar z-z_2}}_u
	=\prodint{u_{z_1,\bar z_2},\phi_1(z_1)\otimes \frac{1}{z-\bar z_2}}
	, &z&\not\in\overline{\operatorname{supp}_2(u)},\\
	(	C_{2}(z))^\dagger&=\prodint{\frac{1}{\bar z-z_1}, (\phi_2(z_2))^\top}_u=\prodint{u_{z_1,\bar z_2},\frac{1}{\bar z-z_1}\otimes(\phi_{2}(z_2))^\dagger },& z&\not\in\overline{ \operatorname{supp}_1(u)}.
 	\end{align*}
\end{defi}
\begin{defi}
The mixed Christoffel--Darboux  kernels are
	\begin{align}\label{kernelChristoffeldefinicionphiC}
	K_{C_2}^{[l]}(\bar z_1,z_2)&:=\sum_{k=0}^{l-1}\overline{C{_{2,k}}(z_1)}H_{k}^{-1}\phi_{1,k}(z_2)=[C_{2}(z_1)^{\dagger}]^{[l]}(H^{-1})^{[l]}[\phi_{1}(z_2)]^{[l]},& z_1&\not\in \overline{\operatorname{supp}_1(u)},\\
		K_{C_1}^{[l]}(\bar z_1,z_2)&:=\sum_{k=0}^{l-1}\overline{\phi{_{2,k}}(z_1)}H_{k}^{-1}C_{1,k}(z_2)=[\phi_{2}(z_1)^{\dagger}]^{[l]}(H^{-1})^{[l]}[C_{1}(z_2)]^{[l]},&z_2&\not\in \overline{\operatorname{supp}_2(u).}\label{kernelChristoffeldefinicionCphi}
	\end{align}
\end{defi}
\begin{pro}
The mixed kernels have the following expressions
	\begin{align*}
	K_{C_2}^{[l]}(\bar x_1,x_2)&=
	\prodint{\frac{1}{\bar x_1-z_1}, \overline{K^{[l]}(\bar z_2,x_2)}}_u,&
	K_{C_1}^{[l]}(\bar x_1,x_2)&:=\prodint{K^{[l]}(\bar x_1,z_1), \frac{1}{\bar x_2-z_2}}_u.
	\end{align*}
\end{pro}
Hence, the mixed kernels can be thought as the projections of the Cauchy kernels or, equivalently, the Cauchy transforms of the Christoffel--Darboux kernels.

\begin{defi}[Prepared  Laurent polynomials] For every $2 n$-degree polynomial $P(z)=P_{2 n}z^{2 n}+\dots+P_0\in\mathbb C[z]$ with $P_0\neq 0$,  its Féjer--Riesz corresponding  Laurent polynomial is given by
	\begin{align}\label{polinomio perturbador}
	L(z)&=z^{-n}P(z)=L_n z^{n}+\cdots + L_{-n} z^{-n}, & L_n&=P_{2 n},&L_{-n}&=P_0.
	\end{align}
	We say that a  Laurent polynomial is prepared whenever it is the  Féjer--Riesz corresponding  Laurent polynomial  of  an even degree polynomial non vanishing at the origin.
\end{defi}
For the consideration of arbitrary multiplicities of the zeros of the perturbing polynomials we need of
\begin{defi}[Spectral jets]
	Given a Laurent polynomial $L(z)$ with zeros and multiplicities  $ \{\zeta _ {i}, m_ {i} \} _ {i = 1} ^ {d} $ we introduce the spectral jet of a  function $ f (z)$ along $L(z)$ as follows:
	\begin{align*}
	\mathcal J^L_f
	:=\left[f(\zeta_{1}),f'(\zeta_{1}),\cdots,\frac{f^{(m_{1}-1)}(\zeta_{1})}{(m_{1}-1)!},\cdots,f(\zeta_{d}),f'(\zeta_{d})\cdots,\frac{f^{(m_{d}-1)}(\zeta_{d})}{(m_{d}-1)!}\right]\in\mathbb C^{2m}.
	\end{align*}
\end{defi}

\section{Geronimus--Uvarov transformations}

We now proceed with the consideration of the Geronimus--Uvarov transformation in the context of complex sesquilinear forms. We will present the extension of the formulas found by Uvarov \cite{Uva} for this situation, but now including masses. We continue with the discussion of \cite{toledano} and compose a first Geronimus transformation with a consecutive Christoffel transformation, in that precise order, and  as we have two types
of such perturbations, either if we multiply or divide by a Laurent polynomial or the complex conjugate of a Laurent polynomial, we have two type of Geronimus--Uvarov perturbations, namely
$\tilde u_{z_1,\bar z_2}^{(1,2)}=\widehat{(\check u^{(1)}_{z_1,\bar z_2})}^{(2)}$ and $\tilde u_{z_1,\bar z_2}^{(2,1)}=\widehat{(\check u^{(2)}_{z_1,\bar z_2})}^{(1)}$.
Let us now proceed with the precise definition.
\begin{defi}
	For a bivariate linear  functional  $u_{z_1,\bar z_2}$ with  a well-defined support,
for $a\in\{{1,2}\}$,	given Laurent polynomials  $L_\Gamma^{(a)}(z)=L^{(a)}_{\Gamma,N^+_{\Gamma}} z^{N^+_{\Gamma}}+\dots+L^{(a)}_{\Gamma,-N^-_{\Gamma}} z^{-N^-_{\Gamma}}$, and $L_{\mathfrak C}^{(a)}(z)=L^{(a)}_{\mathfrak C,N^+_{\mathfrak C}} z^{N^+_{\mathfrak C}}+\dots+L^{(a)}_{\mathfrak C,-N^-_{\mathfrak C}} z^{-N^-_{\mathfrak C}}$ such that  $L^{(a)}_{\Gamma,N^+_\Gamma}L^{(a)}_{\Gamma,-N^-_{\mathfrak C}}\neq 0$, $L^{(a)}_{\mathfrak C,N^+_{\mathfrak C}}L^{(a)}_{\mathfrak C,-N^-_{\mathfrak C}}\neq 0$,  $N^\pm_\Gamma,N^\pm_{\mathfrak C}\in\{1,2,\dots\}$,  and $\sigma(L_\Gamma^{(1)}(z))\cap\operatorname{supp}_1u=\varnothing$,
	$\overline{\sigma(L_\Gamma^{(2)}(z))}\cap\operatorname{supp}_2u=\varnothing$, we consider two possible families of  Geronimus--Uvarov transformations $\tilde u^{(1,2)}_{z_1,\bar z_2}$ and   $\tilde u^{(2,1)}_{z_1,\bar z_2}$ characterized by
	\begin{align*}
	L_\Gamma^{(1)}(z_1)\tilde u^{(1,2)}_{z_1,\bar z_2}&=u_{z_1,\bar z_2}\overline{L_{\mathfrak C}^{(2)}(z_2)},  &
\tilde u^{(2,1)}_{z_1,\bar z_2}\overline{L_\Gamma^{(2)}(z_2)}&=L_{\mathfrak C}^{(1)}(z_1)u_{z_1,\bar z_2} .
	\end{align*}
\end{defi}
The notation $L_\Gamma^{(a)}(z)$ makes reference to a perturbation of Geronimus type while $L_{\mathfrak C}^{(a)}(z)$ to a perturbation of Christoffel type.
Therefore, the perturbed bivariate linear  functionals   are
\begin{align*}
\tilde u^{(1,2)}_{z_1,\bar z_2}&=\frac{\overline{L_{\mathfrak C}^{(2)}(z_2)}
}{L_\Gamma^{(1)}(z_1)}	u_{z_1,\bar z_2}+\overline{L_{\mathfrak C}^{(2)}(z_2)}\sum_{i=1}^{d^{(1)}}\sum_{l=0}^{m^{(1)}_{i}-1}\frac{(-1)^{l}}{l!}\delta^{(l)}(z_1-\zeta^{(1)}_{i})
\otimes \overline{(\xi^{(1)}_{i,l})_{ z_2}},  \\
\tilde u^{(2,1)}_{z_1,\bar z_2}&=\frac{L_{\mathfrak C}^{(1)}(z_1)}{\overline{L_\Gamma^{(2)}(z_2)}}u_{z_1,\bar z_2}+L_{\mathfrak C}^{(1)}(z_1)
\sum_{i=1}^{d^{(2)}}\sum_{l=0}^{m^{(2)}_{i}-1}\frac{(-1)^{l}}{l!}(\xi^{(2)}_{i,l})_{z_1}\otimes\overline{\delta^{(l)}(z_2- \zeta^{(2)}_{i})},
\end{align*}
where $\zeta^{(a)}_{i}$ are zeros with multiplicities  $m^{(a)}_{i}$ of $L_\Gamma^{(a)}(z_1)$,  while $(\xi^{(a)}_{i,l})_{z_a}$ a are univariate  linear functionals.
In terms of  sesquilinear forms we have
\begin{align*}
\prodint{L^{(1)}_\Gamma (z_1)f(z_1),g(z_2)}_{\tilde u^{(1,2)}}&=\prodint{f(z_1), L_{\mathfrak C}^{(2)}(z_2)g(z_2)}_u, &
\prodint{f(z_1), L_\Gamma^{(2)}(z_2)g(z_2)}_{\tilde u^{(2,1)}}=\prodint{L_{\mathfrak C}^{(1)}(z_1)f(z_1),g(z_2))}_u, 
\end{align*}
for all $ f,g\in\mathcal F$.
\begin{pro}
Geronimus--Uvarov transformations associated with  the two couples of perturbing Laurent polynomials $ L_\Gamma^{(1)} (z)$,$L_{\mathfrak C}^{(2)}(z) $ and $L_\Gamma^{(2)} (z)$, $L_{\mathfrak C}^{(1)}(z)$ imply for the corresponding Gram matrices
	\begin{align}\label{GGeronimus}
L_\Gamma^{(1)}(\Upsilon)	\tilde  G^{(1)}&=G(  L_{\mathfrak C}^{(2)}(\Upsilon))^{\dagger}, & \tilde G^{(2)}\big(L_\Gamma^{(2)}(\Upsilon)\big)^\dagger&=L_{\mathfrak C}^{(1)}(\Upsilon)G.
	\end{align}
\end{pro}
\begin{proof}
It follows from
	\begin{align*}
	L_\Gamma^{(1)}(\Upsilon)\tilde{G}^{(1)}
	&=\prodint{ \tilde u^{(1,2)}_{z_1,\bar z_2},L_\Gamma^{(1)}(z_1)\chi(z_1)\otimes (\chi(z_2))^\dagger}
	\\
	&=\prodint{  u_{z_1,\bar z_2},\chi(z_1)\otimes (\chi(z_2))^\dagger\overline{L_{\mathfrak C}^{(2)}(z_2)}}
	\\
	&=\prodint{  u_{z_1,\bar z_2},\chi(z_1)\otimes (\chi(z_2))^\dagger}(  L_{\mathfrak C}^{(2)}(\Upsilon))^{\dagger},
	\\
\tilde{G}^{(2)}\big(L^{(2)}_\Gamma(\Upsilon)\big)^\dagger
	&=\prodint{ \tilde u^{(2,1)}_{z_1,\bar z_2},\chi(z_1)\otimes (\chi(z_2))^\dagger \overline{L_\Gamma^{(2)}(z_2)}}
	\\
	&=\prodint{ u_{z_1,\bar z_2},L_{\mathfrak C}^{(1)}(z_1)\chi(z_1)\otimes (\chi(z_2))^\dagger }
	\\
	&= L_{\mathfrak C}^{(1)}(\Upsilon)\prodint{ u_{z_1,\bar z_2},\chi(z_1)\otimes (\chi(z_2))^\dagger}.
	\end{align*}
\end{proof}

We assume that both Gram matrices  are quasidefinite, i.e. the  Gauss--Borel factorizations
\begin{align}\label{eq:quasidefGramGer}
\tilde G^{(1)}&=\big(\tilde S_1^{(1)}\big)^{-1}\tilde  H^{(1,2)}\big(\tilde S_2^{(1)}\big)^{-\dagger}, &
\check  G^{(2)}&=\big(\tilde S_1^{(2)}\big)^{-1}\tilde  H^{(2,1)}\big(\tilde S_2^{(2)}\big)^{-\dagger},
\end{align}
can be performed.

\subsection{Connection formulas for the CMV biorthogonal Laurent polynomials and its second kind functions}

\begin{defi}[Geronimus--Uvarov connectors]For a bivariate linear functional and Geronimus--Uvarov transformations  and two couples of perturbing Laurent polynomials $ \{L_\Gamma^{(1)} (z),L_{\mathfrak C}^{(1)}(z)\} $ and $ \{L_\Gamma^{(2)} (z),L_{\mathfrak C}^{(2)}(z)\} $, we associate  the following \emph{connectors}
	\begin{align*}
	\Omega^{(1,2)}_1&= S_1 L_\Gamma^{(1)}(\Upsilon)
	(\tilde S_1^{(1)})^{-1}, & 	\Omega^{(1,2)}_2&=\tilde S_2^{(1)}L_{\mathfrak C}^{(2)}(\Upsilon)( S_2)^{-1},\\
	\Omega^{(2,1)}_1&=\tilde S_1^{(2)}L_{\mathfrak C}^{(1)}(\Upsilon)( S_1)^{-1}, &	\Omega^{(2,1)}_2&= S_2 L_\Gamma^{(2)}(\Upsilon)
	(\tilde S_2^{(2)})^{-1}.
	\end{align*}
\end{defi}
\begin{pro}
Geronimus--Uvarov connectors satisfy
\begin{align*}
	\Omega^{(1,2)}_1 \tilde H^{(1,2)}&=H \big(\Omega^{(1,2)}_2\big)^\dagger, &
	\tilde H^{(2,1)} \big(\Omega^{(2,1)}_2\big)^\dagger &= \Omega^{(2,1)}_1 H.
	\end{align*}
\end{pro}
\begin{proof}
From \eqref{GGeronimus}  and  \eqref{eq:quasidefGramGer} we deduce that
	\begin{align*}
	L_\Gamma^{(1)}(\Upsilon)\big(\tilde  S_1^{(1)}\big)^{-1}\tilde H^{(1,2)}\big(\tilde S_2^{(1)}\big)^{-\dagger}&=(S_1)^{-1}H(S_2)^{-\dagger}(  L_{\mathfrak C}^{(2)}(\Upsilon))^{\dagger},\\
	\big(\tilde S_1^{(2)}\big)^{-1}\tilde H^{(2,1)}\big(\tilde S_2^{(2)}\big)^{-\dagger}\big(L_\Gamma^{(2)}(\Upsilon)\big)^{\dagger}&=L_{\mathfrak C}^{(1)}(\Upsilon)(S_1)^{-1}H(S_2)^{-\dagger}.
	\end{align*}
Hence
	\begin{align*}
S_1 L_\Gamma^{(1)}(\Upsilon)\big(\tilde S_1^{(1)}\big)^{-1}\tilde H^{(1,2)}&=H(S_2)^{-\dagger}(L_{\mathfrak C}^{(2)}(\Upsilon))^{\dagger}\big(\tilde S_2^{(1)}\big)^{\dagger},&
	\tilde H^{(2,1)}\big(\tilde S_2^{(2)}\big)^{-\dagger}(L_\Gamma^{(2)}(\Upsilon))^{\dagger}(S_2)^{\dagger}&=\tilde S_1^{(2)}L_{\mathfrak C}^{(1)}(\Upsilon)(S_1)^{-1}H.
	\end{align*}
\end{proof}
\begin{defi}
For our convenience we introduce for $\rho=\mathfrak C,\Gamma$ 
\begin{align*}
L^{(a)}_{\rho,-\underline{N_\rho}}&:=0, & \text{for $-\underline{N_\rho}< -N^-_\rho$,}\\
L^{(a)}_{\rho,\underline{N_\rho}}&:=0, & \text{for $\underline{N_\rho}> N^+_\rho$.}
\end{align*}
 
\end{defi}

\begin{pro}
	Geronimus--Uvarov connectors are banded semi-infinite matrices. In particular, if $\underline N_\Gamma:=\max(N^+_\Gamma,N^-_\Gamma)$ and $\underline N_{\mathfrak C}:=\max(N_{\mathfrak C}^+,N_{\mathfrak C}^-)$
	\begin{enumerate}
\item The connectors $ \Omega ^ {(1,2)} _ 1 $  and $ \Omega ^ {(2,1)} _ 2 $ have as possible nonzero diagonals the first $ 2\underline N_\Gamma$ superdiagonals and  $ 2\underline N_{\mathfrak C}$ subdiagonals.
\item The connectors $ \Omega ^ {(1,2)} _ 2 $  and $ \Omega ^ {(2,1)} _ 1 $have  as possible nonzero diagonals the first  $ 2\underline N_{\mathfrak C} $ superdiagonals and  $ 2\underline N_\Gamma$ subdiagonals .
\item Moreover, we have the formulas
		\begin{align}\label{Omegalambda}
			( \Omega^{(1,2)}_2)_{l,l-2\underline N_\Gamma}&= \overline{L^{(1)}_{\Gamma,(-1)^l\underline N_\Gamma}	\dfrac{ {\tilde H}^{(1,2)}_l}{  H_{k-2\underline N_\Gamma}}}, &
		(\Omega^{(2,1)}_1)_{l,l-2\underline N_\Gamma}&=\overline{ L^{(2)}_{\Gamma,(-1)^l\underline N_\Gamma}}	\dfrac{\tilde H^{(2,1)}_l}{H_{l-2\underline N_\Gamma}},&l\geq2&\underline N_\Gamma,\\
			( \Omega^{(1,2)}_2)_{l,l+2\underline N_{\mathfrak C}}&= {L^{(2)}_{{\mathfrak C},(-1)^l\underline N_{\mathfrak C}}}, &
		(\Omega^{(2,1)}_1)_{l,l+2\underline N_{\mathfrak C}}&={ L^{(1)}_{{\mathfrak C},(-1)^l\underline N_{\mathfrak C}}}.\label{Omegalambda2}
		\end{align}
	\end{enumerate}\end{pro}
To name these semi-infinite matrices as connectors is justified by

	\begin{pro}\label{Geronimus Connection Laurent}
The following connection formulas for the CMV biorthogonal Laurent polynomials hold
		\begin{align*}
		\Omega^{(1,2)}_1  \tilde\phi_1^{(1,2)}(z) &= L_\Gamma^{(1)}(z)\phi_1(z), & \Omega^{(1,2)}_2 \phi_2(z)&=L_{\mathfrak C}^{(2)}(z)\tilde\phi^{(1,2)}_2(z),\\
		\Omega_1^{(2,1)} \phi_1(z)&=L_{\mathfrak C}^{(2)}(z)\tilde\phi^{(2,1)}_1(z), & \Omega^{(2,1)}_2\tilde\phi^{(2,1)}_2(z)&=L_\Gamma^{(2)}(z) \phi_2(z).
		\end{align*}
	\end{pro}

\begin{defi}
Given a Laurent polynomial $L(z)$ we consider
$	\delta L (z_1,z_2):=\frac{L(z_1)-L(z_2)}{z_1-z_2}$
and  the completely homogeneous symmetric polynomials
$	h_j(z_1,z_2):=(z_1)^j+(z_1)^{j-1}z_2+\dots +z_1(z_2)^{j-1}+(z_2)^{j}$ and their duals  $h^*_{j}(z_1,z_2):=(z_1z_2)^{-1}h_{j}\big((z_1)^{-1},(z_2)^{-1}\big)$ with $j\in\{0,1,2,\dots\}$.
\end{defi}

\begin{pro}\label{simetricos}
One has
$	\delta L(z_1,z_2) =\sum\limits_{j=1}^{n} L_{j}  h_{j-1}(z_1, z_2)-\sum\limits_{j=1}^{m}L_{-j} h^*_{j-1}(z_1, z_2)$
	and therefore the bivariate Laurent polynomial $ \delta L (z_1, z_2) $ is symmetrical and, fixing one of the variables, is a  Laurent polynomial   in the other variable of positive maximum degree $ n-1 $ and  negative degree  $-m $.
\end{pro}

\begin{pro}\label{Geronimus Connection Cauchy}
The  second kind functions are subject to the following  connection formulas
	\begin{align}
 \label{CC12}
	(C_2(z))^\dagger \big(\Omega^{(1,2)}_2\big)^\dagger - L_\Gamma^{(1)}(\bar z)(\tilde C^{(1,2)}_2(z))^\dagger &=-\prodint{
		\delta L_\Gamma^{(1)}(z_1,\bar z),\big(\tilde\phi^{(1,2)}_2(z_2)\big)^\top}_{\tilde u^{(1,2)}},\\
	\label{CC21}	\Omega^{(2,1)}_1 C_{1}(z)-\tilde C_{1}^{(2,1)}(z) \overline{L_\Gamma^{(2)}}(z)&=-\prodint{ \tilde\phi^{(2,1)}_1(z_1),
		\delta L_\Gamma^{(2)}(z_2,\bar z)}_{ \tilde u^{(2,1)} },
	\\\label{CC11}
	\Omega^{(1,2)}_1 \tilde C^{(1,2)}_1(z)&=\prodint{ \phi_1(z_1), \frac{L_{\mathfrak C}^{(2)}(z_2)}{\bar{z}-z_2}}_{ u },
	\\\label{CC22}
		\big(\tilde C^{(2,1)}_2(z)\big)^\dagger\big(\Omega_2^{(2)}\big)^\dagger&=\prodint{ \frac{L_{\mathfrak C}^{(1)}(z_1)}{\bar{z}-z_1}, \big(\phi_2(z_2)\big)^\top}_{ u }.
	\end{align}
	\end{pro}

	\begin{pro}
		For $l\geq 2\underline{N}_{\Gamma}$, the second kind functions satisfy the  connection formulas
		\begin{align}
		\label{CC12k>}
		(\Omega^{(1,2)}_2)_{l,l-2\underline{N}_{\Gamma}} C_{2,l-2\underline{N}_{\Gamma}}(z)+\dots+	(\Omega^{(1,2)}_2)_{l,l+2\underline{N}_{\mathfrak C}} C_{2,l+2\underline{N}_{\mathfrak C}}(z)&=  \overline{L_\Gamma^{(1)}}( z)\tilde C^{(1,2)}_{2,l}(z),\\
		\label{CC21k>}
		(\Omega^{(2,1)}_1)_{l,l-2\underline{N}_{\Gamma}} C_{1,l-2\underline{N}_{\Gamma}}(z)+\dots+	(\Omega^{(2,1)}_1)_{l,l+2\underline{N}_{\mathfrak C}} C_{1,l+2\underline{N}_{\mathfrak C}}(z)&= \overline{L_\Gamma^{(2)}}(z)\tilde C_{1,l}^{(2,1)}(z).
		\end{align}
	\end{pro}
	\begin{proof}
		It is a consequence of the orthogonality relations  in Corollary \ref{ortogonalidad}, because they involve
		\begin{align*}
		\prodint{\delta L_\Gamma^{(1)}(\bar z,z_1),\big(\tilde\phi^{(1,2)}_{2,l}(z_2)\big)^\top}_{\tilde u^{(1,2)}} =	\prodint{ \tilde \phi^{(2,1)}_{1,l}(z_1), \delta L_\Gamma^{(2)}(\bar z,z_2)}_{ \tilde  u^{(2,1)} }
		&=0,
		&l&\geq 2\underline{N}_{\Gamma}.
		\end{align*}
	\end{proof}

\subsection{Connection formulas for the Christoffel--Darboux kernels and their mixed versions}
\begin{defi}
	Let's  define the following upper triangular $2\underline N_\Gamma \times 2\underline N_\Gamma$ matrices
	\begin{align*}
	\Gamma^{(1,2)}_{2,l}&:=\begin{bmatrix}
	(\Omega^{(1,2)}_2)_{l,l-2\underline{N}_{\Gamma}} & (\Omega^{(1,2)}_2)_{l,l-2\underline{N}_{\Gamma}+1}&(\Omega^{(1,2)}_2)_{l,l-2\underline{N}_{\Gamma}+2}&\dots& (\bar\Omega^{(1,2)}_2)_{l,l-1}\\
	0 & (\Omega^{(1,2)}_2)_{l+1,l-2\underline{N}_{\Gamma}+1}&(\Omega^{(1,2)}_2)_{l+1,l-2\underline{N}_{\Gamma}+2}&\dots& (\Omega^{(1,2)}_2)_{l+1,l-1}\\
	0 &0 &(\Omega^{(1,2)}_2)_{l+2,l-2\underline{N}_{\Gamma}+1}&\dots& (\Omega^{(1,2)}_2)_{l+2,l-1}\\
	\vdots& &\ddots & &\vdots\\
	0&0&\dots& &(\Omega^{(1,2)}_2)_{l+2\underline{N}_{\Gamma}-1,l-1}
	\end{bmatrix},\\
	\Gamma^{(2,1)}_{1,l}&:=\begin{bmatrix}
	(\Omega^{(2,1)}_1)_{l,l-2\underline{N}_{\Gamma}} & (\Omega^{(2,1)}_1)_{l,l-2\underline{N}_{\Gamma}+1}&(\Omega^{(2,1)}_1)_{l,l-2\underline{N}_{\Gamma}+2}&\dots& (\Omega^{(2,1)}_1)_{l,l-1}\\
	0 & (\Omega^{(2,1)}_1)_{l+1,l-2\underline{N}_{\Gamma}+1}&(\Omega^{(2,1)}_1)_{l+1,l-2\underline{N}_{\Gamma}+2}&\dots& (\Omega^{(2,1)}_1)_{l+1,l-1}\\
	0 &0 &(\Omega^{(2,1)}_1)_{l+2,l-2\underline{N}_{\Gamma}+1}&\dots& (\Omega^{(2,1)}_1)_{l+2,l-1}\\
	\vdots& &\ddots & &\vdots\\
	0&0&\dots& &(\Omega^{(2,1)}_1)_{l+2\underline{N}_{\Gamma}-1,l-1}
	\end{bmatrix},\end{align*}\
and the following  lower triangular $2\underline{N}_{\mathfrak C}\times 2\underline{N}_{\mathfrak C}$ matrices		
		\begin{align*}
	\mathfrak C^{(1,2)}_{2,l}&:=\begin{bmatrix}
	(\Omega^{(1,2)}_2)_{l-2\underline{N}_{\mathfrak C},l} & 0 &0&\dots& 0\\
	(\Omega^{(1,2)}_2)_{l-2\underline{N}_{\mathfrak C}+1,l} & (\Omega^{(1,2)}_2)_{l-2\underline{N}_{\mathfrak C}+1,l+1}&0&\dots& 0\\
	(\Omega^{(1,2)}_2)_{l-2\underline{N}_{\mathfrak C}+2,l} &(\Omega^{(1,2)}_2)_{l-2\underline{N}_{\mathfrak C}+2,l+1} &(\Omega^{(1,2)}_2)_{l-2\underline{N}_{\mathfrak C}+2,l+2}&\dots& 0\\
	\vdots& &&\ddots  &\vdots\\
	(\Omega^{(1,2)}_2)_{l-1,l}&(\Omega^{(1,2)}_2)_{l-1,l+1}&(\Omega^{(1,2)}_2)_{l-1,l+2}&\dots& (\Omega^{(1,2)}_2)_{l-1,l+2\underline{N}_{\mathfrak C}-1}
	\end{bmatrix},\\
		\mathfrak C^{(2,1)}_{1,l}&:=\begin{bmatrix}
		(\Omega^{(2,1)}_1)_{l-2\underline{N}_{\mathfrak C},l} & 0 &0&\dots& 0\\
		(\Omega^{(2,1)}_1)_{l-2\underline{N}_{\mathfrak C}+1,l} & (\Omega^{(2,1)}_1)_{l-2\underline{N}_{\mathfrak C}+1,l+1}&0&\dots& 0\\
		(\Omega^{(2,1)}_1)_{l-2\underline{N}_{\mathfrak C}+2,l} &(\Omega^{(2,1)}_1)_{l-2\underline{N}_{\mathfrak C}+2,l+1} &(\Omega^{(2,1)}_1)_{l-2\underline{N}_{\mathfrak C}+2,l+2}&\dots& 0\\
		\vdots& &&\ddots  &\vdots\\
		(\Omega^{(2,1)}_1)_{l-1,l}&(\Omega^{(2,1)}_1)_{l-1,l+1}&(\Omega^{(2,1)}_1)_{l-1,l+2}&\dots& (\Omega^{(2,1)}_1)_{l-1,l+2\underline{N}_{\mathfrak C}-1}
		\end{bmatrix}.
	\end{align*}
\end{defi}

\begin{pro}\label{Geronimus Connection CD}
For $l \geq 2\max (\underline{N}_{\mathfrak C},\underline{N}_{\Gamma})$, the Christoffel--Darboux kernels and their Geronimus--Uvarov transformations
	satisfy	the following connection formulas
\begin{multline}\label{GerKerNor1}
 \overline{L_{\mathfrak C}^{(2)}( z_1)}	{\tilde K^{(1,2),[l]}(\bar z_1,z_2)}-{L_\Gamma^{(1)}(z_2)} {K^{[l]}(\bar z_1,z_2)}\\={\Big[\tilde \phi^{(1,2)}_{1,l-2\underline{N}_{\mathfrak C}}(z_2)\big({\tilde H}_{l-2\underline{N}_{\mathfrak C}}^{(1,2)}\big)^{-1},\dots,
		\tilde \phi^{(1,2)}_{1,l+2\underline{N}_{\Gamma}-1}(z_2)\big({\tilde H}_{l+2\underline{N}_{\Gamma}-1}^{(1,2)}\big)^{-1}\Big]}	\begin{bmatrix}
	0_{2\underline{N}_{\mathfrak C}\times 2\underline{N}_{\Gamma}} & \overline{\mathfrak C_{2,l}^{(1,2)}}\\
-	 \overline{\Gamma_{2,l}^{(1,2)}}& 0_{2\underline{N}_{\Gamma}\times 2\underline{N}_{\mathfrak C}}
	\end{bmatrix}
\begin{bmatrix}
\overline{ \phi_{2,l-2\underline{N}_{\Gamma}}(z_1)}\\ \vdots\\  \overline{
	\phi_{2,l+2\underline{N}_{\mathfrak C}-1}(z_1)}\end{bmatrix},
\end{multline}
\begin{multline}
\label{GerKerNor2}
L_{\mathfrak C}^{(1)}(z_2)	\tilde K^{(2,1),[l]}(\bar z_1,z_2)- \overline{ L_\Gamma^{(2)}( z_1)} K^{[l]}(\bar z_1,z_2)\\={\Big[\overline{\tilde \phi^{(2,1)}_{2,l-2\underline{N}_{\mathfrak C}}(z_1)}({\tilde H}_{l-2\underline{N}_{\mathfrak C}}^{(2,1)})^{-1},\dots,
	\overline{\tilde \phi^{(2,1)}_{2,l+2\underline{N}_{\Gamma}-1}(z_1)}({\tilde H}_{l+2\underline{N}_{\Gamma}-1}^{(2,1)})^{-1}\Big]}\begin{bmatrix}
0_{2\underline{N}_{\mathfrak C}\times 2\underline{N}_{\Gamma}} &  \mathfrak C_{1,l}^{(2,1)}\\
- \Gamma_{1,l}^{(2,1)}& 0_{2\underline n\times 2\underline{N}_{\mathfrak C}}
\end{bmatrix}
\begin{bmatrix}
 \phi_{1,l-2\underline{N}_{\Gamma}}(z_2)\\ \vdots\\
	\phi_{1,l+2\underline{N}_{\mathfrak C}-1}(z_2)\end{bmatrix}.
	\end{multline}
\end{pro}

\begin{pro}\label{Geronimus Connection mixed}
	For $l\geq 2\max(\underline{N}_{\mathfrak C},\underline{N}_{\Gamma})$, under Geronimus--Uvarov transformations the mixed Christoffel--Darboux kernels satisfy the  connection formulas
	\begin{align}
\label{mix CF1}	
\begin{multlined}[t][0.9\textwidth]
L_\Gamma^{(1)}(\bar x_1)\tilde K_{C_2}^{(1,2),[l]}(\bar x_1,x_2)-	L_\Gamma^{(1)}(x_2)K_{C_2}^{[l]}(\bar x_1,x_2)-	
\delta L_\Gamma^{(1)}(\bar x_1,x_2)
\\= {\Big[\tilde \phi^{(1,2)}_{1,l-2\underline{N}_{\mathfrak C}}(x_2)({\tilde H}_{l-2\underline{N}_{\mathfrak C}}^{(1)})^{-1},\dots,
	\tilde \phi^{(1,2)}_{1,l+2\underline{N}_{\Gamma}-1}(x_2)({\tilde H}_{l+2\underline{N}_{\Gamma}-1}^{(1)})^{-1}\Big]}
\begin{bmatrix}
0_{2\underline{N}_{\mathfrak C}\times 2\underline{N}_{\Gamma}} & \overline{\mathfrak  C^{(1,2)}_{2,l}}\\
-\overline{\Gamma^{(1,2)}_{2,l} }& 0_{2\underline{N}_{\Gamma}\times 2\underline{N}_{\mathfrak C}}
\end{bmatrix}
\begin{bmatrix}
\overline{ C_{2,l-2\underline{N}_{\Gamma}}(x_1)}\\ \vdots\\  \overline{
	C_{2,l+2\underline{N}_{\mathfrak C}-1}(x_1)}
\end{bmatrix},
\end{multlined}	
\end{align}
\begin{align}
\label{mix CF2}	
\begin{multlined}[t][0.9\textwidth]
L_{\mathfrak C}^{(1)}(x_2)\tilde K_{C_2}^{(2,1),[l]}(\bar x_1,x_2)-\prodint{\frac{L_{\mathfrak C}^{(1)}(z_1)}{\bar x_1-z_1},\overline{K^{[l]}(\bar z_2,x_2)}}_{u}
\\=\Big[\overline{\tilde  C^{(2,1)}_{2,l-2\underline{N}_{\mathfrak C}}(x_1)}(\tilde  H^{(2,1)}_{l-2\underline{N}_{\mathfrak C}})^{-1},\dots, \overline{\tilde  C^{(2,1)}_{2,l+2\underline{N}_{\Gamma}-1}(x_1)}(\tilde  H^{(2,1)}_{l+2\underline{N}_{\Gamma}-1})^{-1}\Big]
\begin{bmatrix}
0_{2\underline{N}_{\mathfrak C}\times2\underline{N}_{\Gamma}} &
\mathfrak  C^{(2,1)}_{1,l}\\
-\Gamma^{(2,1)}_{1,l} &0_{2\underline{N}_{\Gamma}\times2\underline{N}_{\mathfrak C}}
\end{bmatrix}
\begin{bmatrix}
\phi_{1,l-2\underline{N}_{\Gamma}}(x_2)\\\vdots\\\phi_{1,l+2\underline{N}_{\mathfrak C}-1}(x_2)
\end{bmatrix}
\end{multlined}
\end{align}		
\begin{align}	\label{mix FC1}	
\begin{multlined}[t][0.9\textwidth]
\overline{L_{\mathfrak C}^{(2)}(x_1)}\tilde K_{C_1}^{(1,2),[l]}(\bar x_1,x_2)-\prodint{K^{[l]}(\bar x_1,z_1),\frac{L_{\mathfrak C}^{(2)}(z_2)}{\bar x_2-z_2}}_{u}\\
=\Big[\tilde  C^{(1,2)}_{1,l-2\underline{N}_{\mathfrak C}}(x_2)({\tilde H}^{(1)}_{l-2\underline{N}_{\mathfrak C}})^{-1},\dots,
	\tilde  C^{(1,2)}_{1,l+2\underline{N}_{\Gamma}-1}(x_2)({\tilde H}^{(1)}_{l+2\underline{N}_{\Gamma}-1})^{-1}\Big]
\begin{bmatrix}
0_{2\underline{N}_{\mathfrak C}\times2\underline{N}_{\Gamma}} &\overline{\mathfrak  C^{(1,2)}_{2,l}}\\
-\overline{\Gamma^{(1,2)}_{2,l} }&0_{2\underline{N}_{\Gamma}\times2\underline{N}_{\mathfrak C}}
\end{bmatrix}
\begin{bmatrix}
\overline{ \phi_{2,l-2\underline{N}_{\Gamma}}(x_1)}\\ \vdots\\  \overline{
	\phi_{2,l+2\underline{N}_{\mathfrak C}-1}(x_1)}
\end{bmatrix},
\end{multlined}
\end{align}
\begin{align}\label{mix FC2}	
\begin{multlined}[t][0.9\textwidth]
 \overline{L_\Gamma^{(2)}}(x_2)\tilde K_{C_1}^{(2,1),[l]}(\bar x_1,x_2)-
 \overline{L_\Gamma^{(2)}}(\bar x_1) K_{C_1}^{[l]}(\bar x_1,x_2)
-	\delta \overline{L_\Gamma^{(2)}}(\bar x_1,x_2)\\
=\Big[\overline{\tilde \phi^{(2,1)}_{2,l-2\underline{N}_{\mathfrak C}}(x_1)}(\tilde  H^{(2,1)}_{l-2\underline{N}_{\mathfrak C}})^{-1},\dots, \overline{\tilde \phi^{(2,1)}_{2,l+2\underline{N}_{\Gamma}-1}(x_1)}(\tilde  H^{(2,1)}_{l+2\underline{N}_{\Gamma}-1})^{-1}\Big]
\begin{bmatrix}
0_{2\underline{N}_{\mathfrak C}\times 2\underline{N}_{\Gamma}}&
\mathfrak  C^{(2,1)}_{1,l}\\
-\Gamma^{(2,1)}_{1,l}&0_{2\underline{N}_{\Gamma}\times 2\underline{N}_{\mathfrak C}}
\end{bmatrix}
\begin{bmatrix}
C_{1,l-2\underline{N}_{\Gamma}}(x_2)\\\vdots\\C_{1,l+2\underline{N}_{\mathfrak C}-1}(x_2)
\end{bmatrix}.
\end{multlined}
\end{align}
\end{pro}

\subsection{Christoffel  formulas for Geronimus--Uvarov transformations}
The Geronimus--Uvarov perturbed  second kind functions are
\begin{align*}
\overline{	\tilde C_{2,k}^{(1,2)}(z)}
&=\prodint{\frac{1}{\bar z-z_1}, \tilde \phi^{(1,2)}_{2,k}(z_2)}_{ \frac{\overline{(L_{\mathfrak C}^{(2)}(z_2))}}{L_\Gamma^{(1)}(z_1)}u}+\sum_{i=1}^{d^{(1)}}\sum_{l=0}^{m^{(1)}_{i}-1}
\frac{1}{l!}\frac{\d^l}{\d \zeta^l}
\Big(
\frac{1}{ \bar z-\zeta}\Big)
\bigg|_{\zeta= \zeta^{(1)}_{i}}\overline{\prodint{L^{(2)}_{\mathfrak C}\xi^{(1)}_{i,l},	\tilde\phi^{(1,2)}_{2,k}}},\end{align*}
for $ \bar z\in\mathbb C\setminus {\operatorname{supp}_1(u)\cup\sigma(L_\Gamma^{(1)})}$ and
\begin{align*}
\tilde C_{1,k}^{(2,1)}(z)
&=\prodint{\tilde\phi^{(2,1)}_{1,k}(z_1),
	\frac{1}{\bar z-z_2}}_{\frac{L_{\mathfrak C}^{(1)}(z_1)}{\overline{(L_\Gamma^{(2)}(z_2))}}u}+\sum_{i=1}^{d^{(2)}}\sum_{l=0}^{m^{(2)}_{i}-1}\prodint{L^{(1)}_{\mathfrak C}\xi^{(2)}_{i,l},	\tilde\phi^{(2,1)}_{1,k}}
\frac{1}{l!}\overline{\frac{\d^l}{\d \zeta^l}
	\Big(
	\frac{1}{ \bar z-\zeta}\Big)
	\bigg|_{\zeta= \zeta^{(2)}_{i}}},
\end{align*}
for $z\in\mathbb C\setminus\overline{\operatorname{supp}_2(u)\cup\sigma(L_\Gamma^{(2)})}$.
\begin{defi}
	\begin{enumerate}
		\item  For  $a\in\{1,2\}$ we define
		\begin{align*}
		\prodint{\xi^{(a)}_i,
			f}&:=\Big[\prodint{\xi^{(a)}_{i,1},
			f},\dots,\prodint{\xi^{(a)}_{i,m_1-1},
			f}\Big], &\prodint{\xi^{(a)},
			f}&:=\Big[\prodint{\xi^{(a)}_{1},	f
		},\dots,\prodint{\xi^{(a)}_{d}, f}
		\Big].
		\end{align*}
		\item The expression   $L(z)=L_{N^+} z^{-N^-}\prod\limits_{j=1}^{d}(z-\zeta_j)^{m_j}$ with  $m_1+\dots+m_d=N^++N^-$, allows us to introduce
		\begin{align*}
		L_{[i]}(z):=
		L_{N^+} z^{-N^-}\prod\limits_{\substack{j=1\\j\neq i}}^{d}(z-\zeta_j)^{m_j}.
		\end{align*}
	\end{enumerate}
\end{defi}

Relations \eqref{CC21k>} and  \eqref{CC12k>} motivate us to consider
\begin{align*}
\overline{	 \overline{L_\Gamma^{(1)}}(z)\tilde C_{2,k}^{(1,2)}(z)}
&=\prodint{\frac{L_\Gamma^{(1)}(\bar z)}{\bar z-z_1}, \tilde \phi^{(1,2)}_{2,k}(z_2)}_{ \frac{\overline{(L_{\mathfrak C}^{(2)}(z_2))}}{L_\Gamma^{(1)}(z_1)}u}+\sum_{i=1}^{d^{(1)}}
\overline{\prodint{L^{(2)}_{\mathfrak C}\xi^{(1)}_{i},	\tilde\phi^{(1,2)}_{2,k}}}
L^{(1)}_{\Gamma,[i]}(\bar z)\begin{bmatrix}
(\bar z- \zeta^{(1)}_{i})^{m^{(1)}_{i}-1}\\ (\bar z- \zeta^{(1)}_{i})^{m^{(1)}_{i}-2}\\\vdots\\1
\end{bmatrix},\\
\overline{ L_\Gamma^{(2)}}(z)\tilde C_{1,k}^{(2,1)}(z)&=
\prodint{\tilde\phi^{(2,1)}_{1,k}(z_1),\frac{L_\Gamma^{(2)}(\bar z)}{\bar z-z_2}}_{\frac{L^{(1)}_{\mathfrak C}(z_1)}{\overline{(L_\Gamma^{(2)}(z_2))}}u}
+\sum_{i=1}^{d^{(2)}}\prodint{L^{(1)}_{\mathfrak C}\xi^{(2)}_{i},	\tilde\phi^{(2,1)}_{1,k}}
 \overline{L^{(2)}_{\Gamma,[i]}}(z)
\begin{bmatrix}
(z-\bar \zeta^{(2)}_{i})^{m^{(2)}_{i}-1}\\ (z-\bar \zeta^{(2)}_{i})^{m^{(2)}_{i}-2}\\\vdots\\1
\end{bmatrix}.
\end{align*}
When the  evaluation of the spectral jets
$\mathcal J^{\overline{ L_\Gamma^{(1)}}}_{\overline{ L_\Gamma^{(1)}}\tilde C_{2,k}^{(1,2)}}$ y $\mathcal J^{\overline{ L_\Gamma^{(2)}}}_{\overline{ L_\Gamma^{(2)}}\tilde C_{1,k}^{(2,1)}}$ is required
we perform  it by taking appropriated  limits  ---and in this manner we take care of the fact that the perturbing polynomial zeros lay  on the border of the perturbed functional support.

\begin{defi}		For $a\in\{1,2\}$, $j\in\{1,\dots,d^{(a)}\}$, $k\in\{0,\dots,m^{(a)}_{j}-1\}$, we define
	\begin{align*}
	\ell_{[j],k}^{(a)}&:=	\frac{1}{k!}
	\frac{\d^{k} \overline{L^{(a)}_{\Gamma,[j]}}(z)}{\d z^{k}}\bigg|_{z=\bar \zeta^{(a)}_{j}}.
	\end{align*}
We also introduce
	\begin{align*}
	\mathcal L^{(a)}_{[j]}&:=	\begin{bmatrix}0 &0 & 0& & \ell^{(a)}_{[j],0}\\
	\vdots&\vdots &\vdots & \iddots&\vdots\\ 0&0 & \ell^{(a)}_{[j],0}&\dots &\ell^{(a)}_{[j],m^{(a)}_{j}-3}\\	
	0&	 \ell^{(a)}_{[j],0} &\ell^{(a)}_{[j],1} & \dots&\ell^{(a)}_{[j],m^{(a)}_{j}-2}\\
 \ell^{(a)}_{[j],0}	&\ell^{(a)}_{[j],1} &\ell^{(a)}_{[j],2} & \dots&\ell^{(a)}_{[j],m^{(a)}_{j}-1}
	\end{bmatrix}, &\mathcal L^{(a)}&:=\diag(\mathcal L^{(a)}_{[1]},\dots,\mathcal L^{(a)}_{[d^{(a)}]}).
	\end{align*}
\end{defi}

\begin{pro}\label{pro.chorros}
	The spectral jets fulfill
	\begin{align}
	\label{caleGer}
	\mathcal J^{ \overline{L_\Gamma^{(2)}}}_{\overline{L_\Gamma^{(2)}}\tilde C^{(2,1)}_{{1,k}}}&=\prodint{L^{(1)}_{\mathfrak C}\xi^{(2)},\tilde\phi^{(2,1)}_{1,k}}\mathcal L^{(2)},\\\label{chorro (1)}
	\mathcal J^{\overline{L_\Gamma^{(1)}}}_{\overline{L_\Gamma^{(1)}}\tilde C_{2,k}^{(1,2)}}&=	{\prodint{L^{(2)}_{\mathfrak C}\xi^{(1)},	\tilde\phi^{(1,2)}_{2,k}}\mathcal L^{(1)}}.
	\end{align}
\end{pro}

\begin{teo}[Christoffel--Geronimus--Uvarov formulas]\label{Christoffel-Geronimus formulas}

	Let's assume that all perturbing  Laurent polynomials are prepared and consider the determinants
\begin{align*}
\tilde \tau_l^{(2,1)}&:=\begin{vsmallmatrix}
\mathcal J^{\overline{L_\Gamma^{(2)}}}_{ C_{1,l-2 N_\Gamma}}-\prodint{\xi^{(2)},\phi_{1,l-2 N_\Gamma}}\mathcal L^{(2)} & \mathcal J^{{L_{\mathfrak C}^{(1)}}}_{ \phi_{1,l-2 N_\Gamma}}\\\vdots  &\vdots \\
\mathcal J^{\overline{L_\Gamma^{(2)}}}_{ C_{1,l+2 N_{\mathfrak C}-1}}-\prodint{\xi^{(2)},\phi_{1,l+2 N_{\mathfrak C}-1}}\mathcal L^{(2)}& 	\mathcal J^{{L_{\mathfrak C}^{(1)}}}_{ \phi_{1,l+2 N_{\mathfrak C}-1}}
\end{vsmallmatrix}, &\tilde \tau^{(1,2)}_l&:=\begin{vsmallmatrix}
	\mathcal J^{\overline{L_\Gamma^{(1)}}}_{ C_{2,l-2 N_\Gamma}}-\prodint{\xi^{(1)},\phi_{2,l-2 N_\Gamma}}\mathcal L^{(1)} & \mathcal J^{{L_{\mathfrak C}^{(2)}}}_{ \phi_{2,l-2 N_\Gamma}}\\\vdots  &\vdots \\
	\mathcal J^{\overline{L_\Gamma^{(1)}}}_{ C_{2,l+2 N_{\mathfrak C}-1}}-\prodint{\xi^{(1)},\phi_{2,l+2 N_{\mathfrak C}-1}}\mathcal L^{(1)}& 	\mathcal J^{{L_{\mathfrak C}^{(2)}}}_{ \phi_{2,l+2 N_{\mathfrak C}-1}}
	\end{vsmallmatrix}.
\end{align*}

Then, for $l\geq 2\max(N_\Gamma,N_{\mathfrak C})$ and  $\tilde \tau_l^{(2,1)}\neq 0$
	we have the following determinantal  formulas
	\begin{gather}\label{Ger12}
\tilde\phi^{(2,1)}_{1,l}(z)
=
\frac{	L^{(1)}_{\mathfrak C,(-1)^{l} N_{\mathfrak C}}}{L_{\mathfrak C}^{(1)}(z)}\frac{1}{\tilde\tau^{(2,1)}_l}\begin{vmatrix}
\mathcal J^{\overline{L_\Gamma^{(2)}}}_{ C_{1,l-2 N_\Gamma}}-\prodint{\xi^{(2)},\phi_{1,l-2 N_\Gamma}}\mathcal L^{(2)} & \mathcal J^{{L_{\mathfrak C}^{(1)}}}_{ \phi_{1,l-2 N_\Gamma}}& \phi_{1,l-2 N_\Gamma}(z)\\\vdots  &\vdots&\vdots \\
\mathcal J^{\overline{L_\Gamma^{(2)}}}_{ C_{1,l+2 N_{\mathfrak C}}}-\prodint{\xi^{(2)},\phi_{1,l+2 N_{\mathfrak C}}}\mathcal L^{(2)}& 	\mathcal J^{{L_{\mathfrak C}^{(1)}}}_{ \phi_{1,l+2 N_{\mathfrak C}}}&\phi_{1,l+2 N_{\mathfrak C}}(z)
\end{vmatrix},\\
\label{GerH2}
\tilde  H^{(2,1)}_{l}=	 \dfrac{L^{(1)}_{\mathfrak C,(-1)^{l} N_{\mathfrak C}}}{ \overline{L^{(2)}_{\Gamma,(-1)^lN_\Gamma}}}H_{l-2 N_\Gamma}
\frac{\tilde\tau^{(2,1)}_{l+1}}{\tilde\tau^{(2,1)}_l},
\\\label{Ger22}
	\overline{\tilde \phi^{(2,1)}_{2,l}(z)}=-\frac{	\overline{ L_\Gamma^{(2)}( z)}}{\overline{L^{(2)}_{\Gamma,(-1)^lN_\Gamma}}}\frac{H_{l-2 N_\Gamma}}{\tilde\tau^{(2,1)}_l}\begin{vmatrix}
\mathcal J^{ \overline{L_\Gamma^{(2)}}}_{ 	C_{1,l-2 N_\Gamma+1}}-\prodint{\xi^{(2)},\phi_{1,l-2 N_{\Gamma}+1}}\mathcal L^{(2)}&\mathcal J^{L_{\mathfrak C}^{(1)}}_{\phi_{1,l-2 N_\Gamma+1}}\\\vdots&\vdots\\\mathcal J^{ \overline{L_\Gamma^{(2)}}}_{ C_{1,l+2 N_{\mathfrak C}-1}}-\prodint{\xi^{(2)},\phi_{1,l+2 N_{\mathfrak C}-1}}\mathcal L^{(2)}&\mathcal J^{L_{\mathfrak C}^{(1)}}_{\phi_{1,l+2 N_{\mathfrak C}-1}}\\
\mathcal  J^{\overline{L_\Gamma^{(2)}}}_{K_{C_1}^{[l-2N_\Gamma+1]}}(\bar z)-\prodint{(\xi^{(2)})_w,K^{[l-2N_\Gamma+1]}(\bar z,w)}\mathcal L^{(2)}
+\frac{1}{\overline{L_\Gamma^{(2)}( z)} }\mathcal  J^{\overline{L_\Gamma^{(2)}}}_{\delta \overline{L_\Gamma^{(2)}}}(\bar z)&
 \mathcal J^{L_{\mathfrak C}^{(1)}}_{K^{[l-2N_\Gamma+1]}}(\bar z)
\end{vmatrix}
\end{gather}
	where the spectral of the mixed Christoffel--Darboux kernel and of  $\delta  \overline{L_\Gamma^{(2)}}$ are taken with respect to the second variable.
	Whenever  $l\geq 2\max(N_\Gamma, N_{\mathfrak C})$ and $\tilde\tau^{(1,2)}_{l}\neq 0$ the following
	formulas are satisfied
\begin{align}\label{Ger21}
	\tilde\phi^{(1,2)}_{2,l}(z)
&=
\frac{	L^{(2)}_{\mathfrak C,(-1)^{l} N_{\mathfrak C}}}{L_{\mathfrak C}^{(2)}(z)}\frac{1}{\tilde \tau^{(1,2)}_l}\begin{vmatrix}
\mathcal J^{\overline{L_\Gamma^{(1)}}}_{ C_{2,l-2 N_\Gamma}}-\prodint{\xi^{(1)},\phi_{2,l-2 N_\Gamma}}\mathcal L^{(1)} & \mathcal J^{{L_{\mathfrak C}^{(2)}}}_{ \phi_{2,l-2 N_\Gamma}}& \phi_{2,l-2 N_\Gamma}(z)\\\vdots  &\vdots&\vdots \\
\mathcal J^{\overline{L_\Gamma^{(1)}}}_{ C_{2,l+2 N_{\mathfrak C}}}-\prodint{\xi^{(1)},\phi_{2,l+2 N_{\mathfrak C}}}\mathcal L^{(1)}& 	\mathcal J^{{L_{\mathfrak C}^{(2)}}}_{ \phi_{2,l+2 N_{\mathfrak C}}}&\phi_{2,l+2 N_{\mathfrak C}}(z)
\end{vmatrix}
	\end{align}
	\begin{align}\label{GerH1}
	\bar{	\tilde H}^{(1,2)}_{l}&=	 \dfrac{L^{(2)}_{\mathfrak C,(-1)^{l} N_{\mathfrak C}}}{ \overline{L^{(2)}_{\Gamma,(-1)^lN_\Gamma}}}\bar H_{l-2 N_\Gamma}\frac{\tilde \tau^{(1,2)}_{l+1}}{\tilde \tau^{(1,2)}_{l}},
	\end{align}
	\begin{multline}	\label{Ger11}
\tilde \phi^{(1,2)}_{1,l}(z)
	\\=-\frac{{L_\Gamma^{(1)}}(z)	}{L^{(1)}_{\Gamma,(-1)^lN_\Gamma}}\frac{H_{l-2 N_\Gamma}}{\tilde\tau^{(1,2)}_l}
\begin{vmatrix}
\overline{ \mathcal J^{ \overline{L_\Gamma^{(1)}}}_{C_{2,l-2 N_\Gamma+1}}-\prodint{\xi^{(1)},\phi_{2,l-2 N_\Gamma+1}}\mathcal L^{(1)}}&\overline{ \mathcal J^{{L_{\mathfrak C}^{(2)}}}_{\phi_{2,l-2 N_\Gamma+1}}}\\
\vdots &\vdots\\  \overline{
	\mathcal J^{ \overline{L_\Gamma^{(1)}}}_{	C_{2,l+2 N_{\mathfrak C}-1}}-\prodint{\xi^{(1)},\phi_{2,l+2 N_{\mathfrak C}-1}}\mathcal L^{(1)}}&\overline{ \mathcal J^{{L_{\mathfrak C}^{(2)}}}_{\phi_{2,l+2 N_{\mathfrak C}-1}}}\\
\mathcal J^{ \overline{L_\Gamma^{(1)}}}_{K_{C_2}^{[l-2N_\Gamma+1]}}(z)
-\prodint{\overline{(\xi^{(1)})_w},{K^{[l-2N_\Gamma+1]}(\bar w,z)}}
\overline{\mathcal L^{(1)}}
+
\frac{1}{L_\Gamma^{(1)}(z)}
\mathcal J^{ \overline{L_\Gamma^{(1)}}}_{\delta {L_\Gamma^{(1)}}}(z)&\mathcal J^{\overline{L_{\mathfrak C}^{(2)}}}_{K^{[l-2N_\Gamma+1]}}(z)
\end{vmatrix},
\end{multline}
	where the spectral jet of the Christoffel--Darboux kernels and of $\delta L_\Gamma^{(1)}$ are taken with respect to the first variable.
\end{teo}

\section{Reductions}

A bivariate linear functional  $u_{z_1,\bar z_2}$ is supported on the diagonal $z_1=z_2$ if
\begin{align*}
\prodint{ f(z_1), g(z_2)}_u=\sum_{0\leq n,m\ll \infty}\prodint{u_z^{(n,m)}, \frac{\partial^nf}{\partial z^n} ( z)\overline{\frac{\partial^{m} g}{\partial z^m} (z)}},
\end{align*}
where $u_z^{(n,m)}$ are univariate linear functionals, i.e., we are dealing with a Sobolev sesquilinear form.
A particularly relevant example, is the  zero order  diagonal case
\begin{align*}
\prodint {f(z),g(z)}_u=\prodint{u_z,f(z)\overline{g(z)}}.
\end{align*}

The  Toeplitz case appears, for  zero order diagonal situations,  when
the support of the liner functional lays in the unit circle, $\operatorname{supp}u\subset \mathbb T$. 
\footnote{For higher orders diagonal situations supported in the unit circle we  can also find Toeplitz matrices. Indeed, that is the case for the following linear functional $\prodint{f(z),g(z)}_u=\prodint{u_z, z\frac{\partial f}{\partial z} ( z)\overline{g(z)}}-\prodint{u_z,f ( z)\overline{z\frac{\partial g}{\partial z} (z)}}$ where $u_z$ is a  functional supported in the unit circle.}

For these cases,  we have
$g_{n,m}=\prodint{u_z,z^{n-m}}$,
which happens to be a  Toeplitz matrix, $g_{n,m}=g_{n+1,m+1}$. In this  Toeplitz scenario   the CMV Gram matrix has the following moment matrix form
$G=\prodint{u_z,\chi(z)(\chi(z))^\dagger}$.

Another important reduction appears for the zero order diagonal case when the linear functional $u_z$ happens to be, not only quasidefinite, but also real, i.e. $\overline{\prodint{u_z,f(z)}}=\prodint{u_z,\overline{f(z)}}$ for every  test function $f(z)\in\mathcal F$.  Then, the Gram matrix is also Hermitian $G=G^\dagger$  and, consequently, $S_1=S_2$ and $H=\bar H$. For non-negative  linear functionals $u_z$ ($\prodint{u_z,f(z)}\in\mathbb R^+$, for every real test function $f:\mathbb C\to\mathbb R^+$).
For  real linear functionals the biorthogonality collapses to pseudo-orthogonality   (or quasidefinite  orthogonality), and for non-negative  linear functionals  to  orthogonality, in this case we have $H_l>0$, $l\in\{0,1,\dots\}$.

	Féjer \cite{fejer} and Riesz \cite{riesz} found a representation  for nonnegative  trigonometric polynomials.  Nonnegative trigonometric polynomials of the form
$f(\theta)=a_0+\sum_{k=1}^n(a_k\cos(k\theta)+b_k\sin(k\theta))$
can  always be written as  $f(\theta)=|p(z)|^2$ where $p(z)=\sum\limits_{l=0}^n p_nz^n$ and $z=\Exp{\operatorname{i}\theta}$.
This is equivalent to write \cite{Grenader}
$f(\theta)=z^{-n} P(z)$ with $z=\Exp{\operatorname{i}\theta}$
and $P(z)\in\mathbb C[z]$ is a polynomial with $\deg P(z)=2n$ such that
$ P(z)=P^*(z)$, fulfilling $z^{-n}P(z)=|P(z)|$ for $z\in\mathbb T$.
The Szeg\H{o} reciprocal polynomial is $ P^*(z):= z^{2 n}\bar P(z^{-1})$ of $P(z)$.  Observe that  for $z\in\mathbb C^*$ the function $L(z)=z^{-n} P(z)$ is not any more a trigonometric polynomial but a Laurent polynomial. Given a Laurent polynomial its reciprocal is given by $L_*(z):=\bar L(z^{-1})$, thus for $L(z)=z^{-n}P(z)$ we have $L_*(z)= z^n\bar P (z^{-1})=z^{-n}P^*(z)$ and if $P(z)=P^*(z)$ we find $L_*(z)=L(z)$; the positivity condition reads: $f(\theta)=L(z)$ with $L(z)$ a Laurent polynomial with $L(z)=L_*(z)$ and $L(z)=|L(z)|$ for $z\in\mathbb T$.
Notice that the self-reciprocal condition $L(z)=L_*(z)$ is simply the reality condition for the corresponding Laurent polynomial on the unit circle,  that is  $L(z)=\overline{L(z)}$ for $z\in\mathbb T$, or equivalently that the corresponding trigonometric polynomial takes real values.

The Geronimus--Uvarov perturbations  within a zero order Toeplitz scenario are of the form
\begin{align*}
\tilde u^{(1,2)}_{z_1,\bar z_2}&=\frac{{L_{\mathfrak C,*}^{(2)}(z)}
}{L_\Gamma^{(1)}(z)}	(u)_z+\overline{{L_{\mathfrak C}^{(2)}(z_2)}}\sum_{i=1}^{d^{(1)}}\sum_{l=0}^{m^{(1)}_{i}-1}\frac{(-1)^{l}}{l!}\delta^{(l)}(z_1-\zeta^{(1)}_{i})
\otimes \overline{(\xi^{(1)}_{i,l})_{ z_2}},  \\
\tilde u^{(2,1)}_{z_1,\bar z_2}
&=\frac{L_{\mathfrak C}^{(1)}(z)}{{L_{\Gamma,*}^{(2)}(z)}}(u)_z+L_{\mathfrak C}^{(1)}(z_1)
\sum_{i=1}^{d^{(2)}}\sum_{l=0}^{m^{(2)}_{i}-1}\frac{(-1)^{l}}{l!}(\xi^{(2)}_{i,l})_{z_1}\otimes\overline{\delta^{(l)}(z_2- \zeta^{(2)}_{i})}.
\end{align*}
When do these two  transformations happen to be same?
Let us
 assume that
$L_{\mathfrak C,*}^{(2)}(z)$ and $L_\Gamma^{(1)}(z)$ are coprime Laurent polynomials and  that
$L_{\mathfrak C}^{(1)}(z)$ and $L_{\Gamma,*}^{(2)}(z)$ are coprime, as well. Thus,  both transformations could possibly be the same  if  only if
\begin{align*}
L_{\mathfrak C,*}^{(2)}(z) &= L_{\mathfrak C}^{(1)}(z)=:L_{\mathfrak C}(z), & L_\Gamma^{(1)}(z)&=L_{\Gamma,*}^{(2)}(z)=:L_\Gamma(z).
\end{align*}
We will take the following linear functionals 
$L_{\mathfrak C,*}(z_2)(\xi^{(1)}_{i,k})_{z_2}=\sum\limits_{j=1}^d\sum\limits_{l=0}^{m_j-1}\frac{(-1)^l}{l!}\bar \Xi_{i,k|j,l}\delta^{(l)}(z_2-(\bar\zeta_{j})^{-1})$ and $
L_{\mathfrak C}(z)(\xi^{(2)}_{j,l})_{z_1}=\sum\limits_{i=1}^d\sum\limits_{k=0}^{m_i-1}\frac{(-1)^k}{k!}\Xi_{i,k|j,l}\delta^{(k)}(z_1-\zeta_{i})$,  with $\zeta_i,m_i$  the zeros and corresponding multiplicities of $L_\Gamma$,
and the Geronimus--Uvarov transforms are
\begin{align*}
\tilde u _{z_1,\bar z_2}&=
\frac{L_{\mathfrak C}(z)}{L_\Gamma(z)}	(u)_z+\sum_{i,j=1}^{d}\sum_{k=0}^{m_{i}-1}\sum_{l=0}^{m_j-1}\frac{(-1)^{k+l}}{k!l!}\Xi_{i,k|j,l}\delta^{(k)}(z_1-\zeta_{i})\otimes\overline{\delta^{(l)}\big(z_2- (\bar \zeta_{j})^{-1}\big)},
\end{align*}
so that
\begin{align*}
\prodint{f(z_1),g(z_2)}_{\tilde u}=\prodint{\frac{L_{\mathfrak C}(z)}{L_\Gamma(z)}(u)_z, f(z)\overline{g(z)}}
	+\sum_{i,j=1}^{d}\sum_{k=0}^{m_{i}-1}\sum_{l=0}^{m_j-1}\frac{1}{k!l!}\Xi_{i,k|j,l}f^{(k)}(\zeta_{i})\overline{g^{(l)}\big( (\bar \zeta_{j})^{-1}\big)}.
\end{align*}
This mass term is possibly not supported neither on the diagonal nor on the unit circle and, therefore, not Toeplitz.
However, it is the most general mass term such that both  Geronimus--Uvarov transformations, of a zero order Toeplitz sesquilinear form,
are equal.  We observe that  \cite{toledano}
\begin{align*}
\prodint{(\xi^{(1)}_{i,k})_{z_2},\phi(z_2)}&=\sum_{j=1}^d\sum_{l=0}^{m_j-1}\frac{1}{l!}\left(\frac{\phi(\zeta)}{L_{\mathfrak C,*}(\zeta)}\right)^{(l)}\bigg|_{\zeta=(\bar\zeta_{j})^{-1}}\bar \Xi_{i,k|j,l},\\
\prodint{(\xi^{(2)}_{j,l})_{z_1},\phi(z_1)}&=\sum_{i=1}^d\sum_{k=0}^{m_i-1}\frac{1}{k!}\left(
\frac{\phi(\zeta)}{L_{\mathfrak C}(\zeta)}\right)^{(k)}\bigg|_{\zeta=\zeta_{i}}\Xi_{i,k|j,l},
\end{align*}
and introduce
\begin{align*}
\Xi&:=
\begin{bmatrix}
\Xi_{1,0|1,0}& \cdots & \Xi_{1,0|1,m_1-1}& \dots &\Xi_{1,0|d,0}& \cdots &\Xi_{1,0|d,m_d-1}\\
\vdots & & \vdots & &\vdots & &\vdots\\
\Xi_{1,m_1-1|1,0}& \cdots& \Xi_{1,m_1-1|1,m_1-1} &\dots &\Xi_{1,m_1-1|d,0}& \cdots &\Xi_{1,m_1-1|d,m_d-1}\\
\vdots & & \vdots & &\vdots & &\vdots\\
\Xi_{d,0|1,0} &\dots &\Xi_{d,0|1,m_1-1}  &\dots &\Xi_{d,0|d,0}& \cdots &\Xi_{d,0|d,m_d-1}\\
\vdots & & \vdots\\
\Xi_{d,m_d-1|1,0}&\cdots &  \Xi_{d,m_d-1|1,m_1-1}& \dots  & \Xi_{d,m_d-1|d,0}& \cdots & \Xi_{d,m_d-1|d,m_d-1}
\end{bmatrix}\in\mathbb C^{2n\times 2n}.
\end{align*}
Then, the following expressions in terms of spectral jets
$\prodint{\xi^{(1)},\phi}=\mathcal J ^{L_{\Gamma,*}}_{\frac{\phi}{L_{\mathfrak C, *}}}\Xi^\dagger$ and $\prodint{\xi^{(2)},\phi}=\mathcal J ^{L_\Gamma}_{\frac{\phi}{L_{\mathfrak C}}}\Xi$ hold.
In this setting it is convenient to introduce
	\begin{align*}
	\tilde \tau_l^{(2,1)}&:=\begin{vmatrix}
	\mathcal J^{\overline{L_{\Gamma,*}}}_{ C_{1,l-2 N_\Gamma}}-
	\mathcal J ^{L_\Gamma}_{\frac{\phi_{1,l-2 N_\Gamma}}{L_{\mathfrak C}}}\Xi\mathcal L_*
& \mathcal J^{{L_{\mathfrak C}}}_{ \phi_{1,l-2 N_\Gamma}}\\\vdots  &\vdots \\
	\mathcal J^{\overline{L_{\Gamma,*}}}_{ C_{1,l+2 N_{\mathfrak C}-1}}-
	\mathcal J^{{L_{\Gamma}}}_{\frac{\phi_{1,l+2 N_{\mathfrak C}-1}}{L_{\mathfrak C}}}
\Xi\mathcal L_*& 	\mathcal J^{{L_{\mathfrak C}}}_{ \phi_{1,l+2 N_{\mathfrak C}-1}}
	\end{vmatrix}, &
	\tilde \tau^{(1,2)}_l&:=\begin{vmatrix}
	\mathcal J^{\overline{L_\Gamma}}_{ C_{2,l-2 N_\Gamma}}-\mathcal J ^{L_{\Gamma,*}}_{\frac{\phi_{2,l-2 N_\Gamma}}{L_{\mathfrak C, *}}}\Xi^\dagger
\mathcal L & \mathcal J^{{L_{\mathfrak C,*}}}_{ \phi_{2,l-2 N_\Gamma}}\\\vdots  &\vdots \\
	\mathcal J^{\overline{L_\Gamma}}_{ C_{2,l+2 N_{\mathfrak C}-1}}-
	\mathcal J ^{L_{\Gamma,*}}_{
		\frac{
			\phi_{2,l+2 N_{\mathfrak C}-1}
		}
	{L_{\mathfrak C, *}}
	}
\Xi^\dagger\mathcal L& 	\mathcal J^{{L_{\mathfrak C,*}}}_{ \phi_{2,l+2 N_{\mathfrak C}-1}}
	\end{vmatrix}.
	\end{align*}

\begin{pro}[Geronimus--Uvarov perturbations of  the zero order Toeplitz case]
Let's assume that all perturbing  Laurent polynomials are prepared. Then, for $l\geq 2\max(N_\Gamma,N_{\mathfrak C})$ and  $\tilde \tau_l^{(2,1)}\tilde\tau^{(1,2)}_{l}\neq 0$
	we have the following determinantal  formulas
	\begin{align*}
	\tilde\phi_{1,l}(z)
	&=
	\frac{	L_{\mathfrak C,(-1)^{l} N_{\mathfrak C}}}{L_{\mathfrak C}(z)}\frac{1}{\tilde\tau^{(2,1)}_l}\begin{vmatrix}
	\mathcal J^{\overline{L_{\Gamma,*}}}_{ C_{1,l-2 N_\Gamma}}-
	\mathcal J ^{L_\Gamma}_{\frac{\phi_{1,l-2 N_\Gamma}}{L_{\mathfrak C}}}\Xi\mathcal L_*
	& \mathcal J^{{L_{\mathfrak C}}}_{ \phi_{1,l-2 N_\Gamma}}& \phi_{1,l-2 N_\Gamma}(z)\\\vdots  &\vdots&\vdots \\
	\mathcal J^{\overline{L_{\Gamma,*}}}_{ C_{1,l+2 N_{\mathfrak C}}}-
	                                                              	\mathcal J^{{L_{\Gamma}}}_{\frac{\phi_{1,l+2 N_{\mathfrak C}}}{L_{\mathfrak C}}}
	\Xi\mathcal L_*&
		\mathcal J^{{L_{\mathfrak C}}}_{ \phi_{1,l+2 N_{\mathfrak C}}}&\phi_{1,l+2 N_{\mathfrak C}}(z)
	\end{vmatrix}\\&=-\frac{{L_\Gamma}(z)	}{L_{\Gamma,(-1)^lN_\Gamma}}\frac{H_{l-2 N_\Gamma}}{\tilde\tau^{(1,2)}_l}
	\begin{vmatrix}
	\overline{ \mathcal J^{ \overline{L_\Gamma}}_{C_{2,l-2 N_\Gamma+1}}-	\mathcal J ^{L_{\Gamma,*}}_{\frac{\phi_{2,l-2 N_\Gamma}}{L_{\mathfrak C, *}}}\Xi^\dagger\mathcal L}&\overline{ \mathcal J^{{L_{\mathfrak C,*}}}_{\phi_{2,l-2 N_\Gamma+1}}}\\
	\vdots &\vdots\\  \overline{
		\mathcal J^{ \overline{L_\Gamma}}_{	C_{2,l+2 N_{\mathfrak C}-1}}-\mathcal J ^{L_{\Gamma,*}}_{\frac{\phi_{2,l+2 N_{\mathfrak C}-1}}{L_{\mathfrak C, *}}}\Xi^\dagger	\mathcal L}&\overline{ \mathcal J^{{L_{\mathfrak C,*}}}_{\phi_{2,l+2 N_{\mathfrak C}-1}}}\\
	\mathcal J^{ \overline{L_\Gamma}}_{K_{C_2}^{[l-2N_\Gamma+1]}}(z)
	-
	\mathcal J ^{\overline{L_{\Gamma,*}}}_{\frac{K^{[l-2N_\Gamma+1]}}{\overline{L_{\mathfrak C, *}}}}\Xi^\top	
	\overline{\mathcal L}
	+
	\frac{1}{L_\Gamma(z)}
	\mathcal J^{ \overline{L_\Gamma}}_{\delta {L_\Gamma}}(z)&\mathcal J^{\overline{L_{\mathfrak C,*}}}_{K^{[l-2N_\Gamma+1]}}(z)
	\end{vmatrix}
	\\
	\tilde  H_{l}&=	 \dfrac{L_{\mathfrak C,(-1)^{l} N_{\mathfrak C}}}{ \overline{L_{\Gamma,(-1)^{l+1}N_\Gamma}}}H_{l-2 N_\Gamma}
	\frac{\tilde\tau^{(2,1)}_{l+1}}{\tilde\tau^{(2,1)}_l}= \dfrac{\overline{L_{\mathfrak C,(-1)^{l} N_{\mathfrak C}}}}{ {L_{\Gamma,(-1)^{l+1}n}}} H_{l-2 N_\Gamma}\frac{\overline{\tilde \tau}^{(1,2)}_{l+1}}{
	\overline{	\tilde \tau}^{(1,2)}_{l}},
\end{align*}
\begin{align*}
	\overline{\tilde \phi_{2,l}(z)}&=-\frac{	\overline{ L_{\Gamma,*}( z)}}{\overline{L_{\Gamma,(-1)^{l+1}N_\Gamma}}}\frac{H_{l-2 N_\Gamma}}{\tilde\tau^{(2,1)}_l}\begin{vmatrix}
	\mathcal J^{ \overline{L_{\Gamma,*}}}_{ 	C_{1,l-2 N_\Gamma+1}}-\mathcal J ^{L_\Gamma}_{\frac{\phi_{1,l-2 N_\Gamma}}{L_{\mathfrak C}}}\Xi\mathcal L_*&\mathcal J^{L_{\mathfrak C}}_{\phi_{1,l-2 N_\Gamma+1}}\\\vdots&\vdots\\\mathcal J^{ \overline{L_{\Gamma,*}}}_{ C_{1,l+2 N_{\mathfrak C}-1}}-	\mathcal J^{{L_{\Gamma}}}_{\frac{\phi_{1,l+2 N_{\mathfrak C}-1}}{L_{\mathfrak C}}}
	\Xi\mathcal L_*&\mathcal J^{L_{\mathfrak C}}_{\phi_{1,l+2 N_{\mathfrak C}-1}}\\
	\mathcal  J^{\overline{L_{\Gamma,*}}}_{K_{C_1}^{[l-2N_\Gamma+1]}}(\bar z)-
		\mathcal J^{{L_{\Gamma}}}_{\frac{K^{[l-2N_\Gamma+1]}}{L_{\mathfrak C}}}
	\Xi\mathcal L_*
	+\frac{1}{\overline{ L_{\Gamma,*}( z)} }\mathcal  J^{\overline{ L_{\Gamma,*}}}_{\delta \overline{ L_{\Gamma,*}}}(\bar z)&
	\mathcal J^{L_{\mathfrak C}}_{K^{[l-2N_\Gamma+1]}}(\bar z)
	\end{vmatrix}\\
		&=
	\frac{	\overline{L_{\mathfrak C,(-1)^{l+1} N_{\mathfrak C}}}}{\overline{L_{\mathfrak C,*}(z)}}\frac{1}{\overline{\tilde \tau}^{(1,2)}_l}\overline{\begin{vmatrix}
	\mathcal J^{\overline{L_\Gamma}}_{ C_{2,l-2 N_\Gamma}}-
	\mathcal J ^{L_{\Gamma,*}}_{\frac{\phi_{2,l-2 N_\Gamma}}{L_{\mathfrak C, *}}}\Xi^\dagger
	\mathcal L & \mathcal J^{{L_{\mathfrak C,*}}}_{ \phi_{2,l-2 N_\Gamma}}& \phi_{2,l-2 N_\Gamma}(z)\\\vdots  &\vdots&\vdots \\
	\mathcal J^{\overline{L_\Gamma}}_{ C_{2,l+2 N_{\mathfrak C}}}-
	\mathcal J ^{L_{\Gamma,*}}_{\frac{\phi_{2,l+2 N_{\mathfrak C}}}{L_{\mathfrak C, *}}}\Xi^\dagger	
	\mathcal L& 	\mathcal J^{{L_{\mathfrak C,*}}}_{ \phi_{2,l+2 N_{\mathfrak C}}}&\phi_{2,l+2 N_{\mathfrak C}}(z)
	\end{vmatrix}}.
	\end{align*}
\end{pro}
Let us further consider  masses  that are restricted to be supported on the diagonal \cite{toledano}. For this discussion we need of
	\begin{align*}
	\mathcal B^{[i]}_{k,j}&:=
	(-1)^k\!\!\!\!\!\!\!\!\!\!\!\!\!\!\!\sum_{\substack{j_1+j_2+\dots+j_{k-j+1}=j\\j_1+2j_2+\dots+(k-j+1)j_{k-j+1}=k}}\!\!\!\!\!\!\!\!\!\!
	\frac{k!}{j_1!\cdots j_{k-j+1}!}(\zeta_i)^{-k-j},\\
	\mathcal B^{[i]}&:=\begin{bsmallmatrix}
	1 & 0 & 0 &0 &\dots &0\\
	0 &\mathcal B^{[i]}_{1,1} & 0&0 &\dots &0\\
	0 &\mathcal B^{[i]}_{2,1} & \mathcal B^{[i]}_{2,2} &0 &\dots &0\\
	0 &\mathcal B^{[i]}_{3,1} & \mathcal B^{[i]}_{3,2} &\mathcal B^{[i]}_{3,3}&\dots &0\\
	\vdots& \vdots&\vdots&\vdots&\ddots&\\
	0 &\mathcal B^{[i]}_{m_i-1,1} & \mathcal B^{[i]}_{m_i-1,2} &\mathcal B^{[i]}_{m_1-1,3}&\dots &\mathcal B^{[i]}_{m_i-1,m_i-1}
	\end{bsmallmatrix},&
	\Xi^i&:=\begin{bsmallmatrix}
	\Xi_0^i &\frac{1}{1!1!}\Xi_1^i &\frac{1}{1!2!}\Xi_2^i&\dots &\frac{1}{1!(m_i-1)!}\Xi_{m_i-1}\\
	\frac{1}{1!1!}\Xi_1^ i&\frac{1}{1!1!}\Xi_2^i & & \iddots &0\\
	\frac{1}{2!1!}\Xi_2^i & & \iddots&\iddots &\vdots\\\vdots&\iddots&\iddots&&\vdots\\
	\frac{1}{(m_i-1)!1!}
	\Xi_{m_i-1}&0 & \dots&\dots &0
	\end{bsmallmatrix},
	\end{align*}
so that, if we choose
$\Xi=\diag(\Xi_1,\dots,\Xi_d)$ with $\Xi_i=\Xi^i\mathcal B^{[i]}$, the Geronimus--Uvarov perturbed sesquilinear form will be
	\begin{align*}
	\prodint{f(z),g(z)}_{\tilde u}=\prodint{\frac{L_{\mathfrak C}(z)}{L_\Gamma (z)}u_z,f(z)\overline{g(z)}}	
	+\prodint{\sum_{i=1}^{d}\sum_{l=0}^{m_i-1}\Xi^i_l\frac{(-1)^l}{l!}\delta^{(l)}(z-\zeta_i),f(z)g_{*}(z)},
	\end{align*}
which is supported on the diagonal but, due to the mass terms,  is not of zero order. The zero order appears when the higher derivatives of the Dirac functionals are cancelled
	\begin{align}\label{circle}
	\prodint{f(z),g(z)}_{\tilde u}=\prodint{\frac{L_{\mathfrak C}(z)}{L_\Gamma (z)}u_z,f(z)\overline{g(z)}}	
	+\prodint{\sum_{i=1}^{d}\Xi^i_0\delta(z-\zeta_i),f(z)g_{*}(z)},
	\end{align}
which is achieved with $
\Xi^i=\Xi_i=\begin{bsmallmatrix}
\xi^i &0&\dots &0\\
0&0&  &0\\
0&0 &\dots &0
\end{bsmallmatrix}\in\mathbb C^{m_i\times m_i}$
and, consequently, we will have the following expression
$
\Xi \mathcal L_*=\Xi^\dagger\mathcal L=\diag\left(\begin{bsmallmatrix}
0  &\dots& 0&\xi^1 \overline{L_{\Gamma,[1]}( \zeta_1)}\\
0 &\dots &0&0\\
\vdots & & \vdots&\vdots\\
0&\dots&0&0
\end{bsmallmatrix} ,\dots, \begin{bsmallmatrix}
0  &\dots& 0&\xi^d \overline{L_{\Gamma,[d]}( \zeta_d)}\\
0 &\dots &0&0\\
\vdots & & \vdots&\vdots\\
0&\dots&0&0
\end{bsmallmatrix} \right)=:C$,
so that
\begin{align*}
\mathcal J^L_{\phi} \Xi\mathcal L_*=\Big[0,\dots,0, \phi(\zeta_1)\xi^1 \overline{L_{\Gamma,[1]}( \zeta_1)},\dots, 0,\dots,0, \phi(\zeta_d)\xi^d\overline{ L_{\Gamma,[d]}( \zeta_d)}\Big].
\end{align*}
As $L_{\mathfrak C}(z)$ and $L_\Gamma(z)$ are coprime the reality of the first term in the RHS of \eqref{circle} is ensured whenever $L_{\mathfrak C}(z)$ and  $L_{\Gamma}(z)$ are sel-reciprocal polynomials, $L_{\mathfrak C,*}(z)=L_{\mathfrak C}(z)$ and $L_{\Gamma,*}(z)=L_{\Gamma}(z)$.
The mass term will be real whenever $\xi^i\in\mathbb R$, $i\in\{1,\dots,d\}$. If we are interested in the non-negative situation, then we must further impose that the self-reciprocal Laurent polynomials $L_{\mathfrak C}(z)$ and $L_\Gamma(z)$ are such that $\frac{L_{\mathfrak C}(z)}{L_\Gamma(z)}=\frac{|L_{\mathfrak C}(z)|}{|L_\Gamma(z)|}$ and also that $\xi^i\geq 0$, $i\in\{1\dots,d\}$.

Self-reciprocality for the Laurent polynomial $L=L_{2n}z^{-n}(z-\zeta_1)^{m_1}\cdots (z-\zeta_d)^{m_d}$, $m_1+\dots +m_d=2n$,
requires
$\overline{L_{2n}}(\bar\zeta_1)^{m_1}\dots(\bar \zeta_d)^{m_d}(z-(\bar \zeta_1)^{-1})^{m_1}\cdots (z-(\bar\zeta_d)^{-1})^{m_d}=L_{2n}(z-\zeta_i)^{m_1}\cdots (z-\zeta_d)^{m_d}$,
that could be fulfilled  if only if
\begin{align*}
L(z)&=L_{2n}z^{-n}(z-\alpha_1)^{n_1}(z-(\bar\alpha_1)^{-1})^{n_1}\cdots (z-\alpha_r)^{n_r}(z-(\bar\alpha_r)^{-1})^{n_r} (z-\beta_1)^{2q_1}\cdots(z-\beta_s)^{2q_s},
\end{align*}
where $ n_1+\dots +n_r+q_1+\cdots+q_s=n$ and $\alpha_i\not\in \mathbb T$, $i\in\{1,\dots,r\}$ and $\beta_i\in\mathbb T$, $i\in\{1,\dots,s\}$, with
\begin{align*}
\operatorname{arg} L_{2n}= -\arg\alpha_1-\dots-\arg\alpha_r-\arg\beta_1-\dots-\arg\beta_s.
\end{align*}

Therefore, for this real Toeplitz zero order diagonal case, and using
\begin{align*}
\tilde \tau_l:=\begin{vmatrix}
	\mathcal J^{\overline{L_{\Gamma}}}_{ C_{l-2 N_\Gamma}}-
	\mathcal J ^{L_\Gamma}_{\frac{\phi_{l-2 N_\Gamma}}{L_{\mathfrak C}}}C
	& \mathcal J^{{L_{\mathfrak C}}}_{ \phi_{l-2 N_\Gamma}}\\\vdots  &\vdots \\
	\mathcal J^{\overline{L_{\Gamma}}}_{ C_{l+2 N_{\mathfrak C}-1}}-
	                                                              	\mathcal J^{{L_{\Gamma}}}_{\frac{\phi_{1,l+2 N_{\mathfrak C}-1}}{L_{\mathfrak C}}}
C&
		\mathcal J^{{L_{\mathfrak C}}}_{ \phi_{l+2 N_{\mathfrak C}-1}}
	\end{vmatrix},
 \end{align*}
we conclude that the perturbed orthogonal polynomials and its norms are
\begin{align*}
	\tilde\phi_{l}(z)
	&=
	\frac{	L_{\mathfrak C,(-1)^{l} N_{\mathfrak C}}}{L_{\mathfrak C}(z)}\frac{1}{\tilde\tau_l}\begin{vmatrix}
	\mathcal J^{\overline{L_{\Gamma}}}_{ C_{l-2 N_\Gamma}}-
	\mathcal J ^{L_\Gamma}_{\frac{\phi_{l-2 N_\Gamma}}{L_{\mathfrak C}}}C
	& \mathcal J^{{L_{\mathfrak C}}}_{ \phi_{l-2 N_\Gamma}}& \phi_{l-2 N_\Gamma}(z)\\\vdots  &\vdots&\vdots \\
	\mathcal J^{\overline{L_{\Gamma}}}_{ C_{l+2 N_{\mathfrak C}}}-
	                                                              	\mathcal J^{{L_{\Gamma}}}_{\frac{\phi_{l+2 N_{\mathfrak C}}}{L_{\mathfrak C}}}
C&
		\mathcal J^{{L_{\mathfrak C}}}_{ \phi_{l+2 N_{\mathfrak C}}}&\phi_{l+2 N_{\mathfrak C}}(z)
	\end{vmatrix}\\&=-\frac{{L_\Gamma}(z)	}{L_{\Gamma,(-1)^lN_\Gamma}}\frac{H_{l-2 N_\Gamma}}{\tilde\tau_l}
	\begin{vmatrix}
	\overline{ \mathcal J^{ \overline{L_\Gamma}}_{C_{l-2 N_\Gamma+1}}-	\mathcal J ^{L_{\Gamma}}_{\frac{\phi_{l-2 N_\Gamma}}{L_{\mathfrak C}}}C}&\overline{ \mathcal J^{{L_{\mathfrak C}}}_{\phi_{l-2 N_\Gamma+1}}}\\
	\vdots &\vdots\\  \overline{
		\mathcal J^{ \overline{L_\Gamma}}_{	C_{l+2 N_{\mathfrak C}-1}}-\mathcal J ^{L_{\Gamma}}_{\frac{\phi_{l+2 N_{\mathfrak C}-1}}{L_{\mathfrak C}}}C}&\overline{ \mathcal J^{{L_{\mathfrak C}}}_{\phi_{l+2 N_{\mathfrak C}-1}}}\\
	\mathcal J^{ \overline{L_\Gamma}}_{K_{C}^{[l-2N_\Gamma+1]}}(z)
	-
	\mathcal J ^{\overline{L_{\Gamma}}}_{\frac{K^{[l-2N_\Gamma+1]}}{\overline{L_{\mathfrak C}}}} \bar C
	+
	\frac{1}{L_\Gamma(z)}
	\mathcal J^{ \overline{L_\Gamma}}_{\delta {L_\Gamma}}(z)&\mathcal J^{\overline{L_{\mathfrak C}}}_{K^{[l-2N_\Gamma+1]}}(z)
	\end{vmatrix},
	\\
	\tilde  H_{l}&=	 \dfrac{L_{\mathfrak C,(-1)^{l} N_{\mathfrak C}}}{ \overline{L_{\Gamma,(-1)^{l+1}N_\Gamma}}}H_{l-2 N_\Gamma}
	\frac{\tilde\tau_{l+1}}{\tilde\tau_l}.
\end{align*}

\appendix
\section{Proofs}
\begin{proof}[Proof of Proposition \ref{Geronimus Connection Cauchy}]
From definition we have
\begin{align*}
\tilde C_{1}^{(1,2)}(z)&=\prodint{\tilde \phi_1^{(1,2)}(z_1), \frac{1}{\bar z-z_2}}_{\tilde u^{(1,2)}},&z&\not\in \overline{\operatorname{supp}_2(\tilde u^{(1,2)}),} \\
\tilde C_{1}^{(2,1)}(z)&=\prodint{\tilde \phi_1^{(2,1)}(z_1), \frac{1}{\bar z-z_2}}_{\tilde u^{(2,1)}}
&z&\not\in \overline{\operatorname{supp}_2(\tilde u^{(2,1)}),}\\
(	\tilde C_{2}^{(1,2)}(z))^\dagger&=\prodint{\frac{1}{\bar z-z_1}, (\tilde \phi^{(1,2)}_2(z_2))^\top}_{\tilde u^{(1.2)}},& z&\not\in \overline{\operatorname{supp}_1(\tilde u^{(1,2)})},\\
(	\tilde C_{2}^{(2,1)}(z))^\dagger&=\prodint{\frac{1}{\bar z-z_1}, (\tilde \phi^{(2,1)}_2(z_2))^\top}_{\tilde u^{(2,1)}},
& z&\not\in \overline{\operatorname{supp}_1(\tilde u^{(2,1)})}.
\end{align*}
Then, \eqref{CC12}  is proven as follows
\begin{align*}
(C_2(z))^\dagger \big(\Omega^{(1,2)}_2\big)^\dagger - L_\Gamma^{(1)}(\bar z)(\tilde C^{(1,2)}_2(z))^\dagger &=
\prodint{\frac{1}{\bar z-z_1},\big(\phi_2(z_2)\big)^\top}_u\big(\Omega^{(1,2)}_2\big)^\dagger -L_\Gamma^{(1)}(\bar z)
\prodint{\frac{1}{\bar z-z_1},\big(\tilde \phi^{(1,2)}_2(z_2)\big)^\top}_{\tilde u^{(1,2)}}\\&=
\prodint{\frac{1}{\bar z-z_1},\big(\tilde\phi^{(1,2)}_2(z_2)\big)^\top L_{\mathfrak C}^{(2)}(z_2)}_{ u}-
\prodint{\frac{L_\Gamma^{(1)}(\bar z)}{\bar z-z_1},\big(\tilde \phi^{(1,2)}_2(z_2)\big)^\top}_{\tilde u^{(1,2)}}\\
&=\prodint{\frac{L_\Gamma^{(1)}(z_1)}{\bar z-z_1},\big(\tilde\phi^{(1,2)}_2(z_2)\big)^\top )}_{ \tilde u^{(1,2)}}-
\prodint{\frac{L_\Gamma^{(1)}(\bar z)}{\bar z-z_1},\big(\tilde \phi^{(1,2)}_2(z_2)\big)^\top}_{\tilde u^{(1,2)}}\\
&=-\prodint{\frac{L_\Gamma^{(1)}(\bar z)-L_\Gamma^{(1)}(z_1)}{\bar z-z_1},\big(\tilde\phi^{(1,2)}_2(z_2)\big)^\top}_{\tilde u^{(1,2)}}.
\end{align*}
For \eqref{CC21} we have
\begin{align*}
\Omega^{(2,1)}_1 C_{1}(z)-\tilde C_{1}^{(2,1)}(z) { \overline{L_\Gamma^{(2)}}(z)}&=\Omega^{(2,1)}_1\prodint{ \phi_1(z_1), \frac{1}{\bar z-z_2}}_{ u}-\prodint{\tilde \phi_1^{(2,1)}(z_1), \frac{1}{\bar z-z_2}}_{\tilde u^{(2,1)}}{ \overline{L_\Gamma^{(2)}}(z)}\\
&=\prodint{ L_{\mathfrak C}^{(1)}(z_1)\tilde\phi^{(2,1)}_1(z_1), \frac{1}{\bar z-z_2}}_{ u}
-\prodint{\tilde \phi_1^{(2,1)}(z_1),
	\frac{	L_\Gamma^{(2)}( \bar z)}{\bar z-z_2}}_{\tilde u^{(2,1)}}\\
&=\prodint{\tilde\phi^{(2,1)}_1(z_1), \frac{L_\Gamma^{(2)}( z_2)}{\bar z-z_2}}_{ \tilde u^{(2,1)}}
-\prodint{\tilde \phi_1^{(2,1)}(z_1),
	\frac{	L_\Gamma^{(2)}( \bar z)}{\bar z-z_2}}_{\tilde u^{(2,1)}}\\
&=-\prodint{ \tilde\phi^{(2,1)}_1(z_1), \frac{L_\Gamma^{(2)}( \bar z)-L_\Gamma^{(2)}(z_2)}{\bar z-z_2}}_{ \tilde u^{(2,1)}}.
\end{align*}
The formula   \eqref{CC11} follows from
$\Omega^{(1,2)}_1 \tilde C^{(1,2)}_1(z)=\prodint{L_\Gamma^{(1)}(z_1)\phi_1(z_1),\frac{1}{\bar z-z_2}}_{\tilde u^{(1,2)}}
=\prodint{\phi_1(z_1),\frac{L_{\mathfrak C}^{(2)}(z_2)}{\bar z-z_2}}_{ u}$.
For  \eqref{CC22} we notice
\begin{align*}
\big(\tilde C^{(2,1)}_2(z)\big)^\dagger\big(\Omega_2^{(2,1)}\big)^\dagger&=
\prodint{\frac{1}{\bar z-z_1}, \big(\tilde\phi_2(z_2)\big)^\top}_{\tilde u^{(2,1)}}\big(\Omega_2^{(2,1)}\big)^\dagger
\\&=\prodint{\frac{1}{\bar z-z_1}, L_\Gamma^{(2)}(z_2)\big(\phi_2(z_2)\big)^\top}_{\tilde u^{(2,1)}}
\\&=\prodint{\frac{L_{\mathfrak C}^{(1)}(z_1)}{\bar z-z_1}, \big(\phi_2(z_2)\big)^\top}_{u}.
\end{align*}
\end{proof}

\begin{proof}[Proof of Proposition \ref{Geronimus Connection CD}]
	The relations
	\begin{align*}
	(\phi_2(z_1))^\dagger(\Omega^{(1,2)}_2 )^\dagger\big(\tilde H^{(1,2)}\big)^{-1} &=\overline{L_{\mathfrak C}^{(2)}(z_1)}(\tilde\phi^{(1,2)}_2(z_1))^\dagger\big(\tilde H^{(1,2)}\big)^{-1} ,&
	\big(\Omega^{(1,2)}_2\big)^\dagger\big(\tilde H^{(1,2)}\big)^{-1} \tilde\phi_1^{(1)}(z_2) = L_\Gamma^{(1)}(z_2)H^{-1}\phi_1(z_2),
	\end{align*}
imply
	\begin{align*}
	&\begin{multlined}[t][0.9\textwidth]
	\left[(\phi_2(z_1))^\dagger\right]^{[l]}\left[	\big(\Omega^{(1,2)}_2\big)^\dagger\big(\tilde  H^{(1,2)}\big)^{-1}\right] ^{[l]}\left[\tilde\phi_1^{(1,2)}(z_2) \right]^{[l]}+\left[(\phi_2(z_1))^\dagger\right]^{[\geq l]}
	\left[	\big(\Omega^{(1,2)}_2\big)^\dagger\big(\tilde  H^{(1,2)}\big)^{-1}\right] ^{[\geq l, l]}\left[\tilde\phi_1^{(1)}(z_2) \right]^{[l]}\\= \overline{ L_{\mathfrak C}^{(2)}( z_1)}\left[(\tilde \phi^{(1,2)}_2(z_1))^\dagger\right]^{[l]}\left[ (\tilde  H^{(1,2)})^{-1}\right]^{[l]}\left[\tilde \phi^{(1,2)}_1(z_2) \right]^{[l]},
	\end{multlined}\\
	&\begin{multlined}[t][0.9\textwidth]\left[(\phi_2(z_1)	 )^\dagger\right]^{[l]}\left[(\Omega^{(1,2)}_2 )^\dagger\big(\tilde  H^{(1,2)}\big)^{-1} \right]^{[l]}\left[\tilde \phi^{(1,2)}_1(z_2)\right]^{[l]}+\left[(\phi_2(z_1)	 )^\dagger\right]^{[l]}\left[(\Omega^{(1,2)}_2 )^\dagger\big(\tilde  H^{(1,2)}\big)^{-1} \right]^{[l,\geq l]}\left[\tilde \phi^{(1,2)}_1(z_2)\right]^{[\geq l]}\\  =L_\Gamma^{(1)}(z_2)\left[(\phi_2(z_1))^\dagger\right]^{[l]}\left[ H^{-1}\right]^{[l]}\left[\phi_1(z_2)\right]^{[l]},
	\end{multlined}
	\end{align*}
so that
	\begin{align*}
	&\begin{multlined}[t][0.9\textwidth]
	\left[(\tilde\phi_1^{(1)}(z_2) )^\top\right]^{[l]}
	\left[\big({\tilde H}^{(1,2)}\big)^{-1}	\bar \Omega^{(1,2)}_2\right] ^{[l,\geq l]}
	\left[ \overline{\phi_2(z_1)}\right]^{[\geq l]}-	\left[(\tilde\phi_1^{(1,2)}(z_2) )^\top\right]^{[\geq l]}
	\left[\big({\tilde H}^{(1,2)}\big)^{-1}	\bar \Omega^{(1,2)}_2\right] ^{[\geq l,l]}
	\left[ \overline{\phi_2(z_1)}\right]^{[ l]}\\
	= \overline{L_{\mathfrak C}^{(2)}( z_1)}{\tilde K^{(1,2),[l]}(\bar z_1,z_2)} -{L_\Gamma^{(1)}(z_2)K^{[l]}(\bar z_1,z_2)},
	\end{multlined}
	\end{align*}
and \eqref{GerKerNor1} follows.
	On the other hand, we have
	\begin{align*}
	(	\tilde H^{(2,1)} )^{-1}	\Omega_1^{(2,1)} \phi_1(z_2)&=L_{\mathfrak C}^{(1)}(z_2)(	\tilde H^{(2,1)} )^{-1}\tilde\phi^{(2,1)}_1(z_2), & (\tilde\phi^{(2,1)}_2(z_1))^\dagger(	\tilde H^{(2,1)} )^{-1}\Omega^{(2,1)}_1 &=\overline{L_\Gamma^{(2)}(z_1)} (\phi_2(z_1))^\dagger H^{-1}.
	\end{align*}
Therefore,
	\begin{align*}
	&\begin{multlined}[t][0.9\textwidth]\left[(\tilde{\phi}_2^{(2,1)}(z_1))^\dagger\right]^{[l]}\left[(	\tilde H^{(2,1)} )^{-1}	\Omega_1^{(2,1)} \right]^{[l]}\left[\phi_1(z_2)\right]^{[l]}+\left[(\tilde{\phi}_2^{(2,1)}(z_1))^\dagger\right]^{[l]}\left[(	\tilde H^{(2,1)} )^{-1}	\Omega_1^{(2,1)} \right]^{[l,\geq l]}\left[\phi_1(z_2)\right]^{[\geq l]}\\
	=L_{\mathfrak C}^{(1)}(z_2)\left[(\tilde{\phi}_2^{(2,1)}(z_1))^\dagger\right]^{[l]}\left[(	\tilde H^{(2,1)} )^{-1}\right]^{[l]}\left[\tilde\phi^{(2,1)}_1(z_2)\right]^{[l]},
	\end{multlined}\\&\begin{multlined}[t][0.9\textwidth]
	\left[(\tilde\phi^{(2,1)}_2(z_1))^\dagger\right]^{[l]}\left[(	\tilde  H^{(2,1)} )^{-1}\Omega^{(2,1)}_1 \right]^{[l]}\left[\phi_1(z_2)\right]^{[l]}+\left[(\tilde\phi^{(2,1)}_2(z_1))^\dagger\right]^{[\geq l]}\left[(	\tilde H^{(2,1)} )^{-1}\Omega^{(2,1)}_1 \right]^{[\geq l,l]}\left[\phi_1(z_2)\right]^{[l]}\\
	=\overline{L_\Gamma^{(2)}(z_1)} \left[(\phi_2(z_1))^\dagger\right] ^{[l]}\left[H^{-1}\right]^{[l]}\left[\phi_1(z_2)\right]^{[l]},
	\end{multlined}
	\end{align*}
from where we deduce that
	\begin{align*}
	&\begin{multlined}[t][0.9\textwidth]
	\left[(\tilde\phi^{(2,1)}_2(z_1))^\dagger\right]^{[ l]}\left[(	\tilde H^{(2,1)} )^{-1}\Omega^{(2,1)}_1 \right]^{[ l,\geq l]}\left[\phi_1(z_2)\right]^{[\geq l]}-	\left[(\tilde\phi^{(2,1)}_2(z_1))^\dagger\right]^{[\geq l]}\left[(	\tilde H^{(2,1)} )^{-1}\Omega^{(2,1)}_1 \right]^{[\geq l,l]}\left[\phi_1(z_2)\right]^{[l]}\\
	=L_{\mathfrak C}^{(1)}(z_2)\tilde K^{(2,1),[l]}(\bar z_1,z_2)-\overline{L_\Gamma^{(2)}(z_1)} K^{[l]}(\bar z_1,z_2),
	\end{multlined}
	\end{align*}
	and \eqref{GerKerNor2} follows.
\end{proof}
\begin{proof}[Proof of Proposition \ref{Geronimus Connection mixed}]
The following equations
	\begin{align*}
	(C_2(x_1))^\dagger \big(\Omega^{(1,2)}_2\big)^\dagger \big(\tilde  H^{(1,2)}\big)^{-1}- L_\Gamma^{(1)}(\bar x_1)(\tilde C^{(1,2)}_2(x_1))^\dagger \big(\tilde  H^{(1,2)}\big)^{-1}&=-\prodint{\delta L_\Gamma^{(1)}(\bar x_1,z_1),\big(\tilde\phi^{(1,2)}_2( z_2)\big)^\top}_{\tilde u^{(1)}}\big(\tilde H^{(1,2)}\big)^{-1},
	\end{align*}
	\begin{align*}
	\big(\Omega^{(1,2)}_2\big)^\dagger(\tilde  H^{(1,2)})^{-1} \tilde\phi_1^{(1,2)}(x_2) &= L_\Gamma^{(1)}(x_2)H^{-1}\phi_1(x_2),
	\end{align*}
imply
	\begin{align*}
	&\begin{multlined}[t][0.9\textwidth]
	\left[	(C_2(x_1))^\dagger \right]^{[l]}\left[\big(\Omega^{(1,2)}_2\big)^\dagger\big(\tilde  H^{(1,2)}\big)^{-1} \right]^{[l]}	[\tilde \phi^{(1,2)}_{1}(x_2)]^{[l]}+	\left[	(C_2(x_1))^\dagger \right]^{[\geq l]}\left[\big(\Omega^{(1,2)}_2\big)^\dagger\big(\tilde  H^{(1,2)}\big)^{-1} \right]^{[\geq l, l]}	[\tilde \phi^{(1,2)}_{1}(x_2)]^{[l]}\\
	-L_\Gamma^{(1)}(\bar x_1)\left[(\tilde C^{(1,2)}_2(x_1))^\dagger\right]^{[l]}\left[\big(\tilde H^{(1,2)}\big)^{-1}\right]^{[l]}	[\tilde \phi^{(1,2)}_{1}(x_2)]^{[l]}\\=
	-\prodint{\delta L_\Gamma^{(1)}(\bar x_1,z_1),\big(\left[\tilde\phi^{(1,2)}_2( z_2)\right]^{[l]}\big)^\top}_{\tilde u^{(1,2)}}\left[\big(\tilde H^{(1,2)}\big)^{-1}\right]^{[l]}	[\tilde \phi^{(1,2)}_{1}(x_2)]^{[l]},
	\end{multlined}\\
	&\begin{multlined}[t][0.9\textwidth]
	[C_{2}(x_1)^{\dagger}]^{[l]}\left[	\big(\Omega^{(1,2)}_2\big)^\dagger(\tilde H^{(1,2)})^{-1} \right]^{[l]}\left[\tilde\phi_1^{(1,2)}(x_2) \right]^{[l]}+[C_{2}(x_1)^{\dagger}]^{[l]}
	\left[	\big(\Omega^{(1,2)}_2\big)^\dagger(\tilde H^{(1,2)})^{-1} \right]^{[l,\geq l]}\left[\tilde\phi_1^{(1,2)}(x_2)\right] ^{[\geq l]}\\
	= L_\Gamma^{(1)}(x_2)[C_{2}(x_1)^{\dagger}]^{[l]}\left[H^{-1}\right]^{[l]}\left[\phi_1(x_2)\right]^{[l]},
	\end{multlined}
	\end{align*}
	and, consequently,
	\begin{align*}
	\begin{multlined}[t][0.9\textwidth]
	\left[(\tilde\phi_1^{(1,2)}(x_2))^\top\right] ^{[\geq l]}
	\left[	({\tilde H}^{(1,2)})^{-1}\big(\bar \Omega^{(1,2)}_2\big) \right]^{[\geq l,l]}[\overline{C_{2}(x_1)}]^{[l]}+	L_\Gamma^{(1)}(\bar x_1)\tilde K_{C_2}^{(1,2),[l]}(\bar x_1,x_2)\\= \left[(\tilde\phi_1^{(1,2)}(x_2))^\top\right] ^{[l]}
	\left[	({\tilde H}^{(1,2)})^{-1}\big(\bar \Omega^{(1,2)}_2\big) \right]^{[l,\geq l]}[\overline{C_{2}(x_1)}]^{[\geq l]}+
	\prodint{\delta L_\Gamma^{(1)}(\bar x_1,z_1), \overline{\tilde K^{(1,2),[l]}(\bar z_2,x_2)}}_{\tilde u^{(1)}}
	\\+ L_\Gamma^{(1)}(x_2)K_{C_2}^{[l]}(\bar x_1,x_2).
	\end{multlined}
	\end{align*}
	and \eqref{mix CF1} follows.
	Now, to check \eqref{mix CF2}, observe that
	\begin{align*}
	\big(\tilde C^{(2,1)}_2(x_1)\big)^\dagger (\tilde H^{(2,1)} )^{-1}\Omega^{(2,1)}_1 &=\prodint{\frac{L_{\mathfrak C}^{(1)}(z_1)}{\bar x_1-z_1},\big(\phi_2(z_2)\big)^\top}_{u} H^{-1},\\
	(\tilde H^{(2,1)})^{-1}\Omega_1^{(2,1)} \phi_1(x_2)&=L_{\mathfrak C}^{(1)}(x_2)(\tilde H^{(2,1)})^{-1}\tilde\phi^{(2,1)}_1(x_2).
	\end{align*}
Thus,
	\begin{align*}
	&\begin{multlined}[t][0.9\textwidth]\left[	\big(\tilde C^{(2,1)}_2(x_1)\big)^\dagger\right] ^{[l]}\left[(\tilde H^{(2,1)} )^{-1}\Omega^{(2,1)}_1\right]^{[l]}\left[\phi_1(x_2)\right]^{[l]}+
	\left[	\big(\tilde C^{(2,1)}_2(x_1)\big)^\dagger\right] ^{[\geq l]}	\left[(\tilde H^{(2,1)})^{-1}\Omega_1^{(2,1)}\right]^{[\geq l,l]} \left[\phi_1(x_2)\right]^{[ l]}\\=\prodint{\frac{L_{\mathfrak C}^{(1)}(z_1)}{\bar x_1-z_1},\big[\big(\phi_2(z_2)\big)^\top\big]^{[l]} \left[H^{-1}\right]^{[l]}}_{u}\left[\phi_1(x_2)\right]^{[l]},
	\end{multlined}\\
	&\begin{multlined}[t][0.9\textwidth]
	\left[	\big(\tilde C^{(2,1)}_2(x_1)\big)^\dagger\right] ^{[l]}\left[	(\tilde H^{(2,1)})^{-1}\Omega_1^{(2,1)}\right]^{[l]} \left[\phi_1(x_2)\right]^{[l]}+\left[	\big(\tilde C^{(2,1)}_2(x_1)\big)^\dagger\right] ^{[l]}\left[	(\tilde H^{(2,1)})^{-1}\Omega_1^{(2,1)}\right]^{[l,\geq l]} \left[\phi_1(x_2)\right]^{[\geq l]}\\
	=L_{\mathfrak C}^{(1)}(x_2)\left[	\big(\tilde C^{(2,1)}_2(x_1)\big)^\dagger\right] ^{[l]}\left[(\tilde H^{(2,1)})^{-1}\right]^{[l]}\left[\tilde\phi^{(2,1)}_1(x_2)\right]^{[l]},
	\end{multlined}
	\end{align*}
	and we deduce
	\begin{align*}
	\begin{multlined}[t][0.9\textwidth]
	\prodint{\frac{L_{\mathfrak C}^{(1)} (z_1)}{\bar x_1-z_1},\overline{K^{[l]}(\bar z_2,x_2)}}_{u}-
	\left[	\big(\tilde C^{(2,1)}_2(x_1)\big)^\dagger\right] ^{[\geq l]}	\left[(\tilde H^{(2,1)})^{-1}\Omega_1^{(2)}\right]^{[\geq l, l]} \left[\phi_1(x_2)\right]^{[ l]}=L_{\mathfrak C}^{(1)} (x_2)\tilde K_{C_2}^{(2),[l]}(\bar x_1,x_2)\\
	-\left[	\big(\tilde C^{(2,1)}_2(x_1)\big)^\dagger\right] ^{[l]}	\left[(\tilde  H^{(2,1)})^{-1}\Omega_1^{(2,1)}\right]^{[l,\geq l]} \left[\phi_1(x_2)\right]^{[ \geq l]}.
	\end{multlined}
	\end{align*}
	
	To get \eqref{mix FC1} we notice that
	\begin{align*}
	(\phi_2(x_1))^\dagger
\big(\Omega^{(1,2)}_2 \big)^\dagger({\tilde H}^{(1,2)})^{-1}&=\overline{L_{\mathfrak C}^{(2)}(x_1)}\big(\tilde\phi^{(1,2)}_2(x_1)\big)^\dagger({\tilde H}^{(1,2)})^{-1},\\
	\big(\Omega^{(1,2)}_2\big)^\dagger(\tilde H^{(1,2)})^{-1}\tilde C^{(1,2)}_1(x_2)&=
	H^{-1}\prodint{\phi_1(z_1),\frac{L_{\mathfrak C}^{(2)}(z_2)}{\bar x_2-z_2}}_{u},
	\end{align*}
	and, therefore,
	\begin{align*}
	&\begin{multlined}[t][0.9\textwidth]
	\left[(\phi_2(x_1))^\dagger\right]^{[l]}   \left[\big(\Omega^{(1,2)}_2\big)^\dagger(\tilde  H^{(1,2)})^{-1}\right]^{[l]}\left[\tilde C^{(1,2)}_1(x_2)\right]^{[l]}+
	\left[(\phi_2(x_1))^\dagger\right]^{[\geq l]}    \left[\big(\Omega^{(1,2)}_2\big)^\dagger(\tilde  H^{(1,2)})^{-1}\right]^{[\geq l, l]}\left[\tilde C^{(1,2)}_1(z)\right]^{[l]}\\=
	\overline{L_{\mathfrak C}^{(2)}(x_1)}\left[(\tilde \phi^{(1,2)}_2(x_1))^\dagger\right]^{[l]}   \left[(\tilde H^{(1,2)})^{-1}\right]^{[l]}\left[\tilde C^{(1,2)}_1(x_2)\right]^{[l]}
	\end{multlined},\\
	&\begin{multlined}[t][0.9\textwidth]
	\left[(\phi_2(x_1))^\dagger\right]^{[l]}\left[ \big(\Omega^{(1,2)}_2 \big)^\dagger({\tilde H}^{(1)})^{-1}\right]^{[l]}\left[\tilde C^{(1,2)}_{1}(x_2)\right]^{[l]}+	\left[(\phi_2(x_1))^\dagger\right]^{[l]}\left[ \big(\Omega^{(1,2)}_2 \big)^\dagger({\tilde H}^{(1,2)})^{-1}\right]^{[l,\geq l]}\left[\tilde C^{(1,2)}_{1}(x_2)\right]^{[\geq l]}\\
	=\left[\big(\phi_2(x_1)\big)^\dagger\right]^{[l]}\prodint{\left[(H)^{-1}\right]^{[l]}\big[\phi_1(z_1)\big]^{[l]},\frac{L_{\mathfrak C}^{(2)}(z_2)}{\bar x_2-z_2}}_{u},
	\end{multlined}
	\end{align*}
so that
	\begin{align*}
	&\begin{multlined}[t][0.9\textwidth]
	\left[(\tilde C^{(1,2)}_1(x_2))^\top\right]^{[ l]}
	\left[(\tilde H^{(1,2)})^{-1}\bar\Omega^{(1,2)}_2\right]^{[ l,\geq l]}\left[\overline{(\phi_2(x_1))}\right]^{[\geq l]}  -\left[(\tilde C^{(1,2)}_1(x_2))^\top\right]^{[\geq l]}
	\left[(\tilde H^{(1,2)})^{-1}\bar\Omega^{(1,2)}_2\right]^{[\geq l,l]}\left[\overline{(\phi_2(x_1))}\right]^{[l]} \\
	  =	\overline{L_{\mathfrak C}^{(2)}(x_1)}\tilde K_{C_1}^{(1,2),[l]}(\bar x_1,x_2)-\prodint{K^{[l]}(\bar x_1,z_1),\frac{L_{\mathfrak C}^{(2)}(z_2)}{\bar x_2-z_2}}_{u}.
	\end{multlined}
	\end{align*}
	
	Finally,  we prove \eqref{mix FC2}. For that aim   consider
	\begin{align*}
	(\tilde H^{(2,1)})^{-1}	\Omega^{(2,1)}_1 C_{1}(x_2)- \overline{L_\Gamma^{(2)}}(x_2)(\tilde H^{(2,1)})^{-1}	\tilde C_{1}^{(2,1)}(x_2)&=-(\tilde H^{(2,1)})^{-1}	\prodint{ \tilde\phi^{(2,1)}_1(z_1), \delta L^{2}_\Gamma(\bar x_2,z_2)}_{ \tilde u^{(2,1)} },\\
	(\tilde\phi^{(2,1)}_2(x_1))	^\dagger (\tilde H^{(2,1)})^{-1}\Omega^{(2,1)}_1&=\overline{L_\Gamma^{(2)}(x_1)} (\phi_2(x_1))^\dagger H^{-1},\\
	\end{align*}
and as a consequence we find
	\begin{align*}
	&\begin{multlined}[t][0.9\textwidth]
	\left[(\tilde\phi^{(2,1)}_2(x_1))^\dagger\right]^{[l]}\left[(\tilde H^{(2,1)})^{-1}\Omega^{(2,1)}_1 \right]^{[l]}\left[C_{1}(x_2)\right]^{[l]}+\left[(\tilde\phi^{(2,1)}_2(x_1))^\dagger\right]^{[l]}\left[(\tilde H^{(2,1)})^{-1}\Omega^{(2,1)}_1 \right]^{[l,\geq l]}\left[C_{1}(x_2)\right]^{[\geq l]}\\
	-\overline{L_\Gamma^{(2)}}(x_2)\left[(\tilde\phi^{(2,1)}_2(x_1))^\dagger\right]^{[l]}\left[(\tilde H^{(2,1)})^{-1}\right]^{[l]}\left[\tilde C_{1}^{(2)}(x_2)\right]^{[l]}\\=-\left[(\tilde\phi^{(2,1)}_2(x_1))^\dagger\right]^{[l]}\prodint{\left[(\tilde H^{(2,1)})^{-1}	\right]^{[l]} \left[\tilde\phi^{(2,1)}_1(z_1)\right]^{[l]},\delta L_\Gamma^{(2)}(\bar x_2,z_2)}_{ \tilde u^{(2,1)} },
	\end{multlined}\\
	&\begin{multlined}[t][0.9\textwidth]
	\left[(\tilde\phi^{(2,1)}_2(x_1))	^\dagger \right]^{[l]}\left[(\tilde H^{(2,1)})^{-1}\Omega^{(2,1)}_1\right]^{[l]}\left[C_{1}(x_2)\right]^{[l]}+\left[(\tilde\phi^{(2,1)}_2(x_1))	^\dagger \right]^{[\geq l]}\left[(\tilde H^{(2,1)})^{-1}\Omega^{(2,1)}_1\right]^{[\geq l,l]}\left[C_{1}(x_2)\right]^{[l]}
	\\=\overline{L_\Gamma^{(2)}(x_1)}\left[ (\phi_2(x_1))^\dagger\right]\left[ H^{-1}\right]^{[l]}\left[C_{1}(x_2)\right]^{[l]},
	\end{multlined}
	\end{align*}
	that gives
	\begin{multline*}
	\left[(\tilde\phi^{(2,1)}_2(x_1))^\dagger \right]^{[\geq l]}\left[(\tilde H^{(2,1)})^{-1}\Omega^{(2,1)}_1\right]^{[\geq l,l]}\left[C_{1}(x_2)\right]^{[l]}-\left[(\tilde\phi^{(2,1)}_2(x_1))^\dagger \right]^{[l]}\left[(\tilde H^{(2,1)})^{-1}\Omega^{(2,1)}_1\right]^{[l,\geq l]}\left[C_{1}(x_2)\right]^{[\geq l]}\\
	=\overline{L_\Gamma^{(2)}(x_1)} K_{C_1}^{[l]}(\bar x_1,x_2)- \overline{L_\Gamma^{(2)}}(x_2)\tilde K_{C_1}^{(2,1),[l]}(\bar x_1,x_2)
	+\prodint{ \tilde K^{(2,1),[l]}(\bar x_1,z_1), \delta L_\Gamma^{(2)}(\bar x_2,z_2)}_{\tilde u^{(2,1)}}.
	\end{multline*}
\end{proof}

\begin{proof}[Proof of Proposition \ref{pro.chorros}]
First, let us show that, for  $l=0,\dots,m^{(2)}_{j}-1$ we have
	\begin{align*}
	\frac{\d^l}{\d z^l}\bigg|_{z=\bar \zeta^{(2)}_{j}}
	\prodint{\tilde\phi^{(2,1)}_{1,k}(z_1),\frac{L_\Gamma^{(2)}(\bar z)}{\bar z-z_2}}_{\frac{L_{\mathfrak C}^{(1)}(z_1)}{\overline{(L_\Gamma^{(2)}(z_2))}}u}=0.
	\end{align*}
	This is a consequence of
	\begin{gather*}
	\frac{\d^l}{\d z^l}\bigg|_{ z=\bar\zeta^{(2)}_{j}}
	\prodint{\tilde\phi^{(2,1)}_{1,k}(z_1),\frac{L_\Gamma^{(2)}(\bar z)}{\bar z-z_2}}_{\frac{L_{\mathfrak C}^{(1)}(z_1)}{\overline{(L_\Gamma^{(2)}(z_2))}}u}=
	\prodint{u_{z_1,\bar z_2},L_{\mathfrak C}^{(1)}(z_1)\tilde\phi^{(2,1)}_{1,k}(z_1)
		\otimes(\overline{L^{(2)}_\Gamma(z_2)})^{-1}\frac{\d^l}{\d z^l}\bigg|_{z=\bar \zeta^{(2)}_{j}}\frac{ \overline{L_\Gamma^{(2)}}( z)}{z-\bar z_2}}\\
	=\sum_{\nu=0}^{l}\binom{l}{\nu}\prodint{u_{z_1,\bar z_2},L_{\mathfrak C}^{(1)}(z_1)\tilde\phi^{(2,1)}_{1,k}(z_1)\otimes(\overline{L_\Gamma^{(2)}(z_2)})^{-1}\left(\frac{\d^{\nu}	
			\overline{ L_\Gamma^{(2)}}( z)}{\d z^\nu}\bigg|_{z=\bar\zeta^{(2)}_{j}}
		\frac{\d^{l-\nu}}{\d z^{l-\nu}}\bigg|_{z=\bar \zeta^{(2)}_{j}}\frac{1}{z-\bar z_2}\right)},
	\end{gather*}
	but $\dfrac{\d^\nu	 \overline{L_\Gamma^{(2)}}( z)}{\d z^\nu}\bigg|_{z=\bar\zeta^{(2)}_{j}}=0$ for  $\nu\in\{0,\dots,m^{(2)}_{j}-1\}$, and
	since $\operatorname{supp}_2(u)\cap \overline{\sigma(L_\Gamma^{(2)})}=\varnothing$, we get 	
\eqref{caleGer}.  Thus,
	\begin{align*}
	\frac{\d^l}{\d z^l}\bigg|_{z=\bar \zeta^{(2)}_{j}} \overline{L_\Gamma^{(2)}}(z)\tilde C_{1,k}^{(2,1)}(z)&=\sum_{i=1}^{d^{(2)}}\prodint{L^{(1)}_{\mathfrak C}\xi^{(2)}_{i},	\tilde\phi^{(2,1)}_{1,k}}\frac{\d^l}{\d z^l}\bigg|_{z=\bar \zeta^{(2)}_{j}}	 \overline{L^{(2)}_{\Gamma,[i]}}(z)\begin{bmatrix}
	(z-\bar \zeta^{(2)}_{i})^{m^{(2)}_i-1}\\ (z-\bar \zeta^{(2)}_{i})^{m^{(2)}_i-2}\\\vdots\\1
	\end{bmatrix}.
	\end{align*}
	For $l>0$,  we have
	\begin{align*}
	\frac{\d^l}{\d z^l}\bigg|_{z=\bar\zeta_{2,j}}	 \overline{L^{(2)}_{\Gamma,[i]}}(z)(z-\bar \zeta^{(2)}_{i})^{m}
	&=\sum_{\nu=0}^{m}\binom{l}{m-\nu} \Big(\frac{\d^{l-m+\nu}}{\d z^{l-m+\nu}}\bigg|_{z=\bar\zeta^{(2)}_{j}}	 \overline{L^{(2)}_{\Gamma,[i]}}(z)\Big)
	\frac{m!}{(m-\nu)!}\big(\bar \zeta^{(2)}_{j}-\bar \zeta^{(2)}_{i}\big)^{\nu}\\
	&=\sum_{\nu=0}^{m}
	\sum_{\sigma=1}^{l-m+\nu}
	\frac{l!m!}{\sigma!(m-\nu)!}
	\overline{(L^{(2)}_{\Gamma,[i]})^{(\sigma)}(\zeta^{(2)}_{j})}
	\big(\bar \zeta^{(2)}_{j}-\bar \zeta^{(2)}_{i}\big)^{\nu}.
	\end{align*}		
	But, if $i\neq j$,  $(L^{(2)}_{\Gamma,[i]})^{(\sigma)}(\zeta^{(2)}_{j})=0$ for $\sigma\in\{0,1,\dots,m^{(2)}_j-1\}$, which is our case because  $l\in\{0,1,\dots,m^{(2)}_{j}-1\}$; when $i=j$ we get that only  terms with $ \nu= 0 $ will survive and, therefore, $m\leq l$ with
	\begin{align*}
	\frac{1}{l!}	\frac{\d^l}{\d z^l}\bigg|_{z=\bar\zeta^{(2)}_{j}}	 \overline{L^{(2)}_{\Gamma,[j]}}(z)(z-\bar \zeta^{(2)}_{j})^{m}
	&=\ell^{(2)}_{{[j]},l-m}.
	\end{align*}
	
	To show \eqref{chorro (1)}	let's compute
	$\dfrac{1}{l!}\dfrac{\d^l}{\d \bar z^l}\bigg|_{\bar z=\zeta^{(1)}_{j}}\overline{	\overline{ L_\Gamma^{(1)}}(z)\tilde C_{2,k}^{(1,2)}(z)}$ with $l\in\{0,1,\dots,m^{(1)}_{j}-1\}$.  For that aim we evaluate
	\begin{multline*}
	\dfrac{\d^l}{\d \bar z^l}\bigg|_{\bar z=\zeta^{(1)}_{j}}\prodint{\frac{L_\Gamma^{(1)}(\bar z)}{\bar z-z_1}, \tilde \phi^{(1,2)}_{2,k}(z_2)}_{\frac{\overline{(L_{\mathfrak C}^{(1)}(z_2))}}{L_\Gamma^{(1)}(z_1)} u}=
	\prodint{\dfrac{\d^l}{\d \bar z^l}\bigg|_{\bar z=\zeta^{(1)}_{j}}\frac{L_\Gamma^{(1)}(\bar z)}{\bar z-z_1}, \tilde \phi^{(1,2)}_{2,k}(z_2)}_{\frac{\overline{(L_{\mathfrak C}^{(1)}(z_2))}}{L_\Gamma^{(1)}(z_1)} u}\\
	=\sum_{k=0}^l\binom{l}{k} (l-k)!\dfrac{\d^{k} L_\Gamma^{(1)}(\bar z)}{\d \bar z^{k}}\bigg|_{\bar z=\zeta^{(1)}_{j}}(-1)^{l-k}\prodint{\frac{1}{(\zeta^{(1)}_{j}-z_1)^{l-k+1}}, \tilde \phi^{(1,2)}_{2,k}(z_2)}_{\frac{\overline{(L_{\mathfrak C}^{(1)}(z_2))}}{L_\Gamma^{(1)}(z_1)} u},
	\end{multline*}
	that,  remembering that the zeros are not in support of the linear functional, vanishes. Finally we realize, that
	\begin{align*}
	\frac{1}{l!}\dfrac{\d^l}{\d \bar z^l}\bigg|_{\bar z=\zeta^{(1)}_{j}}L^{(1)}_{\Gamma,[i]}(\bar z) (\bar z-\zeta_i^{(1)})^m&=
	\sum_{k=0}^{l} \binom{l}{k}\dfrac{\d^kL^{(1)}_{\Gamma,[i]}(\bar z)}{\d \bar z^k}\bigg|_{\bar z=\zeta^{(1)}_{j}}\frac{m!}{(l-k)!}(\zeta^{(1)}_{j}-\zeta^{(1)}_{i})^{m-l+k}\\
	&=\begin{cases}
	0, & i\neq j,\\
	{	\bar \ell^{(1)}_{[j],l-m}},& i=j.
	\end{cases}
	\end{align*}
\end{proof}

\begin{proof}[Proof of Theorem \ref{Christoffel-Geronimus formulas}]
 From  \eqref{CC21k>} we get
	\begin{align*}
	\mathcal J^{\overline{L_\Gamma^{(2)}}}_{ \overline{L_\Gamma^{(2)}}\tilde  C^{(2,1)}_{{1,l}}} &=\big[
	(\Omega^{(2,1)}_1)_{l,l-2 N_\Gamma},\dots,(\Omega^{(2,1)}_1)_{l,l+2r-1} \big]
	\begin{bmatrix}
	\mathcal J^{\overline{L_\Gamma^{(2)}}}_{ C_{1,l-2 N_\Gamma}}\\\vdots \\ \mathcal J^{\overline{L_\Gamma^{(2)}}}_{C_{1,l+2 N_{\mathfrak C}-1}}
	\end{bmatrix}+(\Omega^{(2,1)}_1)_{l,l+2 N_{\mathfrak C}}\mathcal J^{\overline{L_\Gamma^{(2)}}}_{C_{1,l+2 N_{\mathfrak C}}}.
	\end{align*}
	Using \eqref{caleGer} and
	\begin{align}\label{conexion}
	(\Omega_1^{(2)})_{l,l-2 N_\Gamma}
	\phi_{1,l-2 N_\Gamma}(z)+\dots +(\Omega_1^{(2)})_{l,l+2 N_{\mathfrak C}}\phi_{1,l+2 N_{\mathfrak C}}(z)=L_{\mathfrak C}^{(1)}(z)\tilde\phi^{(2,1)}_1(z),
	\end{align}
	see Proposition \ref{Geronimus Connection Laurent}, we deduce that
	\begin{align*}
	\mathcal J^{\overline{L_\Gamma^{(2)}}}_{\overline{L_\Gamma^{(2)}}\tilde C^{(2,1)}_{{1,l}}} &=\big[
	(\Omega^{(2,1)}_1)_{l,l-2 N_\Gamma},\dots,(\Omega^{(2,1)}_1)_{l,l+2 N_{\mathfrak C}-1} \big]
	\begin{bmatrix}
	\prodint{\xi^{(2)},\phi_{1,l-2 N_\Gamma}}\\\vdots \\ \prodint{\xi^{(2)},\phi_{1,l+2 N_{\mathfrak C}-1}}
	\end{bmatrix}
	\mathcal L^{(2)}+(\Omega^{(2,1)}_1)_{l,l+2 N_{\mathfrak C}}\prodint{\xi^{(2)},\phi_{1,l+2 N_{\mathfrak C}}}\mathcal L^{(2)},
	\end{align*}
	and, consequently,
	\begin{multline*}
	\big[
	(\Omega^{(2,1)}_1)_{l,l-2 N_\Gamma},\dots,(\Omega^{(2,1)}_1)_{l,l+2 N_{\mathfrak C}-1} \big]\begin{bmatrix}
	\mathcal J^{\overline{L_\Gamma^{(2)}}}_{ C_{1,l-2 N_\Gamma}}-\prodint{\xi^{(2)},\phi_{1,l-2 N_\Gamma}}\mathcal L^{(2)}\\\vdots \\
	\mathcal J^{\overline{L_\Gamma^{(2)}}}_{ C_{1,l+2 N_{\mathfrak C}-1}}-\prodint{\xi^{(2)},\phi_{1,l+2 N_{\mathfrak C}-1}}\mathcal L^{(2)}
	\end{bmatrix}\\=-(\Omega^{(2,1)}_1)_{l,l+2 N_{\mathfrak C}}\bigg( \mathcal J^{\overline{L_\Gamma^{(2)}}}_{C_{1,l+2 N_{\mathfrak C}}}-\prodint{\xi^{(2)},\phi_{1,l+2 N_{\mathfrak C}}}\mathcal L^{(2)}\bigg).
	\end{multline*}
	Taking spectral jets along $L_{\mathfrak C}^{(1)}(z)$ of \eqref{conexion} we get
	\begin{align*}
	\big[
	(\Omega^{(2,1)}_1)_{l,l-2 N_\Gamma},\dots,(\Omega^{(2,1)}_1)_{l,l+2 N_{\mathfrak C}-1} \big]\begin{bmatrix}
	\mathcal J^{{L_{\mathfrak C}^{(1)}}}_{ \phi_{1,l-2 N_\Gamma}}\\\vdots \\
	\mathcal J^{{L_{\mathfrak C}^{(1)}}}_{ \phi_{1,l+2 N_{\mathfrak C}-1}}
	\end{bmatrix}=-(\Omega^{(2,1)}_1)_{l,l+2 N_{\mathfrak C}}\mathcal J^{{L_{\mathfrak C}^{(1)}}}_{\phi_{1,l+2 N_{\mathfrak C}}}.
	\end{align*}
	The previous two relations leads to
		\begin{multline*}
	\big[
	(\Omega^{(2,1)}_1)_{l,l-2 N_\Gamma},\dots,(\Omega^{(2,1)}_1)_{l,l+2 N_{\mathfrak C}-1} \big]\begin{bmatrix}
	\mathcal J^{\overline{L_\Gamma^{(2)}}}_{ C_{1,l-2 N_\Gamma}}-\prodint{\xi^{(2)},\phi_{1,l-2 N_\Gamma}}\mathcal L^{(2)} & \mathcal J^{{L_{\mathfrak C}^{(1)}}}_{ \phi_{1,l-2 N_\Gamma}}\\\vdots  &\vdots \\
	\mathcal J^{\overline{L_\Gamma^{(2)}}}_{ C_{1,l+2 N_{\mathfrak C}-1}}-\prodint{\xi^{(2)},\phi_{1,l+2 N_{\mathfrak C}-1}}\mathcal L^{(2)}& 	\mathcal J^{{L_{\mathfrak C}^{(1)}}}_{ \phi_{1,l+2 N_{\mathfrak C}-1}}
	\end{bmatrix}\\=-(\Omega^{(2,1)}_1)_{l,l+2 N_{\mathfrak C}}\bigg[\mathcal J^{\overline{L_\Gamma^{(2)}}}_{C_{1,l+2 N_{\mathfrak C}}}-\prodint{\xi^{(2)},\phi_{1,l+2 N_{\mathfrak C}}}\mathcal L^{(2)},\mathcal J^{{L_{\mathfrak C}^{(1)}}}_{\phi_{1,l+2 N_{\mathfrak C}}}\bigg].
	\end{multline*}
	and we have
			\begin{multline*}
	\big[
	(\Omega^{(2,1)}_1)_{l,l-2 N_\Gamma},\dots,(\Omega^{(2,1)}_1)_{l,l+2 N_{\mathfrak C}-1} \big]\\=-L^{(1)}_{\mathfrak C,(-1)^{l} N_{\mathfrak C}}\bigg[\mathcal J^{\overline{L_\Gamma^{(2)}}}_{C_{1,l+2 N_{\mathfrak C}}}-\prodint{\xi^{(2)},\phi_{1,l+2 N_{\mathfrak C}}}\mathcal L^{(2)},\mathcal J^{{L_{\mathfrak C}^{(1)}}}_{\phi_{1,l+2 N_{\mathfrak C}}}\bigg]\begin{bmatrix}
	\mathcal J^{\overline{L_\Gamma^{(2)}}}_{ C_{1,l-2 N_\Gamma}}-\prodint{\xi^{(2)},\phi_{1,l-2 N_\Gamma}}\mathcal L^{(2)} & \mathcal J^{{L_{\mathfrak C}^{(1)}}}_{ \phi_{1,l-2 N_\Gamma}}\\\vdots  &\vdots \\
	\mathcal J^{\overline{L_\Gamma^{(2)}}}_{ C_{1,l+2 N_{\mathfrak C}-1}}-\prodint{\xi^{(2)},\phi_{1,l+2 N_{\mathfrak C}-1}}\mathcal L^{(2)}& 	\mathcal J^{{L_{\mathfrak C}^{(1)}}}_{ \phi_{1,l+2 N_{\mathfrak C}-1}}
	\end{bmatrix}^{-1}.
	\end{multline*}
		Recalling  \eqref{conexion}  we conclude with the proof of this Christoffel formula. From this  equation, we also obtain
	\begin{multline*}
\hspace*{-1cm}	(\Omega^{(2,1)}_1)_{l,l-2 N_\Gamma}\\\hspace*{-1cm}	=-L^{(1)}_{\mathfrak C,(-1)^{l} N_{\mathfrak C}}\bigg[\mathcal J^{\overline{L_\Gamma^{(2)}}}_{C_{1,l+2 N_{\mathfrak C}}}-\prodint{\xi^{(2)},\phi_{1,l+2 N_{\mathfrak C}}}\mathcal L^{(2)},\mathcal J^{{L_{\mathfrak C}^{(1)}}}_{\phi_{1,l+2 N_{\mathfrak C}}}\bigg]\begin{bmatrix}
	\mathcal J^{\overline{L_\Gamma^{(2)}}}_{ C_{1,l-2 N_\Gamma}}-\prodint{\xi^{(2)},\phi_{1,l-2 N_\Gamma}}\mathcal L^{(2)} & \mathcal J^{{L_{\mathfrak C}^{(1)}}}_{ \phi_{1,l-2 N_\Gamma}}\\\vdots  &\vdots \\
	\mathcal J^{\overline{L_\Gamma^{(2)}}}_{ C_{1,l+2 N_{\mathfrak C}-1}}-\prodint{\xi^{(2)},\phi_{1,l+2 N_{\mathfrak C}-1}}\mathcal L^{(2)}& 	\mathcal J^{{L_{\mathfrak C}^{(1)}}}_{ \phi_{1,l+2 N_{\mathfrak C}-1}}
	\end{bmatrix}^{-1}\begin{bmatrix}
	1\\0\\\vdots\\0
	\end{bmatrix}
	\end{multline*}
	and recall \eqref{Omegalambda}.
Thus, in terms of last quasideterminants \cite{gelfand,Olver,toledano} we find
\begin{align*}
	\tilde\phi^{(2,1)}_{1,l}(z)
	&=
\frac{	L^{(1)}_{\mathfrak C,(-1)^{l} N_{\mathfrak C}}}{L_{\mathfrak C}^{(1)}(z)}\Theta_*\begin{bmatrix}
	\mathcal J^{\overline{L_\Gamma^{(2)}}}_{ C_{1,l-2 N_\Gamma}}-\prodint{\xi^{(2)},\phi_{1,l-2 N_\Gamma}}\mathcal L^{(2)} & \mathcal J^{{L_{\mathfrak C}^{(1)}}}_{ \phi_{1,l-2 N_\Gamma}}& \phi_{1,l-2 N_\Gamma}(z)\\\vdots  &\vdots&\vdots \\
	\mathcal J^{\overline{L_\Gamma^{(2)}}}_{ C_{1,l+2 N_{\mathfrak C}}}-\prodint{\xi^{(2)},\phi_{1,l+2 N_{\mathfrak C}}}\mathcal L^{(2)}& 	\mathcal J^{{L_{\mathfrak C}^{(1)}}}_{ \phi_{1,l+2 N_{\mathfrak C}}}&\phi_{1,l+2 N_{\mathfrak C}}(z)
	\end{bmatrix},
 	\end{align*}
	\begin{align*}
	\tilde H^{(2,1)}_{l}&=	 \dfrac{L^{(1)}_{\mathfrak C,(-1)^{l} N_{\mathfrak C}}}{ \overline{L^{(2)}_{\Gamma,(-1)^lN_\Gamma}}}H_{l-2 N_\Gamma}\Theta_*
	\begin{bmatrix}
	\mathcal J^{\overline{L_\Gamma^{(2)}}}_{ C_{1,l-2 N_\Gamma}}-\prodint{\xi^{(2)},\phi_{1,l-2 N_\Gamma}}\mathcal L^{(2)} & \mathcal J^{{L_{\mathfrak C}^{(1)}}}_{ \phi_{1,l-2 N_\Gamma}}&
1\\
		\mathcal J^{\overline{L_\Gamma^{(2)}}}_{ C_{1,l-2 N_\Gamma+1}}-\prodint{\xi^{(2)},\phi_{1,l-2 N_\Gamma+1}}\mathcal L^{(2)} & \mathcal J^{{L_{\mathfrak C}^{(1)}}}_{ \phi_{1,l-2 N_\Gamma+1}}&0\\
		\vdots  &\vdots&\vdots \\
	\mathcal J^{\overline{L_\Gamma^{(2)}}}_{ C_{1,l+2 N_{\mathfrak C}}}-\prodint{\xi^{(2)},\phi_{1,l+2 N_{\mathfrak C}}}\mathcal L^{(2)}& 	\mathcal J^{{L_{\mathfrak C}^{(1)}}}_{ \phi_{1,l+2 N_{\mathfrak C}}}&0
	\end{bmatrix},
	\end{align*}
that, expressing the quasideterminant as a quotient of determinants, gives \eqref{Ger12} and \eqref{GerH2}.
	
On the one hand we observe that \eqref{CC12k>}	implies
	\begin{align}
	\Big[(\Omega^{(1,2)}_2)_{l,l-2 N_\Gamma} , \dots, (\Omega^{(1,2)}_2)_{l,l+2 N_{\mathfrak C}-1} \Big]\begin{bmatrix}
	\mathcal J^{ \overline{L_\Gamma^{(1)}}}_{C_{2,l-2 N_\Gamma}}\\\vdots\\\mathcal J^{\overline{L_\Gamma^{(1)}}}_{C_{2,l+2 N_{\mathfrak C}-1}}
	\end{bmatrix}+(\Omega^{(1,2)}_2)_{l,l+2 N_{\mathfrak C}}\mathcal J^{\overline{L_\Gamma^{(1)}}}_{C_{2,l+2 N_{\mathfrak C}}}=\mathcal J^{\overline{L_\Gamma^{(1)}}}_{\overline{L_\Gamma^{(1)}}\tilde C^{(1,1)}_{2,l}}.
	\end{align}
	On the other hand,  from Proposition \ref{Geronimus Connection Laurent} we know that
	\begin{align}\label{conexion2}
	(\Omega^{(1,2)}_2)_{l,l-2 N_\Gamma}
	\phi_{2,l-2 N_\Gamma}(z)+\dots
	+(\Omega^{(1,2)}_2)_{l,l+2 N_{\mathfrak C}-1}
	\phi_{2,l+2 N_{\mathfrak C}-1}(z)+(\Omega^{(1,2)}_2)_{l,l+2 N_{\mathfrak C}}\phi_{2,l+2r}(z)=L_{\mathfrak C}^{(2)}(z)\tilde \phi_{2,l}^{(1,2)}(z),
	\end{align}
and taking into account \eqref{chorro (1)} we conclude
	\begin{multline}
	\Big[(\Omega^{(1,2)}_2)_{l,l-2 N_\Gamma} , \dots, (\Omega^{(1,2)}_2)_{l,l+2 N_{\mathfrak C}-1} \Big]\begin{bmatrix}
	\mathcal J^{\overline{L_\Gamma^{(1)}}}_{C_{2,l-2 N_\Gamma}}-{\prodint{\xi^{(1)},\phi_{2,l-2 N_\Gamma}}}{\mathcal L_\Gamma^{(1)}}\\\vdots\\\mathcal J^{\overline{L_\Gamma^{(1)}}}_{C_{2,l+2 N_{\mathfrak C}-1}}
	-{\prodint{\xi^{(1)},\phi_{2,l+2 N_{\mathfrak C}-1}}}{\mathcal L_\Gamma^{(1)}}\end{bmatrix}\\=-(\Omega_2^{(1,2)})_{l,l+2 N_{\mathfrak C}}\bigg(\mathcal J^{ \overline{L^{(1)}_\Gamma}}_{C_{2,l+2r}}-{\prodint{	\xi^{(1)},\phi_{2,l+2 N_{\mathfrak C}}}}{\mathcal L_\Gamma^{(1)}}\bigg).
	\end{multline}
	Now, the computation of  the spectral jet of \eqref{conexion2} along $L_{\mathfrak C}^{(2)}(z)$ leads to
		\begin{align*}
	\big[
	(\Omega^{(1,2)}_2)_{l,l-2 N_\Gamma},\dots,(\Omega^{(1,2)}_2)_{l,l+2 N_{\mathfrak C}-1} \big]\begin{bmatrix}
	\mathcal J^{{L_{\mathfrak C}^{(2)}}}_{ \phi_{2,l-2 N_\Gamma}}\\\vdots \\
	\mathcal J^{{L_{\mathfrak C}^{(2)}}}_{ \phi_{2,l+2 N_{\mathfrak C}-1}}
	\end{bmatrix}=-(\Omega^{(1,2)}_2)_{l,l+2 N_{\mathfrak C}}\mathcal J^{{L_{\mathfrak C}^{(2)}}}_{\phi_{2,l+2 N_{\mathfrak C}}}.
	\end{align*}
	Hence, if we gather all this information together,  we obtain
	\begin{multline*}
	\big[
	(\Omega^{(1,2)}_1)_{l,l-2 N_\Gamma},\dots,(\Omega^{(1,2)}_1)_{l,l+2 N_{\mathfrak C}-1} \big]\\=-L^{(2)}_{\mathfrak C,(-1)^{l} N_{\mathfrak C}}\bigg[\mathcal J^{\overline{L_\Gamma^{(1)}}}_{C_{2,l+2 N_{\mathfrak C}}}-\prodint{\xi^{(1)},\phi_{2,l+2 N_{\mathfrak C}}}\mathcal L^{(1)},\mathcal J^{{L_{\mathfrak C}^{(2)}}}_{\phi_{2,l+2 N_{\mathfrak C}}}\bigg]\begin{bmatrix}
	\mathcal J^{\overline{L_\Gamma^{(1)}}}_{ C_{2,l-2 N_\Gamma}}-\prodint{\xi^{(1)},\phi_{2,l-2 N_\Gamma}}\mathcal L^{(1)} & \mathcal J^{{L_{\mathfrak C}^{(2)}}}_{ \phi_{2,l-2 N_\Gamma}}\\\vdots  &\vdots \\
	\mathcal J^{\overline{L_\Gamma^{(1)}}}_{ C_{2,l+2 N_{\mathfrak C}-1}}-\prodint{\xi^{(1)},\phi_{2,l+2 N_{\mathfrak C}-1}}\mathcal L^{(1)}& 	\mathcal J^{{L_{\mathfrak C}^{(2)}}}_{ \phi_{2,l+2 N_{\mathfrak C}-1}}
	\end{bmatrix}^{-1},
	\end{multline*}
which as a byproduct offers
	\begin{multline}
	(\Omega^{(1,2)}_2)_{l,l-2 N_\Gamma}\\=-L^{(2)}_{\mathfrak C,(-1)^{l} N_{\mathfrak C}}\bigg[\mathcal J^{\overline{L_\Gamma^{(1)}}}_{C_{2,l+2 N_{\mathfrak C}}}-\prodint{\xi^{(1)},\phi_{2,l+2 N_{\mathfrak C}}}\mathcal L^{(1)},\mathcal J^{{L_{\mathfrak C}^{(2)}}}_{\phi_{2,l+2 N_{\mathfrak C}}}\bigg]\begin{bmatrix}
	\mathcal J^{\overline{L_\Gamma^{(1)}}}_{ C_{2,l-2 N_\Gamma}}-\prodint{\xi^{(1)},\phi_{2,l-2 N_\Gamma}}\mathcal L^{(1)} & \mathcal J^{{L_{\mathfrak C}^{(2)}}}_{ \phi_{2,l-2 N_\Gamma}}\\\vdots  &\vdots \\
	\mathcal J^{\overline{L_\Gamma^{(1)}}}_{ C_{2,l+2 N_{\mathfrak C}-1}}-\prodint{\xi^{(1)},\phi_{2,l+2 N_{\mathfrak C}-1}}\mathcal L^{(1)}& 	\mathcal J^{{L_{\mathfrak C}^{(2)}}}_{ \phi_{2,l+2 N_{\mathfrak C}-1}}
	\end{bmatrix}^{-1}\begin{bmatrix}
	1\\0\\\vdots\\0
	\end{bmatrix}.
	\end{multline}
	Therefore, we have proven the following last quasideterminantal expressions
\begin{align*}
	\tilde\phi^{(1,2)}_{2,l}(z)
&=
\frac{	L^{(2)}_{\mathfrak C,(-1)^{l} N_{\mathfrak C}}}{L_{\mathfrak C}^{(2)}(z)}\Theta_*\begin{bmatrix}
\mathcal J^{\overline{L_\Gamma^{(1)}}}_{ C_{2,l-2 N_\Gamma}}-\prodint{\xi^{(1)},\phi_{2,l-2 N_\Gamma}}\mathcal L^{(1)} & \mathcal J^{{L_{\mathfrak C}^{(2)}}}_{ \phi_{2,l-2 N_\Gamma}}& \phi_{2,l-2 N_\Gamma}(z)\\\vdots  &\vdots&\vdots \\
\mathcal J^{\overline{L_\Gamma^{(1)}}}_{ C_{2,l+2 N_{\mathfrak C}}}-\prodint{\xi^{(1)},\phi_{2,l+2 N_{\mathfrak C}}}\mathcal L^{(1)}& 	\mathcal J^{{L_{\mathfrak C}^{(2)}}}_{ \phi_{2,l+2 N_{\mathfrak C}}}&\phi_{2,l+2 N_{\mathfrak C}}(z)
\end{bmatrix},
	\end{align*}
	\begin{align*}
	\bar{	\tilde H}^{(1,2)}_{l}&=	 \dfrac{L^{(2)}_{\mathfrak C,(-1)^{l} N_{\mathfrak C}}}{ \overline{L^{(2)}_{\Gamma,(-1)^ln}}}\bar H_{l-2 N_\Gamma}\Theta_*
	\begin{bmatrix}
	\mathcal J^{\overline{L_\Gamma^{(1)}}}_{ C_{2,l-2 N_\Gamma}}-\prodint{\xi^{(1)},\phi_{2,l-2 N_\Gamma}}\mathcal L^{(1)} & \mathcal J^{{L_{\mathfrak C}^{(2)}}}_{ \phi_{2,l-2 N_\Gamma}}&
	1\\
	\mathcal J^{\overline{L_\Gamma^{(1)}}}_{ C_{2,l-2 N_\Gamma+1}}-\prodint{\xi^{(1)},\phi_{2,l-2 N_\Gamma+1}}\mathcal L^{(1)} & \mathcal J^{{L_{\mathfrak C}^{(2)}}}_{ \phi_{2,l-2 N_\Gamma+1}}&0\\
	\vdots  &\vdots&\vdots \\
	\mathcal J^{\overline{L_\Gamma^{(1)}}}_{ C_{2,l+2 N_{\mathfrak C}}}-\prodint{\xi^{(1)},\phi_{2,l+2 N_{\mathfrak C}}}\mathcal L^{(1)}& 	\mathcal J^{{L_{\mathfrak C}^{(2)}}}_{ \phi_{2,l+2 N_{\mathfrak C}}}&0
	\end{bmatrix},
	\end{align*}
	from where the determinantal formulas \eqref{Ger21} and \eqref{GerH1} follow immediately.

	We will prove  \eqref{Ger11} y \eqref{Ger22} simultaneously.
	Let's write   \eqref{mix CF1}  and  \eqref{mix FC2}  as follows
	\begin{align*}
	&\begin{multlined}[t][0.9\textwidth]
	\sum_{k=0}^{l-1}\overline{	 \overline{L_\Gamma^{(1)}}( x_1)\tilde  C^{(1,2)}_{2,k}(x_1)}(\tilde  H^{(1,2)}_{k})^{-1}\tilde\phi^{(1,2)}_{1,k}(x_2)-	 {L_\Gamma^{(1)}}(x_2)K_{C_2}^{[l]}(\bar x_1,x_2)-
	\delta {L_\Gamma^{(1)}}(\bar x_1,x_2)
	\\= {\Big[\tilde \phi^{(1,2)}_{1,l- 2N_{\mathfrak C}}(x_2)({\tilde H}_{l- 2N_{\mathfrak C}}^{(1,2)})^{-1},\dots,
		\tilde \phi^{(1,2)}_{1,l+ 2 N_\Gamma-1}(x_2)({\tilde H}_{l+ 2 N_\Gamma-1}^{(1,2)})^{-1}\Big]}
\begin{bmatrix}
0_{2 N_{\mathfrak C}\times 2 N_\Gamma}&
\overline{\mathfrak  C^{(1,2)}_{1,l}}\\
-\overline{\Gamma^{(1,2)}_{2,l}}&0_{2 N_\Gamma\times 2 N_{\mathfrak C}}
\end{bmatrix}
	\begin{bmatrix}
	\overline{ C_{2,l-2 N_\Gamma}(x_1)}\\ \vdots\\  \overline{
		C_{2,l+2 N_{\mathfrak C}-1}(x_1)}
	\end{bmatrix},
	\end{multlined}	
	\\
	&\begin{multlined}[t][0.9\textwidth]
	\sum_{k=0}^{l-1}\overline{\tilde \phi^{(2,1)}_{2,k}(x_1)}(\tilde  H^{(2,1)}_{k})^{-1} \overline{L_\Gamma^{(2)}}(x_2)\tilde  C^{(2,1)}_{1,k}(x_2)-	 \overline{L_\Gamma^{(2)}}(\bar x_1) K_{C_1}^{[l]}(\bar x_1,x_2)
	-\delta  \overline{L_\Gamma^{(2)}}(\bar x_1,x_2)
	\\=\Big[\overline{\tilde \phi^{(2,1)}_{2,l- 2N_{\mathfrak C}}(x_1)}(\tilde  H^{(2,1)}_{l- 2N_{\mathfrak C}})^{-1},\dots, \overline{\tilde \phi^{(2,1)}_{2,l+ 2 N_\Gamma-1}(x_1)}(\tilde  H^{(2,1)}_{l+ 2 N_\Gamma-1})^{-1}\Big]
\begin{bmatrix}
0_{2 N_{\mathfrak C}\times 2 N_\Gamma}&\mathfrak  C^{(2,1)}_{1,l}
\\
-\Gamma^{(2,1)}_{1,l}&0_{2 N_\Gamma\times 2 N_{\mathfrak C}}
\end{bmatrix}
	\begin{bmatrix}
	C_{1,l-2 N_\Gamma}(x_2)\\\vdots\\C_{1,l+2 N_{\mathfrak C}-1}(x_2)
	\end{bmatrix}.
	\end{multlined}
		\end{align*}
Now,  we  compute the  following spectral jets
\begin{align*}
&\begin{multlined}[t][0.9\textwidth]
\sum_{k=0}^{l-1}(\tilde  H^{(1,2)}_{k})^{-1}\tilde\phi^{(1,2)}_{1,k}(x_2)
\overline{\mathcal J^{ \overline{L_\Gamma^{(1)}}}_{ \overline{L_\Gamma^{(1)}}\tilde  C^{(1,2)}_{2,k}}}
-	 {L_\Gamma^{(1)}}(x_2)\mathcal J^{ \overline{L_\Gamma^{(1)}}}_{K_{C_2}^{[l]}}(x_2)-
\mathcal J^{ \overline{L_\Gamma^{(1)}}}_{\delta {L_\Gamma^{(1)}}}(x_2)
\\=
 {\Big[\tilde \phi^{(1,2)}_{1,l- 2N_{\mathfrak C}}(x_2)({\tilde H}_{l- 2N_{\mathfrak C}}^{(1,2)}]^{-1},\dots,
	\tilde \phi^{(1,2)}_{1,l+ 2 N_\Gamma-1}(x_2)({\tilde H}_{l+ 2 N_\Gamma-1}^{(1,2)})^{-1}\Big]}
\begin{bmatrix}
0_{2 N_{\mathfrak C}\times 2 N_\Gamma}&
\overline{\mathfrak  C^{(1,2)}_{1,l}}\\
-\overline{\Gamma^{(1,2)}_{2,l}}&0_{2 N_\Gamma\times 2 N_{\mathfrak C}}
\end{bmatrix}
\begin{bmatrix}
\overline{ \mathcal J^{ \overline{L_\Gamma^{(1)}}}_{C_{2,l-2 N_\Gamma}}}\\ \vdots\\  \overline{
	\mathcal J^{ \overline{L_\Gamma^{(1)}}}_{	C_{2,l+2 N_{\mathfrak C}-1}}}
\end{bmatrix},
\end{multlined}	
\end{align*}
	\begin{align*}
&\begin{multlined}[t][0.9\textwidth]
\sum_{k=0}^{l-1}\overline{\tilde \phi^{(2,1)}{_{2,k}}(x_1)}(\tilde  H^{(2,1)}_{k})^{-1}\mathcal  J^{ \overline{L_\Gamma^{(2)}}}_{ \overline{L^{(2)}_\Gamma}\tilde  C^{(2,1)}_{1,k}}-	\overline{L_\Gamma^{(2)}}(\bar x_1) \mathcal  J^{\overline{L_\Gamma^{(2)}}}_{K_{C_1}^{[l]}}(\bar x_1)
-\mathcal  J^{\overline{L_\Gamma^{(2)}}}_{\delta \overline{L_\Gamma^{(2)}}}(\bar x_1)
\\=\Big[\overline{\tilde \phi^{(2,1)}_{2,l- 2N_{\mathfrak C}}(x_1)}(\tilde  H^{(2,1)}_{l- 2N_{\mathfrak C}})^{-1},\dots, \overline{\tilde \phi^{(2,1)}_{2,l+ 2 N_\Gamma-1}(x_1)}(\tilde  H^{(2,1)}_{l+ 2 N_\Gamma-1})^{-1}\Big]
\begin{bmatrix}
0_{2 N_{\mathfrak C}\times 2 N_\Gamma}&\mathfrak  C^{(2,1)}_{1,l}
\\
-\Gamma^{(2,1)}_{1,l}&0_{2 N_\Gamma\times 2 N_{\mathfrak C}}
\end{bmatrix}
\begin{bmatrix}
\mathcal J^{ \overline{L_\Gamma^{(2)}}}_{ {L_\Gamma^{(2)}}	C_{1,l-2 N_\Gamma}}\\\vdots\\\mathcal J^{ \overline{L_\Gamma^{(2)}}}_{ {L_\Gamma^{(2)}}C_{1,l-1}(x_2)}
\end{bmatrix}.
\end{multlined}
\end{align*}
From \eqref{caleGer} and \eqref{chorro (1)} we deduce
\begin{align*}
&\begin{multlined}[t][0.9\textwidth]
\sum_{k=0}^{l-1}(\tilde  H^{(1,2)}_{k})^{-1}\tilde\phi^{(1,2)}_{1,k}(x_2)
	\overline{\prodint{L^{(1)}_{\mathfrak C}\xi^{(1)},	\tilde\phi^{(1,2)}_{2,k}}	\mathcal L^{(1)}}
-	 {L_\Gamma^{(1)}}(x_2)\mathcal J^{ \overline{L_\Gamma^{(1)}}}_{K_{C_2}^{[l]}}(x_2)-
\mathcal J^{ \overline{L_\Gamma^{(1)}}}_{\delta {L_\Gamma^{(1)}}}(x_2)
\\=
{\Big[\tilde \phi^{(1,2)}_{1,l- 2N_{\mathfrak C}}(x_2)({\tilde H}_{l- 2N_{\mathfrak C}}^{(1,2)})^{-1},\dots,
	\tilde \phi^{(1,2)}_{1,l+ 2 N_\Gamma-1}(x_2)({\tilde H}_{l+ 2 N_\Gamma-1}^{(1,2)})^{-1}\Big]}
\begin{bmatrix}
0_{2 N_{\mathfrak C}\times 2 N_\Gamma}&
\overline{\mathfrak  C^{(1,2)}_{1,l}}\\
-\overline{\Gamma^{(1,2)}_{2,l}}&0_{2 N_\Gamma\times 2 N_{\mathfrak C}}
\end{bmatrix}
\begin{bmatrix}
\overline{ \mathcal J^{ \overline{L_\Gamma^{(1)}}}_{C_{2,l-2 N_\Gamma}}}\\ \vdots\\  \overline{
	\mathcal J^{ \overline{L_\Gamma^{(1)}}}_{	C_{2,l+2 N_{\mathfrak C}-1}}}
\end{bmatrix},
\end{multlined}	
\end{align*}
\begin{align*}
&\begin{multlined}[t][0.9\textwidth]
\sum_{k=0}^{l-1}\overline{\tilde \phi^{(2,1)}{_{2,k}}(x_1)}(\tilde  H^{(2,1)}_{k})^{-1}	\prodint{L^{(2)}_{\mathfrak C}\xi^{(2)},\tilde\phi^{(2,1)}_{1,k}}\mathcal L^{(2)}
-	\overline{L_\Gamma^{(2)}}(\bar x_1) \mathcal  J^{\overline{L_\Gamma^{(2)}}}_{K_{C_1}^{[l]}}(\bar x_1)
-\mathcal  J^{\overline{L_\Gamma^{(2)}}}_{\delta \overline{L_\Gamma^{(2)}}}(\bar x_1)
\\=\Big[\overline{\tilde \phi^{(2,1)}_{2,l- 2N_{\mathfrak C}}(x_1)}(\tilde  H^{(2,1)}_{l- 2N_{\mathfrak C}})^{-1},\dots, \overline{\tilde \phi^{(2,1)}_{2,l+ 2 N_\Gamma-1}(x_1)}(\tilde  H^{(2,1)}_{l+ 2 N_\Gamma-1})^{-1}\Big]
\begin{bmatrix}
0_{2 N_{\mathfrak C}\times 2 N_\Gamma}&\mathfrak  C^{(2,1)}_{1,l}
\\
-\Gamma^{(2,1)}_{1,l}&0_{2 N_\Gamma\times 2 N_{\mathfrak C}}
\end{bmatrix}
\begin{bmatrix}
\mathcal J^{ \overline{L_\Gamma^{(2)}}}_{ {L_\Gamma^{(2)}}	C_{1,l-2 N_\Gamma}}\\\vdots\\\mathcal J^{ \overline{L_\Gamma^{(2)}}}_{ {L_\Gamma^{(2)}}C_{1,l-1}(x_2)}
\end{bmatrix}.
\end{multlined}
\end{align*}
Recalling \eqref{kernelChristoffeldefinicion} we get
	\begin{align*}
	&\begin{multlined}[t][0.9\textwidth]
	\prodint{\overline{L^{(2)}_{\mathfrak C}(z)(\xi^{(1)})_z},	\tilde K^{(1,2),[l]}(\bar z,x_2)}	\overline{		\mathcal L^{(1)}}
	=	 {L_\Gamma^{(1)}}(x_2)\mathcal J^{ \overline{L_\Gamma^{(1)}}}_{K_{C_2}^{[l]}}(x_2)+
	\mathcal J^{ \overline{L_\Gamma^{(1)}}}_{\delta {L_\Gamma^{(1)}}}(x_2)
	\\+
	{\Big[\tilde \phi^{(1,2)}_{1,l- 2N_{\mathfrak C}}(x_2)({\tilde H}_{l- 2N_{\mathfrak C}}^{(1,2)})^{-1},\dots,
		\tilde \phi^{(1,2)}_{1,l+ 2 N_\Gamma-1}(x_2)({\tilde H}_{l+ 2 N_\Gamma-1}^{(1,2)})^{-1}\Big]}
\begin{bmatrix}
0_{2 N_{\mathfrak C}\times 2 N_\Gamma}&
\overline{\mathfrak  C^{(1,2)}_{1,l}}\\
-\overline{\Gamma^{(1,2)}_{2,l}}&0_{2 N_\Gamma\times 2 N_{\mathfrak C}}
\end{bmatrix}
	\begin{bmatrix}
	\overline{ \mathcal J^{ \overline{L_\Gamma^{(1)}}}_{C_{2,l-2 N_\Gamma}}}\\ \vdots\\  \overline{
		\mathcal J^{ \overline{L_\Gamma^{(1)}}}_{	C_{2,l+2 N_{\mathfrak C}-1}}}
	\end{bmatrix},
	\end{multlined}	
	\end{align*}
	\begin{align*}
	&\begin{multlined}[t][0.9\textwidth]
	\prodint{L^{(1)}_{\mathfrak C}(z)(\xi^{(2)})_z,\tilde K^{(2,1),[l]}(\bar x_1,z)}\mathcal L^{(2)}
	=	\overline{L_\Gamma^{(2)}}(\bar x_1) \mathcal  J^{\overline{L_\Gamma^{(2)}}}_{K_{C_1}^{[l]}}(\bar x_1)
	+\mathcal  J^{\overline{L_\Gamma^{(2)}}}_{\delta \overline{L_\Gamma^{(2)}}}(\bar x_1)
	\\+\Big[\overline{\tilde \phi^{(2,1)}_{2,l- 2N_{\mathfrak C}}(x_1)}(\tilde  H^{(2,1)}_{l- 2N_{\mathfrak C}})^{-1},\dots, \overline{\tilde \phi^{(2,1)}_{2,l+ 2 N_\Gamma-1}(x_1)}(\tilde  H^{(2,1)}_{l+ 2 N_\Gamma-1})^{-1}\Big]
\begin{bmatrix}
0_{2 N_{\mathfrak C}\times 2 N_\Gamma}&\mathfrak  C^{(2,1)}_{1,l}
\\
-\Gamma^{(2,1)}_{1,l}&0_{2 N_\Gamma\times 2 N_{\mathfrak C}}
\end{bmatrix}
	\begin{bmatrix}
	\mathcal J^{ \overline{L_\Gamma^{(2)}}}_{ {L_\Gamma^{(2)}}	C_{1,l-2 N_\Gamma}}\\\vdots\\\mathcal J^{ \overline{L_\Gamma^{(2)}}}_{ {L_\Gamma^{(2)}}C_{1,l-1}(x_2)}
	\end{bmatrix}.
	\end{multlined}
	\end{align*}
	Now, from 	 \eqref{GerKerNor1} and \eqref{GerKerNor2}, we deduce	
	\begin{multline*}
\prodint{\overline{L^{(2)}_{\mathfrak C}(z)(\xi^{(1)})_z},{\tilde K^{(1,2),[l]}(\bar z,x_2)}}\overline{\mathcal L^{(1)}}={L_\Gamma^{(1)}(x_2)} \prodint{\overline{(\xi^{(1)})_z},{K^{[l]}(\bar z,x_2)}}
\overline{\mathcal L^{(1)}}\\+{\Big[\tilde \phi^{(1,2)}_{1,l- 2N_{\mathfrak C}}(z_2)\big({\tilde H}_{l- 2N_{\mathfrak C}}^{(1,2)}\big)^{-1},\dots,
		\tilde \phi^{(1,2)}_{1,l+2 N_\Gamma-1}(z_2)\big({\tilde H}_{l+ 2 N_\Gamma-1}^{(1,2)}\big)^{-1}\Big]}
\begin{bmatrix}
0_{2 N_{\mathfrak C}\times 2 N_\Gamma}&
\overline{\mathfrak  C^{(1,2)}_{1,l}}\\
-\overline{\Gamma^{(1,2)}_{2,l}}&0_{2 N_\Gamma\times 2 N_{\mathfrak C}}
\end{bmatrix}
	\begin{bmatrix}
	\overline{ \prodint{\xi^{(1)},\phi_{2,l-2 N_\Gamma}}\mathcal L^{(1)}}\\ \vdots\\ 	\overline{ \prodint{\xi^{(1)},\phi_{2,l+2 N_{\mathfrak C}-1}}\mathcal L^{(1)}}\end{bmatrix},
	\end{multline*}
	\begin{multline*}
\prodint{L^{(1)}_{\mathfrak C}(z)(\xi^{(2)})_z,\tilde K^{(2),[l]}(\bar x_1,z)}\mathcal L^{(2)}= \overline{ L_\Gamma^{(2)}( x_1)} \prodint{{(\xi^{(2)})_z},K^{[l]}(\bar x_1,z)}\mathcal L^{(2)}\\+{\Big[\overline{\tilde \phi^{(2,1)}_{2,l- 2N_{\mathfrak C}}(x_1)}({\tilde H}_{l- 2N_{\mathfrak C}}^{(2,1)})^{-1},\dots,
		\overline{\tilde \phi^{(2,1)}_{2,l+ 2 N_\Gamma-1}(x_1)}({\tilde H}_{l+ 2 N_\Gamma-1}^{(2,1)})^{-1}\Big]}	
\begin{bmatrix}
0_{2 N_{\mathfrak C}\times 2 N_\Gamma}&\mathfrak  C^{(2,1)}_{1,l}
\\
-\Gamma^{(2,1)}_{1,l}&0_{2 N_\Gamma\times 2 N_{\mathfrak C}}
\end{bmatrix}
	\begin{bmatrix}
\prodint{\xi^{(2)},\phi_{1,l-2 N_\Gamma}}\mathcal L^{(2)}\\ \vdots\\
	\prodint{\xi^{(2)},\phi_{1,l+2 N_{\mathfrak C}-1}}\mathcal L^{(2)}\end{bmatrix}.
	\end{multline*}	
	Therefore, we conclude
		\begin{align*}
	&\begin{multlined}[t][0.9\textwidth]
 {L_\Gamma^{(1)}}(x_2)\mathcal J^{ \overline{L_\Gamma^{(1)}}}_{K_{C_2}^{[l]}}(x_2)+
	\mathcal J^{ \overline{L_\Gamma^{(1)}}}_{\delta {L_\Gamma^{(1)}}}(x_2)-{L_\Gamma^{(1)}(x_2)} \prodint{\overline{(\xi^{(1)})_z},{K^{[l]}(\bar z,x_2)}}
	\overline{\mathcal L^{(1)}}
	\\=-
	{\Big[\tilde \phi^{(1,2)}_{1,l- 2N_{\mathfrak C}}(x_2)({\tilde H}_{l- 2N_{\mathfrak C}}^{(1,2)})^{-1},\dots,
		\tilde \phi^{(1,2)}_{1,l+ 2 N_\Gamma-1}(x_2)({\tilde H}_{l+ 2 N_\Gamma-1}^{(1,2)})^{-1}\Big]}
\\\times	
\begin{bmatrix}
0_{2 N_{\mathfrak C}\times 2 N_\Gamma}&
\overline{\mathfrak  C^{(1,2)}_{1,l}}\\
-\overline{\Gamma^{(1,2)}_{2,l}}&0_{2 N_\Gamma\times 2 N_{\mathfrak C}}
\end{bmatrix}
	\begin{bmatrix}
	\overline{ \mathcal J^{ \overline{L_\Gamma^{(1)}}}_{C_{2,l-2 N_\Gamma}}-\prodint{\xi^{(1)},\phi_{2,l-2 N_\Gamma}}\mathcal L^{(1)}}\\ \vdots\\  \overline{
		\mathcal J^{ \overline{L_\Gamma^{(1)}}}_{	C_{2,l+2 N_{\mathfrak C}-1}}-\prodint{\xi^{(1)},\phi_{2,l+2 N_{\mathfrak C}-1}}\mathcal L^{(1)}}
	\end{bmatrix},
	\end{multlined}	
	\end{align*}
	\begin{align*}
	&\begin{multlined}[t][0.9\textwidth]
	\overline{L_\Gamma^{(2)}( x_1)} \mathcal  J^{\overline{L_\Gamma^{(2)}}}_{K_{C_1}^{[l]}}(\bar x_1)
	+\mathcal  J^{\overline{L_\Gamma^{(2)}}}_{\delta \overline{L_\Gamma^{(2)}}}(\bar x_1)-\overline{ L_\Gamma^{(2)}( x_1)} \prodint{(\xi^{(2)})_z,K^{[l]}(\bar x_1,z)}\mathcal L^{(2)}
	\\=-\Big[\overline{\tilde \phi^{(2,1)}_{2,l- 2N_{\mathfrak C}}(x_1)}(\tilde  H^{(2,1)}_{l- 2N_{\mathfrak C}})^{-1},\dots, \overline{\tilde \phi^{(2,1)}_{2,l+ 2 N_\Gamma-1}(x_1)}(\tilde  H^{(2,1)}_{l+ 2 N_\Gamma-1})^{-1}\Big]\\\times
\begin{bmatrix}
0_{2 N_{\mathfrak C}\times 2 N_\Gamma}&\mathfrak  C^{(2,1)}_{1,l}\\
-\Gamma^{(2,1)}_{1,l}&0_{2 N_\Gamma\times 2 N_{\mathfrak C}}
\end{bmatrix}
\begin{bmatrix}
	\mathcal J^{ \overline{L_\Gamma^{(2)}}}_{ 	C_{1,l-2 N_\Gamma}}-\prodint{\xi^{(2)},\phi_{1,l-2 N_\Gamma}}\mathcal L^{(2)}\\\vdots\\\mathcal J^{ \overline{L_\Gamma^{(2)}}}_{ C_{1,l+2 N_{\mathfrak C}-1}(x_2)}-\prodint{\xi^{(2)},\phi_{1,l+2 N_{\mathfrak C}-1}}\mathcal L^{(2)}
	\end{bmatrix}.
	\end{multlined}
	\end{align*}
We return to  \eqref{GerKerNor1} and \eqref{GerKerNor2}, and deduce, taking spectral jets, that
\begin{multline}
{L_\Gamma^{(1)}(x_2)} \mathcal J^{\overline{L_{\mathfrak C}^{(2)}}}_{K^{[l]}}(x_2)\\=-{\Big[\tilde \phi^{(1,2)}_{1,l- 2N_{\mathfrak C}}(x_2)\big({\tilde H}_{l- 2N_{\mathfrak C}}^{(1,2)}\big)^{-1},\dots,
	\tilde \phi^{(1,2)}_{1,l+2 N_\Gamma-1}(x_2)\big({\tilde H}_{l+ 2 N_\Gamma-1}^{(1,2)}\big)^{-1}\Big]}
\begin{bmatrix}
0_{2 N_{\mathfrak C}\times 2 N_\Gamma}&
\overline{\mathfrak  C^{(1,2)}_{1,l}}\\
-\overline{\Gamma^{(1,2)}_{2,l}}&0_{2 N_\Gamma\times 2 N_{\mathfrak C}}
\end{bmatrix}
\begin{bmatrix}
\overline{ \mathcal J^{{L_{\mathfrak C}^{(2)}}}_{\phi_{2,l-2 N_\Gamma}}}\\ \vdots\\  \overline{
	\mathcal J^{{L_{\mathfrak C}^{(2)}}}_{\phi_{2,l+2 N_{\mathfrak C}-1}}}\end{bmatrix},
\end{multline}
\begin{multline}
\overline{ L_\Gamma^{(2)}( x_1)} \mathcal J^{L_{\mathfrak C}^{(1)}}_{K^{[l]}}(\bar x_1)\\=-{\Big[\overline{\tilde \phi^{(2,1)}_{2,l- 2N_{\mathfrak C}}(x_1)}({\tilde H}_{l- 2N_{\mathfrak C}}^{(2,1)})^{-1},\dots,
	\overline{\tilde \phi^{(2,1)}_{2,l+ 2 N_\Gamma-1}(x_1)}({\tilde H}_{l+ 2 N_\Gamma-1}^{(2,1)})^{-1}\Big]}	
\begin{bmatrix}
0_{2 N_{\mathfrak C}\times 2 N_\Gamma}&\mathfrak  C^{(2,1)}_{1,l}\\
-\Gamma^{(2,1)}_{1,l}&0_{2 N_\Gamma\times 2 N_{\mathfrak C}}
\end{bmatrix}
\begin{bmatrix}
\mathcal J^{L_{\mathfrak C}^{(1)}}_{\phi_{1,l-2 N_\Gamma}}\\ \vdots\\
\mathcal J^{L_{\mathfrak C}^{(1)}}_{\phi_{1,l+2 N_{\mathfrak C}-1}}\end{bmatrix}.
\end{multline}
Consequently,
		\begin{align*}
&\begin{multlined}[t][0.9\textwidth]
\bigg[{L_\Gamma^{(1)}}(x_2)\mathcal J^{ \overline{L_\Gamma^{(1)}}}_{K_{C_2}^{[l]}}(x_2)+
\mathcal J^{ \overline{L_\Gamma^{(1)}}}_{\delta {L_\Gamma^{(1)}}}(x_2)-{L_\Gamma^{(1)}(x_2)} \prodint{\overline{(\xi^{(1)})_z},{K^{[l]}(\bar z,x_2)}}
\overline{\mathcal L^{(1)}},{L_\Gamma^{(1)}(x_2)} \mathcal J^{\overline{L_{\mathfrak C}^{(2)}}}_{K^{[l]}}(x_2)\bigg]
\\=-
{\Big[\tilde \phi^{(1,2)}_{1,l- 2N_{\mathfrak C}}(x_2)({\tilde H}_{l- 2N_{\mathfrak C}}^{(1,2)})^{-1},\dots,
	\tilde \phi^{(1,2)}_{1,l+ 2 N_\Gamma-1}(x_2)({\tilde H}_{l+ 2 N_\Gamma-1}^{(1,2)})^{-1}\Big]}
\\\times	
\begin{bmatrix}
0_{2 N_{\mathfrak C}\times 2 N_\Gamma}&
\overline{\mathfrak  C^{(1,2)}_{1,l}}\\
-\overline{\Gamma^{(1,2)}_{2,l}}&0_{2 N_\Gamma\times 2 N_{\mathfrak C}}
\end{bmatrix}
\begin{bmatrix}
\overline{ \mathcal J^{ \overline{L_\Gamma^{(1)}}}_{C_{2,l-2 N_\Gamma}}-\prodint{\xi^{(1)},\phi_{2,l-2 N_\Gamma}}\mathcal L^{(1)}}&\overline{ \mathcal J^{{L_{\mathfrak C}^{(2)}}}_{\phi_{2,l-2 N_\Gamma}}}\\ \vdots &\vdots\\  \overline{
	\mathcal J^{ \overline{L_\Gamma^{(1)}}}_{	C_{2,l+2 N_{\mathfrak C}-1}}-\prodint{\xi^{(1)},\phi_{2,l+2 N_{\mathfrak C}-1}}\mathcal L^{(1)}}&\overline{ \mathcal J^{{L_{\mathfrak C}^{(2)}}}_{\phi_{2,l+2 N_{\mathfrak C}-1}}}
\end{bmatrix},
\end{multlined}	
\end{align*}
\begin{align*}
&\begin{multlined}[t][0.9\textwidth]
\bigg[\overline{L_\Gamma^{(2)}( x_1)} \mathcal  J^{\overline{L_\Gamma^{(2)}}}_{K_{C_1}^{[l]}}(\bar x_1)
+\mathcal  J^{\overline{L_\Gamma^{(2)}}}_{\delta \overline{L_\Gamma^{(2)}}}(\bar x_1)-\overline{ L_\Gamma^{(2)}( x_1)} \prodint{(\xi^{(2)})_z,K^{[l]}(\bar x_1,z)}\mathcal L^{(2)},\overline{ L_\Gamma^{(2)}( x_1)} \mathcal J^{L_{\mathfrak C}^{(1)}}_{K^{[l]}}(\bar x_1)\bigg]
\\=-\Big[\overline{\tilde \phi^{(2,1)}_{2,l- 2N_{\mathfrak C}}(x_1)}(\tilde  H^{(2,1)}_{l- 2N_{\mathfrak C}})^{-1},\dots, \overline{\tilde \phi^{(2,1)}_{2,l+ 2 N_\Gamma-1}(x_1)}(\tilde  H^{(2,1)}_{l+ 2 N_\Gamma-1})^{-1}\Big]
\\
\times\begin{bmatrix}
0_{2 N_{\mathfrak C}\times 2 N_\Gamma}&\mathfrak  C^{(2,1)}_{1,l}\\
-\Gamma^{(2,1)}_{1,l}&0_{2 N_\Gamma\times 2 N_{\mathfrak C}}
\end{bmatrix}\begin{bmatrix}
\mathcal J^{ \overline{L_\Gamma^{(2)}}}_{ 	C_{1,l-2 N_\Gamma}}-\prodint{\xi^{(2)},\phi_{1,l-2 N_\Gamma}}\mathcal L^{(2)}&\mathcal J^{L_{\mathfrak C}^{(1)}}_{\phi_{1,l-2 N_\Gamma}}\\\vdots&\vdots\\\mathcal J^{ \overline{L_\Gamma^{(2)}}}_{ C_{1,l+2 N_{\mathfrak C}-1}(x_2)}-\prodint{\xi^{(2)},\phi_{1,l+2 N_{\mathfrak C}-1}}\mathcal L^{(2)}&\mathcal J^{L_{\mathfrak C}^{(1)}}_{\phi_{1,l+2 N_{\mathfrak C}-1}}
\end{bmatrix}.
\end{multlined}
\end{align*}
Hence, we conclude
	\begin{align*}
&\begin{multlined}[t][0.9\textwidth]
{\Big[\tilde \phi^{(1,2)}_{1,l- 2N_{\mathfrak C}}(x_2)({\tilde H}_{l- 2N_{\mathfrak C}}^{(1,2)})^{-1},\dots,
	\tilde \phi^{(1,2)}_{1,l+ 2 N_\Gamma-1}(x_2)({\tilde H}_{l+ 2 N_\Gamma-1}^{(1,2)})^{-1}\Big]}
\begin{bmatrix}
0_{2 N_{\mathfrak C}\times 2 N_\Gamma}&
\overline{\mathfrak  C^{(1,2)}_{1,l}}\\
-\overline{\Gamma^{(1,2)}_{2,l}}&0_{2 N_\Gamma\times 2 N_{\mathfrak C}}
\end{bmatrix}\\=-\bigg[{L_\Gamma^{(1)}}(x_2)\mathcal J^{ \overline{L_\Gamma^{(1)}}}_{K_{C_2}^{[l]}}(x_2)+
\mathcal J^{ \overline{L_\Gamma^{(1)}}}_{\delta {L_\Gamma^{(1)}}}(x_2)-{L_\Gamma^{(1)}(x_2)} \prodint{\overline{(\xi^{(1)})_z},{K^{[l]}(\bar z,x_2)}}
\overline{\mathcal L^{(1)}},{L_\Gamma^{(1)}(x_2)} \mathcal J^{\overline{L_{\mathfrak C}^{(2)}}}_{K^{[l]}}(x_2)\bigg]
\\\times
\begin{bmatrix}
\overline{ \mathcal J^{ \overline{L_\Gamma^{(1)}}}_{C_{2,l-2 N_\Gamma}}-\prodint{\xi^{(1)},\phi_{2,l-2 N_\Gamma}}\mathcal L^{(1)}}&\overline{ \mathcal J^{{L_{\mathfrak C}^{(2)}}}_{\phi_{2,l-2 N_\Gamma}}}\\ \vdots &\vdots\\  \overline{
	\mathcal J^{ \overline{L_\Gamma^{(1)}}}_{	C_{2,l+2 N_{\mathfrak C}-1}}-\prodint{\xi^{(1)},\phi_{2,l+2 N_{\mathfrak C}-1}}\mathcal L^{(1)}}&\overline{ \mathcal J^{{L_{\mathfrak C}^{(2)}}}_{\phi_{2,l+2 N_{\mathfrak C}-1}}}
\end{bmatrix}^{-1},
\end{multlined}	
\end{align*}
\begin{align*}
&\begin{multlined}[t][0.9\textwidth]
\Big[\overline{\tilde \phi^{(2,1)}_{2,l- 2N_{\mathfrak C}}(x_1)}(\tilde  H^{(2,1)}_{l- 2N_{\mathfrak C}})^{-1},\dots, \overline{\tilde \phi^{(2,1)}_{2,l+ 2 N_\Gamma-1}(x_1)}(\tilde  H^{(2,1)}_{l+ 2 N_\Gamma-1})^{-1}\Big]
\begin{bmatrix}
0_{2 N_{\mathfrak C}\times 2 N_\Gamma}&\mathfrak  C^{(2,1)}_{1,l}\\
-\Gamma^{(2,1)}_{1,l}&0_{2 N_\Gamma\times 2 N_{\mathfrak C}}
\end{bmatrix}\\=-
\bigg[\overline{L_\Gamma^{(2)}( x_1)} \mathcal  J^{\overline{L_\Gamma^{(2)}}}_{K_{C_1}^{[l]}}(\bar x_1)
+\mathcal  J^{\overline{L_\Gamma^{(2)}}}_{\delta \overline{L_\Gamma^{(2)}}}(\bar x_1)-\overline{ L_\Gamma^{(2)}( x_1)} \prodint{(\xi^{(2)})_z,K^{[l]}(\bar x_1,z)}\mathcal L^{(2)},\overline{ L_\Gamma^{(2)}( x_1)} \mathcal J^{L_{\mathfrak C}^{(1)}}_{K^{[l]}}(\bar x_1)\bigg]
\\
\times\begin{bmatrix}
\mathcal J^{ \overline{L_\Gamma^{(2)}}}_{ 	C_{1,l-2 N_\Gamma}}-\prodint{\xi^{(2)},\phi_{1,l-2 N_\Gamma}}\mathcal L^{(2)}&\mathcal J^{L_{\mathfrak C}^{(1)}}_{\phi_{1,l-2 N_\Gamma}}\\\vdots&\vdots\\\mathcal J^{ \overline{L_\Gamma^{(2)}}}_{ C_{1,l+2 N_{\mathfrak C}-1}(x_2)}-\prodint{\xi^{(2)},\phi_{1,l+2 N_{\mathfrak C}-1}}\mathcal L^{(2)}&\mathcal J^{L_{\mathfrak C}^{(1)}}_{\phi_{1,l+2 N_{\mathfrak C}-1}}
\end{bmatrix}^{-1}.
\end{multlined}
\end{align*}
Recalling \eqref{Omegalambda} we get
	\begin{align*}
&\begin{multlined}[t][0.9\textwidth]
\tilde \phi^{(1,2)}_{1,l}(x_2)
=\bigg[{L_\Gamma^{(1)}}(x_2)\mathcal J^{ \overline{L_\Gamma^{(1)}}}_{K_{C_2}^{[l]}}(x_2)+
\mathcal J^{ \overline{L_\Gamma^{(1)}}}_{\delta {L_\Gamma^{(1)}}}(x_2)-{L_\Gamma^{(1)}(x_2)} \prodint{\overline{(\xi^{(1)})_z},{K^{[l]}(\bar z,x_2)}}
\overline{\mathcal L^{(1)}},{L_\Gamma^{(1)}(x_2)} \mathcal J^{\overline{L_{\mathfrak C}^{(2)}}}_{K^{[l]}}(x_2)\bigg]
\\\times
\begin{bmatrix}
\overline{ \mathcal J^{ \overline{L_\Gamma^{(1)}}}_{C_{2,l-2 N_\Gamma}}-\prodint{\xi^{(1)},\phi_{2,l-2 N_\Gamma}}\mathcal L^{(1)}}&\overline{ \mathcal J^{{L_{\mathfrak C}^{(2)}}}_{\phi_{2,l-2 N_\Gamma}}}\\ \vdots &\vdots\\  \overline{
	\mathcal J^{ \overline{L_\Gamma^{(1)}}}_{	C_{2,l+2 N_{\mathfrak C}-1}}-\prodint{\xi^{(1)},\phi_{2,l+2 N_{\mathfrak C}-1}}\mathcal L^{(1)}}&\overline{ \mathcal J^{{L_{\mathfrak C}^{(2)}}}_{\phi_{2,l+2 N_{\mathfrak C}-1}}}
\end{bmatrix}^{-1}\begin{bmatrix}
	\frac{H_{l-2 N_\Gamma}}{L^{(1)}_{\Gamma,(-1)^ln}}\\0\\\vdots\\0
	\end{bmatrix},
\end{multlined}	
\end{align*}
\begin{align*}
\begin{multlined}[t][0.9\textwidth]
\overline{\tilde \phi^{(2,1)}_{1,l}(x_1)}=
\bigg[\overline{L_\Gamma^{(2)}( x_1)} \mathcal  J^{\overline{L_\Gamma^{(2)}}}_{K_{C_1}^{[l]}}(\bar x_1)
+\mathcal  J^{\overline{L_\Gamma^{(2)}}}_{\delta \overline{L_\Gamma^{(2)}}}(\bar x_1)-\overline{ L_\Gamma^{(2)}( x_1)} \prodint{(\xi^{(2)})_z,K^{[l]}(\bar x_1,z)}\mathcal L^{(2)},\overline{ L_\Gamma^{(2)}( x_1)} \mathcal J^{L_{\mathfrak C}^{(1)}}_{K^{[l]}}(\bar x_1)\bigg]
\\
\times\begin{bmatrix}
\mathcal J^{ \overline{L_\Gamma^{(2)}}}_{ 	C_{1,l-2 N_\Gamma}}-\prodint{\xi^{(2)},\phi_{1,l-2 N_\Gamma}}\mathcal L^{(2)}&\mathcal J^{L_{\mathfrak C}^{(1)}}_{\phi_{1,l-2 N_\Gamma}}\\\vdots&\vdots\\\mathcal J^{ \overline{L_\Gamma^{(2)}}}_{ C_{1,l+2 N_{\mathfrak C}-1}(x_2)}-\prodint{\xi^{(2)},\phi_{1,l+2 N_{\mathfrak C}-1}}\mathcal L^{(2)}&\mathcal J^{L_{\mathfrak C}^{(1)}}_{\phi_{1,l+2 N_{\mathfrak C}-1}}
\end{bmatrix}^{-1}\begin{bmatrix}
\frac{H_{l-2 N_\Gamma}}{\overline{L^{(2)}_{\Gamma,(-1)^ln}}}\\0\\\vdots\\0
\end{bmatrix}.
\end{multlined}
\end{align*}
Thus, have proven the following last quasideterminantal expressions
\begin{align*}
	\overline{\tilde \phi^{(2,1)}_{2,l}(z)}&=-\Theta_*\begin{bmatrix}
\mathcal J^{ \overline{L_\Gamma^{(2)}}}_{ 	C_{1,l-2 N_\Gamma}}-\prodint{\xi^{(2)},\phi_{1,l-2 N_\Gamma}}\mathcal L^{(2)}&\mathcal J^{L_{\mathfrak C}^{(1)}}_{\phi_{1,l-2 N_\Gamma}}& \frac{	H_{l-2 N_\Gamma}}{\overline{L^{(2)}_{\Gamma,(-1)^lN_\Gamma}}}\\
\mathcal J^{ \overline{L_\Gamma^{(2)}}}_{ 	C_{1,l-2 N_\Gamma+1}}-\prodint{\xi^{(2)},\phi_{1,l-2 N_{\Gamma}+1}}\mathcal L^{(2)}&\mathcal J^{L_{\mathfrak C}^{(1)}}_{\phi_{1,l-2 N_\Gamma+1}}&0\\\vdots&\vdots&\vdots\\\mathcal J^{ \overline{L_\Gamma^{(2)}}}_{ C_{1,l+2 N_{\mathfrak C}-1}}-\prodint{\xi^{(2)},\phi_{1,l+2 N_{\mathfrak C}-1}}\mathcal L^{(2)}&\mathcal J^{L_{\mathfrak C}^{(1)}}_{\phi_{1,l+2 N_{\mathfrak C}-1}}&0\\
\overline{L_\Gamma^{(2)}( z)} \mathcal  J^{\overline{L_\Gamma^{(2)}}}_{K_{C_1}^{[l]}}(\bar z)
+\mathcal  J^{\overline{L_\Gamma^{(2)}}}_{\delta \overline{L_\Gamma^{(2)}}}(\bar z)-\overline{ L_\Gamma^{(2)}( z)} \prodint{(\xi^{(2)})_w,K^{[l]}(\bar z,w)}\mathcal L^{(2)}&\overline{ L_\Gamma^{(2)}( z)} \mathcal J^{L_{\mathfrak C}^{(1)}}_{K^{[l]}}(\bar z)&0
\end{bmatrix},
\end{align*}
\begin{align*}
\tilde \phi^{(1,2)}_{1,l}(z)
	&=-\Theta_*
\begin{bmatrix}
\overline{ \mathcal J^{ \overline{L_\Gamma^{(1)}}}_{C_{2,l-2 N_\Gamma}}-\prodint{\xi^{(1)},\phi_{2,l-2 N_\Gamma}}\mathcal L^{(1)}}&\overline{ \mathcal J^{{L_{\mathfrak C}^{(2)}}}_{\phi_{2,l-2 N_\Gamma}}}&\frac{	H_{l-2 N_\Gamma}}{L^{(1)}_{\Gamma,(-1)^lN_\Gamma}}\\
\overline{ \mathcal J^{ \overline{L_\Gamma^{(1)}}}_{C_{2,l-2 N_\Gamma+1}}-\prodint{\xi^{(1)},\phi_{2,l-2 N_\Gamma+1}}\mathcal L^{(1)}}&\overline{ \mathcal J^{{L_{\mathfrak C}^{(2)}}}_{\phi_{2,l-2 N_\Gamma+1}}}&0\\
\vdots &\vdots&\vdots\\  \overline{
	\mathcal J^{ \overline{L_\Gamma^{(1)}}}_{	C_{2,l+2 N_{\mathfrak C}-1}}-\prodint{\xi^{(1)},\phi_{2,l+2 N_{\mathfrak C}-1}}\mathcal L^{(1)}}&\overline{ \mathcal J^{{L_{\mathfrak C}^{(2)}}}_{\phi_{2,l+2 N_{\mathfrak C}-1}}}&0\\
{L_\Gamma^{(1)}}(z)\mathcal J^{ \overline{L_\Gamma^{(1)}}}_{K_{C_2}^{[l]}}(z)+
\mathcal J^{ \overline{L_\Gamma^{(1)}}}_{\delta {L_\Gamma^{(1)}}}(z)-{L_\Gamma^{(1)}(z)} \prodint{\overline{(\xi^{(1)})_w},{K^{[l]}(\bar w,z)}}
\overline{\mathcal L^{(1)}} &{L_\Gamma^{(1)}(z)} \mathcal J^{\overline{L_{\mathfrak C}^{(2)}}}_{K^{[l]}}(z)&0
\end{bmatrix}.
\end{align*}

These formulas  can be expressed in terms of determinants
\begin{align*}	
	\overline{\tilde \phi^{(2,1)}_{2,l}(z)}&=-\frac{	H_{l-2 N_\Gamma}}{\overline{L^{(2)}_{\Gamma,(-1)^lN_\Gamma}}}\frac{1}{\tilde\tau^{(2,1)}_l}\begin{vmatrix}
\mathcal J^{ \overline{L_\Gamma^{(2)}}}_{ 	C_{1,l-2 N_\Gamma+1}}-\prodint{\xi^{(2)},\phi_{1,l-2 N_{\Gamma}+1}}\mathcal L^{(2)}&\mathcal J^{L_{\mathfrak C}^{(1)}}_{\phi_{1,l-2 N_\Gamma+1}}\\\vdots&\vdots\\\mathcal J^{ \overline{L_\Gamma^{(2)}}}_{ C_{1,l+2 N_{\mathfrak C}-1}}-\prodint{\xi^{(2)},\phi_{1,l+2 N_{\mathfrak C}-1}}\mathcal L^{(2)}&\mathcal J^{L_{\mathfrak C}^{(1)}}_{\phi_{1,l+2 N_{\mathfrak C}-1}}\\
\overline{L_\Gamma^{(2)}( z)} \mathcal  J^{\overline{L_\Gamma^{(2)}}}_{K_{C_1}^{[l]}}(\bar z)
+\mathcal  J^{\overline{L_\Gamma^{(2)}}}_{\delta \overline{L_\Gamma^{(2)}}}(\bar z)-\overline{ L_\Gamma^{(2)}( z)} \prodint{(\xi^{(2)})_w,K^{[l]}(\bar z,w)}\mathcal L^{(2)}&\overline{ L_\Gamma^{(2)}( z)} \mathcal J^{L_{\mathfrak C}^{(1)}}_{K^{[l]}}(\bar z)
\end{vmatrix},
\end{align*}
\begin{align*}	
\tilde \phi^{(1,2)}_{1,l}(z)
	&=-\frac{	H_{l-2 N_\Gamma}}{L^{(1)}_{\Gamma,(-1)^lN_\Gamma}}\frac{1}{\tilde\tau^{(1,2)}_l}
\begin{vmatrix}
\overline{ \mathcal J^{ \overline{L_\Gamma^{(1)}}}_{C_{2,l-2 N_\Gamma+1}}-\prodint{\xi^{(1)},\phi_{2,l-2 N_\Gamma+1}}\mathcal L^{(1)}}&\overline{ \mathcal J^{{L_{\mathfrak C}^{(2)}}}_{\phi_{2,l-2 N_\Gamma+1}}}\\
\vdots &\vdots\\  \overline{
	\mathcal J^{ \overline{L_\Gamma^{(1)}}}_{	C_{2,l+2 N_{\mathfrak C}-1}}-\prodint{\xi^{(1)},\phi_{2,l+2 N_{\mathfrak C}-1}}\mathcal L^{(1)}}&\overline{ \mathcal J^{{L_{\mathfrak C}^{(2)}}}_{\phi_{2,l+2 N_{\mathfrak C}-1}}}\\
{L_\Gamma^{(1)}}(z)\mathcal J^{ \overline{L_\Gamma^{(1)}}}_{K_{C_2}^{[l]}}(z)+
\mathcal J^{ \overline{L_\Gamma^{(1)}}}_{\delta {L_\Gamma^{(1)}}}(z)-{L_\Gamma^{(1)}(z)} \prodint{\overline{(\xi^{(1)})_w},{K^{[l]}(\bar w,z)}}
\overline{\mathcal L^{(1)}} &{L_\Gamma^{(1)}(z)} \mathcal J^{\overline{L_{\mathfrak C}^{(2)}}}_{K^{[l]}}(z)
\end{vmatrix}.
\end{align*}

But, we have the following linear expressions for the last rows of these matrices
\begin{multline*}
\bigg[\overline{L_\Gamma^{(2)}( z)} \mathcal  J^{\overline{L_\Gamma^{(2)}}}_{K_{C_1}^{[l]}}(\bar z)
+\mathcal  J^{\overline{L_\Gamma^{(2)}}}_{\delta \overline{L_\Gamma^{(2)}}}(\bar z)-\overline{ L_\Gamma^{(2)}( z)} \prodint{(\xi^{(2)})_w,K^{[l]}(\bar z,w)}\mathcal L^{(2)},\overline{ L_\Gamma^{(2)}( z)} \mathcal J^{L_{\mathfrak C}^{(1)}}_{K^{[l]}}(\bar z)\bigg]\\=
\begin{aligned}[t]
&\bigg[\overline{L_\Gamma^{(2)}( z)} \mathcal  J^{\overline{L_\Gamma^{(2)}}}_{K_{C_1}^{[l-2N_\Gamma+1]}}(\bar z)
+\mathcal  J^{\overline{L_\Gamma^{(2)}}}_{\delta \overline{L_\Gamma^{(2)}}}(\bar z)-\overline{ L_\Gamma^{(2)}( z)} \prodint{(\xi^{(2)})_w,K^{[l-2N_\Gamma+1]}(\bar z,w)}\mathcal L^{(2)},\overline{ L_\Gamma^{(2)}( z)} \mathcal J^{L_{\mathfrak C}^{(1)}}_{K^{[l-2N_\Gamma+1]}}(\bar z)\bigg]
\\&+
\overline{L_\Gamma^{(2)}( z)}\sum_{k=l-2N_\Gamma+1}^{l-1}\overline{\phi_{2,k}(z)}(H_k)^{-1}\bigg[ \mathcal  J^{\overline{L_\Gamma^{(2)}}}_{C_{1,k}}
-\prodint{\xi^{(2)},\phi_{1,k} }\mathcal L^{(2)},\mathcal J^{L_{\mathfrak C}^{(1)}}_{\phi_{1,k}}\bigg],
\end{aligned}
\end{multline*}
\begin{multline*}
\bigg[{L_\Gamma^{(1)}( z)} \mathcal  J^{\overline{L_\Gamma^{(1)}}}_{K_{C_2}^{[l]}}( z)
+\mathcal  J^{\overline{L_\Gamma^{(1)}}}_{\delta {L_\Gamma^{(1)}}}( z)-{ L_\Gamma^{(1)}( z)} \prodint{\overline{(\xi^{(1)})_w},K^{[l]}(\bar w,z)}\overline{\mathcal L^{(1)}},{ L_\Gamma^{(1)}( z)} \mathcal J^{\overline{L_{\mathfrak C}^{(2)}}}_{K^{[l]}}( z)\bigg]\\=
\begin{aligned}[t]
&\bigg[{L_\Gamma^{(1)}( z)} \mathcal  J^{\overline{L_\Gamma^{(1)}}}_{K_{C_2}^{[l-2N_\Gamma+1]}}( z)
+\mathcal  J^{\overline{L_\Gamma^{(1)}}}_{\delta {L_\Gamma^{(1)}}}( z)-{ L_\Gamma^{(1)}( z)} \prodint{\overline{(\xi^{(1)})_w},K^{[l-2N_\Gamma+1]}(\bar w,z)}\overline{\mathcal L^{(1)}},{ L_\Gamma^{(1)}( z)} \mathcal J^{\overline{L_{\mathfrak C}^{(2)}}}_{K^{[l-2N_\Gamma+1]}}( z)\bigg]
\\&+
{L_\Gamma^{(1)}( z)}\sum_{k=l-2N_\Gamma+1}^{l-1}{\phi_{1,k}(z)}(H_k)^{-1}\bigg[ \overline{\mathcal  J^{\overline{L_\Gamma^{(1)}}}_{C_{2,k}}
-\prodint{\xi^{(1)},\phi_{2,k} }\mathcal L^{(1)}},\overline{\mathcal J^{L_{\mathfrak C}^{(2)}}_{\phi_{2,k}}}\bigg],
\end{aligned}
\end{multline*}
in where we see that the last term in the RHS is a linear combination of the rows present in the matrix and consequently, can be disregarded in the computation of the determinant and   we find \eqref{Ger22} and \eqref{Ger11}.
\end{proof}

\end{document}